\documentclass[11pt, letterpaper, notitlepage, oneside]{article}

\usepackage{palatino,url}
\usepackage{graphicx}
\usepackage{amsmath}
\usepackage{amsfonts}
\usepackage{mathrsfs}
\usepackage{amsthm}
\usepackage{amssymb}
\usepackage[all]{xy}
\usepackage{verbatim}
\usepackage{fancyhdr}
\usepackage{setspace}
\usepackage{color}
\usepackage{xr} 

\usepackage[vcentermath]{youngtab}
\newcommand{\Y}[1]{{\tiny\yng(#1)}}

\newcommand{\w}[1][]{\omega_{#1}}
\newcommand{\s}{\sigma}

\newcommand{\y}{\lambda}

\renewcommand{\a}{\alpha}

\renewcommand{\t}{\tau}

\newcommand{\sMod}{\text{--Mod}}

\newcommand{\Ind}{\text{Ind}}

\newcommand{\Res}{\text{Res}}
\newcommand{\Hom}{\text{Hom}}

\newcommand{\Stab}{\text{Stab}}

\newcommand{\End}{\text{End}}

\newcommand{\Span}{\text{Span}}
\renewcommand{\dim}{\text{dim}}
\newcommand{\rank}{\text{rank}}

\newcommand{\FIW}{\text{FI}_{\mathcal{W}}}
\newcommand{\oFIW}{\text{FI}_{\overline{\mathcal{W}}}}
\newcommand{\FI}{\text{FI}}
\newcommand{\coker}{\text{coker}}

\newcommand{\SL}{\text{SL}}
\newcommand{\Sp}{\text{Sp}}
\newcommand{\SO}{\text{SO}}

\newcommand{\cC}{\mathcal{C}}

\newcommand{\cI}{\mathcal{I}}

\newcommand{\bn}{{\bf n}}
\newcommand{\bm}{ {\bf m}}

\newcommand{\ba}{ {\bf a}}

\newcommand{\bX}{ {\bf X}}

\newcommand{\W}{\mathcal{W}}
\newcommand{\oW}{\overline{\mathcal{W}}}
\newcommand{\sh}{\sharp}

\newcommand{\C}{\mathbb{C}}

\newcommand{\Q}{\mathbb{Q}}

\newcommand{\Z}{\mathbb{Z}}

\newcommand{\G}{{\bf G}}
\newcommand{\B}{{\bf B}}

\newtheorem{thm}{Theorem}[section]
\newtheorem{prop}[thm]{Proposition}
\newtheorem{lem}[thm]{Lemma}
\newtheorem{cor}[thm]{Corollary}

\theoremstyle{definition}
\newtheorem{defn}[thm]{Definition}
\newtheorem{rem}[thm]{Remark}
\newtheorem{example}[thm]{Example}
\newtheorem{problem}[thm]{Problem}

\theoremstyle{plain}

\date{Dec 2014}
\title{$\FIW$--modules and stability criteria for representations of classical Weyl groups }
\author{Jennifer C. H. Wilson}

\begin{document}
\maketitle

\begin{abstract}

 In this paper we develop machinery for studying sequences of representations of any of the three families of classical Weyl groups, extending work of Church,  Ellenberg, Farb, and Nagpal \cite{CEF}, \cite{CEFN} on the symmetric groups $S_n$ to the signed permutation groups $B_n$ and the even-signed permutation groups $D_n$. For each family $\W_n$, we present an algebraic framework where a sequence $V_n$ of $\W_n$--representations is encoded into a single object we call an \emph{$\FIW$--module}. We prove that if an $\FIW$--module $V$ satisfies a simple \emph{finite generation} condition then the structure of the sequence is highly constrained. One consequence is that the sequence is \emph{uniformly representation stable} in the sense of Church--Farb, that is, the pattern of irreducible representations in the decomposition of each $V_n$ eventually stabilizes in a precise sense. Using the theory developed here we obtain new results about the cohomology of generalized flag varieties associated to the classical 
Weyl groups, and more generally the $r$-diagonal coinvariant algebras.
 
 We analyze the algebraic structure of the category of $\FIW$--modules, and introduce restriction and induction operations that enable us to study interactions between the three families of groups.  We use this theory to prove analogues of Murnaghan's 1938 stability theorem for Kronecker coefficients for the families $B_n$ and $D_n$. The theory of $\FIW$--modules gives a conceptual framework for stability results such as these. 
 
\end{abstract}

\setcounter{tocdepth}{2}
{\small \tableofcontents }

\section{Introduction}

Let $\W_n$ denote any of the one-parameter families of Weyl groups: the symmetric groups $S_n$, the hyperoctahedral groups (signed permutation groups) $B_n$, or the even-signed permutation groups $D_n$. In this paper we develop theory to study sequences $\{V_n\}$ of $\W_n$--representations. These Weyl groups' connections to Lie theory and realizations as finite reflection groups make such sequences prevalent in a broad range of mathematical subject areas. Our work builds on the theory of FI--modules developed by Church, Ellenberg, Farb, and Nagpal to study sequences of $S_n$--representations \cite{CEF}, \cite{CEFN}.

We prove that if a sequence of $\W_n$--representations has the structure of what we call a \emph{finitely generated $\FIW$--module} (Section \ref{SectionDefnFIWModules}), there are strong constraints on the growth of the representations $V_n$ and the pattern of irreducible $\W_n$--representations in the decomposition of $V_n$. 

To establish this finitely generated $\FIW$--module structure, it is enough to verify certain elementary compatibility and finiteness conditions on $\{V_n \}$. These conditions are often easily checked, and hold for numerous examples of sequences in geometry, algebraic topology, algebra, and combinatorics. Our work implies any such sequence of $\W_n$--representations over characteristic zero is \emph{uniformly representation stable} (defined in Section \ref{BackgroundRepStability}). In particular, in the notation of Section  \ref{BackgroundRepStability}, the multiplicity of each irreducible $\W_n$--representation $V(\y)_n$ in $V_n$ is eventually independent of $n$.

\newtheorem*{FinGenIffRepStabIntro}{Theorem \ref{FinGenIffRepStable}}
\begin{FinGenIffRepStabIntro} {\bf (Finite generation $\Longleftrightarrow$ Uniform representation stability).}
 Let $V$ be an $\FIW$--module over a field of characteristic zero. Then $\{V_n, (I_n)_* \}$ is uniformly representation stable  if and only if $V$ is finitely generated. 
\end{FinGenIffRepStabIntro}
 See Section \ref{SectionFinGenRepStability} for bounds on the stable ranges.

In a sequel \cite{FIW2} we will show that if a sequence of $\W_n$--representations over characteristic zero has the structure of a finitely generated $\FIW$--module, the characters have, for $n$ large, a stable and very special form: they are given by a \emph{character polynomial}, a polynomial in the signed cycle counting functions, which is independent of $n$. We will show moreover that, given a finitely generated $\FIW$--module $V$ over any field, the dimensions of the representations $V_n$ are \emph{eventually polynomial} in $n$.

In Section \ref{SectionCoinvariantAlgebras} we give applications to the diagonal coinvariant algebras $\cC^{(r)}(n)$ associated to $S_n$, $B_n$, and $D_n$. For $r=1$ these are the cohomology algebras of the associated generalized flag varieties. In type A, these results recover work of Church, Ellenberg, Farb, and Nagpal \cite[Theorems 3.4]{CEF}, \cite[Theorem 1.9]{CEFN}.

\newtheorem*{CoinvariantApplicationIntro}{ Corollaries \ref{CoinvariantAlgRepStable}, \ref{CoinvariantAlgCharPoly}, and \ref{CoinvariantAlgPolyDim}}
\begin{CoinvariantApplicationIntro}
Let $\cC^{(r)}(n)$ denote the $r$-diagonal coinvariant algebra associated to the Weyl groups $\W_n$ with coefficients in a field $k$.
 \begin{itemize}
  \item If $k$ has characteristic zero, each graded piece $\cC^{(r)}_J(n)$ is uniformly multiplicity stable. 
    \item If $k$ has characteristic zero, the characters of the graded piece $\cC^{(r)}_J(n)$ are eventually equal to a character polynomial of degree at most  $|J|$.
    \item Over an arbitrary field $k$, the dimensions of the graded pieces $\cC^{(r)}_J(n)$ are (for $n$ large) equal to a polynomial in $n$.
 \end{itemize}

\end{CoinvariantApplicationIntro}

The set of $\FIW$--modules has a rich algebraic structure. $\FIW$--modules in many ways resemble modules over a ring: there are natural notions of $\FIW$--module maps with quotients, kernels, and cokernels defined pointwise.  We prove in Section \ref{SectionNoetherian} that $\FIW$--modules are Noetherian:
\newtheorem*{NoetherianIntro}{Theorem \ref{Noetherian}}
\begin{NoetherianIntro} {\bf ($\FIW$--modules are Noetherian).} Let $k$ be a Noetherian ring. Then any sub--$\FIW$--module of a finitely generated $\FIW$--module over $k$ is itself finitely generated.
\end{NoetherianIntro}
There are direct sum and tensor product operations on $\FIW$--modules, which we analyze in Section \ref{SectionFIAlgebras}. In Sections \ref{Section:Restriction} and \ref{Section:Induction} we develop restriction and induction operations between sequences of the different families of Weyl groups, using the category-theoretic concept of a Kan extension. This algebraic structure provides a conceptual framework and an assortment of tools for analyzing sequences of $\W_n$--representations. 

Results of this $\FIW$--modules theory include an analogue of Murnaghan's 1938 stability theorem for Kronecker coefficients \cite{MurnaghanKronecker} for the hyperoctahedral group $B_n$ and even-signed permutation group $D_n$, which we prove in Section \ref{SectionFIAlgebras}. These are stated here using notation for rational irreducible $B_n$ and $D_n$--representations defined in Section \ref{BackgroundRepStability}.

\newtheorem*{MurnBC}{Theorem \ref{MurnaghanWn}}
\begin{MurnBC}{\bf (Murnaghan's stability theorem for $B_n$).} For any pair of double partitions $\y= (\y^+, \y^-)$ and $\mu = (\mu^+, \mu^-)$, there exist nonnegative integers $g^{\nu}_{\y, \mu}$, independent of $n$, such that for all $n$ sufficiently large:
\begin{equation}\tag{\ref{Eqn:MurnaghanBn}} V(\y)_n \otimes V(\mu)_n = \bigoplus_{\nu}  g^{\nu}_{\y, \mu} V(\nu)_n. \end{equation} 
The coefficients $g^{\nu}_{\y, \mu}$ are nonzero for only finitely many double partitions $\nu$.
\end{MurnBC}

Theorem \ref{MurnaghanWn} implies the following:

\newtheorem*{MurnD}{Corollary \ref{MurnaghanDn}}
\begin{MurnD}{\bf (Murnaghan's stability theorem for $D_n$).} With double partitions $\y= (\y^+, \y^-)$ and $\mu = (\mu^+, \mu^-)$ as above, for all $n$ sufficiently large the tensor product of the $D_n$--representations $V(\y)_n \otimes V(\mu)_n$ has a stable decomposition:
$$ V(\y)_n \otimes V(\mu)_n = \bigoplus_{\nu}  g^{\nu}_{\y, \mu} V(\nu)_n $$ 
where $g^{\nu}_{\y, \mu}$ are the structure constants of Equation (\ref{Eqn:MurnaghanBn}).
\end{MurnD}

In the context of $\FIW$--module theory, these stability results follow easily from a structural property of $\FIW$--modules: tensor products of finitely generated $\FIW$--modules are themselves finitely generated $\FIW$--modules. 

Many aspects of the theory of $\FIW$--modules parallels the work \cite{CEF} and \cite{CEFN}. We encounter several additional challenges, however, particularly in type D. Section \ref{SectionEarlierWork} summarizes the relationship to recent work and new phenomena in this paper. 

\subsection{ $\FIW$--modules and finite generation} \label{SectionDefnFIWModules}

We will now define our central concepts, $\FIW$--modules and finite generation.

\begin{defn}{\bf (The Category $\FIW$).} \label{DefnFIW} Let $\W_n$ denote the Weyl group in type A$_{n-1}$, B$_{n}$/C$_n$, or D$_{n}$, and accordingly let $\FIW$ denote the category $\FI_A$, $\FI_{BC}$, or $\FI_D$, as shown in the table below. \\

\resizebox{\textwidth}{!}{
\noindent {\footnotesize
\begin{tabular}{|p{1.2cm}|p{3.2cm}|p{7.6cm}|}
\hline Category & Objects &  Morphisms \\ \hline && \\ 
 $\FI_{BC}$ & $\bn = \{ \pm 1, \pm 2, \ldots, \pm n \}$& $\{$ injections  $f:\bm \to \bn \text{  $\; | \; f(-a) = -f(a) \; \; \forall$ $a \in \bm$} \}$ \\ 
 & ${\bf 0} = \varnothing $ & \qquad $\End(\bn) \cong B_n$ \\ && \\
$\FI_D$ &$\bn = \{ \pm 1, \pm 2, \ldots, \pm n \}$ \newline ${\bf 0} = \varnothing$  &  $\{$ injections $f:\bm \to \bn \; |\; f(-a) = -f(a)$ $\; \; \forall$ $a \in \bm$;  \newline  isomorphisms must reverse an even number of signs $\}$ \\ 
&& \qquad $\End(\bn) \cong D_n$ \\  && \\
$\FI_A$  & $\bn = \{ \pm 1, \pm 2, \ldots, \pm n \}$ \newline ${\bf 0} = \varnothing $  &  $\{$ injections $f:\bm \to \bn \; | \;  f(-a) = -f(a)$ $\; \; \forall$ $a \in \bm$;  \newline  $f$ preserves signs$\}$ \indent $ \newline \qquad \End(\bn) \cong S_n$ \\ && \\ \hline
\end{tabular}
} }
\quad \\

\noindent In each case, the objects of $\FIW$ are indexed by the natural numbers $\Z_{\geq 0}$; we will write these objects in boldface throughout the paper. The endomorphisms $\End(\bn)$ are isomorphic to the group $\W_n$, and the morphisms are generated by $\End(\bn)$ and the natural inclusions $I_{n}: \bn \hookrightarrow {\bf (n+1) }$. The category $\FI_A$ is equivalent to the category $\FI$ defined by Church--Ellenberg--Farb \cite{CEF} as the category of {\bf  F}inite sets and {\bf  I}njective maps. There are inclusions of categories $\FI_A \hookrightarrow \FI_D \hookrightarrow \FI_{BC}$. \end{defn}

\begin{defn} {\bf ($\FIW$--module).} \label{DefnFIWModule} Let $\FIW$ denote $\FI_A$, $\FI_{BC}$, or $\FI_D$, and accordingly let $\W_n$ denote $S_n$, $B_n$, or $D_n$. We define an \emph{$\FIW$--module} $V$ over a ring $k$ to be a (covariant) functor from $\FIW$ to the category of $k$--modules. We will assume $k$ is commutative and with unit. The image of an $\FIW$--module is a sequence of $\W_n$--representations $V_n := V(\bn)$ equipped with an array of maps $V_m \to V_n$ compatible with the $\W_n$--action. For $f \in \Hom_{\FIW}(\bm, \bn)$, we write $f_*$ (or simply $f$) to denote the linear map $V(f): V_m \to V_n$. 
\end{defn}
This definition of an $\FI_A$--module is equivalent to that of an $\FI$--module given by \cite{CEF}. A schematic of an $\FIW$--module is shown in Figure \ref{fig:FIWModule}. 

\begin{figure}[h!]
\begin{center}
\setlength\fboxsep{5pt}
\setlength\fboxrule{0.5pt}
\fbox{ \includegraphics[scale=3.5]{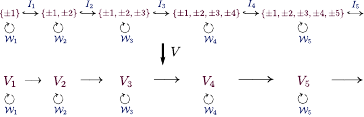}}
 \caption{ {\small An $\FIW$--module $V$} }
\label{fig:FIWModule}
\end{center}
\end{figure}

We similarly define a \emph{co--$\FIW$--module} over a ring $k$ as a functor from the dual category $\FIW^{op}$ to $k$--Modules.

\begin{defn} {\bf  (Finite generation, Degree of generation).} \label{DefnFinGen}  We say an $\FIW$--module $V$ is \emph{finitely generated} if there is a finite set of elements of $\coprod_{n=0}^{\infty} V_n$ that are not contained in any proper sub--$\FIW$--module. The images of these elements under the $\FIW$ morphisms span each $k[\W_n]$--module  $V_n$. We say $V$ is finitely generated in \emph{degree $\leq d$} if it has a finite generating set $\{ v_i \}$ with $v_i \in V_{m_i}$,  $m_i \leq d$ for each $i$. \end{defn} 
 
 \begin{example} \label{Example:PolyAlg} {\bf (Some finitely and infinitely generated $\FIW$--modules).} For a basic example to illustrate Definition \ref{DefnFinGen}, let $V_n := k[x_1, \ldots, x_n]$ be the polynomial ring on $n$ variables $x_i$ with the obvious inclusions $V_{n-1} \hookrightarrow V_n$. The group $\W_n$ acts on $V_n$ by permuting and (for $D_n$ or $B_n$) negating the variables. The $\FIW$--module formed by the spaces $V_n$ is infinitely generated, but for each integer $d \geq 0$ the subspaces of homogeneous degree-$d$ polynomials $k[x_1, \ldots, x_n]_{(d)}$ form a sub--$\FIW$--module finitely generated in degree $\leq d$. Figure \ref{fig:FinGenExample} shows a finite generating set for the $\FIW$--module of homogeneous degree-$2$ polynomials. \end{example}
 \begin{figure}[h!]
\begin{center}
\setlength\fboxsep{5pt}
\setlength\fboxrule{0.5pt}
\fbox{ \includegraphics[scale=4]{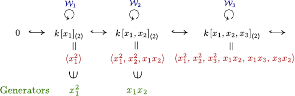}}
 \caption{ {\small The finitely generated $\FIW$--module $k[x_1, \ldots, x_n]_{(2)}$}  }
\label{fig:FinGenExample}
\end{center}
\end{figure}

The property of being finitely generated is easy to verify in many applications, but has strong implications for the structure of the underlying sequence of $\W_n$--representations. 

\subsection{Representation stability of finitely generated $\FIW$--modules} 

Prior to their work with Ellenberg on $\FI$--modules, Church and Farb defined and developed the theory of \emph{representation stability} for families of groups $G_n$ including $S_n$ and $B_n$ \cite{RepStability}. For a sequence $V_n$ of rational $G_n$--representations to be representation stable, the multiplicities of the irreducible constituents $V(\y)_n$ of $V_n$ must eventually be constant in $n$; crucial to this definition is the appropriate classification of irreducible $G_n$--representations $V(\y)_n$ as functions of $n$. We describe these definitions in more detail in Section \ref{BackgroundRepStability}, where we also introduce a definition of representation stability for sequences of $D_n$--representations. 

It is shown in \cite[Theorem 1.14]{CEF} that, for sequences of $S_n$--representations with the structure of an $\FI$--module, finite generation is equivalent to uniform representation stability. We prove this phenomenon holds more generally:

\newtheorem*{FinGenRepStab}{Theorems \ref{FinGenImpliesRepStability} and \ref{RepStabilityImpliesFinGen}}
\begin{FinGenRepStab} {\bf ($\FIW$--modules are uniformly representation stable iff they are finitely generated).}
Suppose that $k$ is a field of characteristic zero, and $\W_n$ is $S_n$, $D_n$, or $B_n$. Let $V$ be a finitely generated $\FI_{\W}$--module. Take $d$ to be an upper bound on the weight of $V$, $g$ an upper bound on its degree of generation, and $r$  an upper bound on its relation degree.  Then $\{V_n \}$ is uniformly representation stable with respect to the maps induced by the natural inclusions $I_n: \bn \to {\bf (n+1)}$, stabilizing once $n \geq \max(g,r)+d$; when $\W_n$ is $D_n$ and $d=0$ we need the additional condition that $n \geq g+1$. 

Suppose conversely that $V$ is an $\FIW$--module, and that $\{V_n, (I_n)_* \}$ is uniformly representation stable for $n \geq N$. Then $V$ is finitely generated in degree $\leq N$. 
\end{FinGenRepStab}

The classification of rational irreducible $B_n$ and $D_n$--representations are described in Sections \ref{RepTheoryWn} and \ref{RepTheoryDn}, and the precise definition of $V(\y)_n$ and criteria for representation stability are given in Section \ref{BackgroundRepStability}.

\subsection{Forthcoming results} 

\noindent {\bf Character polynomials in type B/C and D. \quad} Let $k$ be a field of characteristic zero. If $V$ is finitely generated $\FIW$--module over $k$, we have shown that the sequence $\{V_n\}$ is uniformly representation stable, and in a sequel \cite{FIW2} we will show that this has strong implications for the characters $\chi_n$ of $V_n$. Specifically, we show that the sequence of characters is, for $n$ large, equal to a \emph{character polynomial} which does not depend on $n$. This was proven for symmetric groups in \cite[Theorem 2.67]{CEF}, and in \cite[Section \ref{FIW2-SectionCharPolys}]{FIW2} we extend these results to the groups $D_n$ and $B_n$.

Conjugacy classes of the hyperoctahedral group are classified by \emph{signed cycle type}, as described in Section \ref{RepTheoryWn}. Define class functions $X_r$, $Y_r$ on $\coprod_{n=0}^{\infty} B_n$ such that 
\begin{align*}
 X_r (\w) & \text{ is the number of positive $r$--cycles in $\w$,} \\
Y_r (\w) & \text{ is the number of negative $r$--cycles in $\w$.}
\end{align*}
The functions $X_r, Y_r$ are algebraically independent, and the polynomial ring $$k[X_1, Y_1, X_2, Y_2, \ldots]$$ spans the class functions on $B_n$ for each $n \geq 0$. Elements of this ring are \emph{character polynomials}. They also restrict to class functions on the subgroups  $D_n \subset B_n$. We prove the following. 

\newtheorem*{CharacterPolynomials}{Theorem \ref{FIW2-WnPersinomial} \cite{FIW2}}
\begin{CharacterPolynomials}
{\bf (Characters of finitely generated $\FI_{\W}$--modules are eventually polynomial).}
Let $k$ be a field of characteristic zero. Suppose that $V$ is a finitely generated $\FI_{BC}$--module with weight $\leq d$ and stability degree $\leq s$, or,  alternatively, suppose that $V$ is a finitely generated $\FI_D$--module with weight $\leq d$ such that $\Ind_{D}^{BC}\;V$ has stability degree $\leq s$. In either case, there is a unique polynomial $$F_V \in k[X_1, Y_1, X_2, Y_2, \ldots]$$ such that the character of $\W_n$ on $V_n$ is given by $F_V$ for all $n \geq s +d$. The polynomial $F_V$ has degree $\leq d$, with deg$(X_i)=$deg$(Y_i)=i$.
\end{CharacterPolynomials}

\noindent Weight and stability degree are defined in Sections \ref{Section:Weight} and \ref{Section:CoinvariantsAndStabilityDeg}; these quantities are necessarily finite for $V$ finitely generated.  

{ \noindent \bf Eventually polynomial dimensions. \quad} Let $V$ be a finitely generated $\FIW$--module. If $V$ is defined over characteristic zero, then we can take the associated character polynomial $F_V$ and evaluate it at the identity elements to compute dim$(V_n) = F_V(n,0,0,0, \ldots).$ We conclude the following.

\newtheorem*{polygrowth}{Corollary \ref{FIW2-PolyGrowth} \cite{FIW2}}
\begin{polygrowth} {\bf (Polynomial growth of dimension).}
 Let $V$ be an $\FIW$--module over a field of characteristic zero, and suppose $V$ is finitely generated in degree $\leq d$. Then for large $n$, dim$(V_n)$ is equal to a polynomial in $n$ of degree at most $d$. Equality holds for $n$ in the stable range given in \cite[Theorem \ref{FIW2-WnPersinomial}]{FIW2}.
\end{polygrowth}

This ``eventually polynomial'' growth of dimension holds even over positive characteristic. The following theorem uses results of Church--Ellenberg--Farb--Nagpal, who prove the theorem for finitely generated $\FI_A$--modules \cite[Theorem 1.2]{CEFN}.

\newtheorem*{polygrowthcharp}{Theorem \ref{FIW2-PolyDimCharP} \cite{FIW2}}
\begin{polygrowthcharp}{\bf (Polynomial growth of dimension over arbitrary fields).} Let $k$ be any field, and let $V$ be a finitely generated $\FIW$--module over $k$. Then there exists an integer-valued polynomial $P(T) \in \Q[T]$ such that $$ \dim_k (V_n) = P(n) \qquad \text{for all $n$ sufficiently large.}$$
\end{polygrowthcharp}

\noindent {\bf $\FIW\sh$--modules. \quad} In \cite[Section \ref{FIW2-SectionFISharp}]{FIW2} we describe a certain class of $\FI_{BC}$--modules called \emph{$\FI_{BC}\sh$--modules}, analogues of the $\FI\sh$--modules (``FI sharp modules'') defined by Church--Ellenberg--Farb. An $\FI_{BC}\sh$--module is a sequence of $B_n$--representations that simultaneously admits a functor from $\FI_{BC}$ and a functor from the dual category $\FI_{BC}^{\mathrm{op}}$ in some compatible sense; see \cite[ Definition \ref{FIW2-DefnFISharp}]{FIW2}.

A finitely generated $\FI_{BC}\sh$--module structure places even stronger constraints on the structure of a sequence of $B_n$--representations. For example, we show in  \cite[Section \ref{FIW2-CharacterPolynomialsFISharp}]{FIW2} that a $\FI_{BC}\sh$--module finitely generated in degree $\leq d$ has characters equal to a unique character polynomial of degree at most $d$ for \emph{all} values of $n$, and dimensions given by a polynomial in $n$ of degree at most $d$ for all $n$. 

We prove in \cite[Theorem \ref{FIW2-ClassifyingFISharp}]{FIW2} that $\FI_{BC}\sh$--modules are direct sums of sequences of the form
$$ \left\{ \bigoplus_{m=0} \Ind_{B_{m}\times B_{n-m}}^{B_n} U_m \boxtimes k \right\}_n .  $$ \noindent Here, $k$ denotes the trivial $B_{n-m}$--representation, and $U_m$ is a $B_m$--representation, possibly $0$. The external tensor product $(U_m \boxtimes k)$ is the $k$--module $(U_m \otimes_k k)$ as a $(B_{m}\times B_{n-m})$--representation. The result \cite[Theorem \ref{FIW2-ClassifyingFISharp}]{FIW2} extends \cite[Theorem 2.24]{CEF}, which is the analogous statement in type A. 

\subsection{An application: diagonal coinvariant algebras.}

The theory of $\FIW$--modules developed in this paper gives new, concrete results about a variety of known objects in geometry and combinatorics. In Section \ref{SectionCoinvariantAlgebras} we give applications to the diagonal coinvariant algebras associated to the reflection groups $\W_n$.

Let $k$ be a field, and consider the canonical action of $\W_n$ on $$V_n := k^n \cong \Span_k \langle x_1, \ldots, x_n \rangle .$$ The group $S_n$ acts by permutation matrices, $B_n$ by signed permutation matrices, and $D_n$ by signed permutation matrices with an even number of entries equal to $-1$. 

There is an induced diagonal action of $\W_n$ on $V_n^{\oplus r}$, and so an induced action on the symmetric algebra Sym$(V_n^{\oplus r})$, isomorphic to the polynomial algebra $$k[x_1^{(1)}, \ldots x_n^{(1)}, \ldots, x_1^{(r)}, \ldots, x_n^{(r)}].$$ The \emph{$r$-diagonal coinvariant algebra} $\cC^{(r)}(n)$ is the quotient of this algebra by the ideal $\cI_n$ of constant-term-free $\W_n$--invariant polynomials. The algebra  $\cC^{(r)}(n)$ has a natural multigrading by $r$-tuples $J=(j_1, \ldots, j_r)$, where $j_{\ell}$ specifies the total degree of a monomial in the variables $x_1^{\ell}, \ldots, x_n^{\ell}$. 

The coinvariant algebras $\cC^{(1)}(n)$ were studied classically for their connections to representation theory of Lie groups. The $r$-diagonal coinvariant algebras have been studied since the 1990s, with major contributors including Garsia, Haiman, Hagland, Gordon, Bergeron, and Biagioli; see Section \ref{SectionCoinvariantAlgebras} for more history. Haiman \cite{HaimanCombinatorics} and Bergeron \cite{BergeronBook} offer in-depth background on coinvariant algebras and their many connections to other areas of algebraic combinatorics. 

In Section \ref{SectionCoinvariantAlgebras} we prove that each multigraded component $\cC^{(r)}_J(n)$ of $\cC^{(r)}(n)$ is a finitely generated co--$\FIW$--module. Understanding the characters of the multigraded components of $\cC^{(r)}(n)$ is a well-known open problem; little is known except for very small values of $r$ and $n$. The following result, inspired by the work of \cite{CEF} and \cite{CEFN} on diagonal coinvariant algebras in type A, reveals underlying structure and patterns in these sequences of representations. 

\newtheorem*{CoinvariantAlgebrasStable}{Theorem \ref{CoinvariantAlgebrasCoFI}}
\begin{CoinvariantAlgebrasStable}
Let $k$ be a field, and let $V_n \cong k^n$ be the canonical representation of $\W_n$ by (signed) permutation matrices. Given $r \in \Z_{>0}$, the sequence of coinvariant algebras $\cC^{(r)} := k[V_{\bullet}^{\bigoplus r}]/\cI$ is a graded co--$\FIW$--algebra of finite type. When $k$ has characteristic zero, the weight of the multigraded component $\cC^{(r)}_J$ is $\leq |J|$. \end{CoinvariantAlgebrasStable}

\newtheorem*{CoinvariantAlgebrasRepStable}{Corollary \ref{CoinvariantAlgRepStable} }
\begin{CoinvariantAlgebrasRepStable}
 Let $k$ be a field of characteristic zero. For $n$ sufficiently large (depending on the $r$-tuple $J$), the sequence $\cC^{(r)}_J(n)$ is uniformly multiplicity stable. 
\end{CoinvariantAlgebrasRepStable}

\newtheorem*{CoinvariantAlgebrasCharPoly}{Corollary \ref{CoinvariantAlgCharPoly}}
\begin{CoinvariantAlgebrasCharPoly}
Let $k$ be a field of characteristic zero. For $n$ sufficiently large (depending on the $r$-tuple $J$), the characters of $\cC^{(r)}_J(n)$ are given by a character polynomial $F_J$ of degree $\leq |J|$. In particular the dimension of $\cC^{(r)}_J(n)$ is given by the degree $|J|$ polynomial $\dim_k \cC^{(r)}_J(n) = F_J(n, 0, 0, 0 \ldots )$ for all $n$ in the stable range. 
\end{CoinvariantAlgebrasCharPoly}

\newtheorem*{coinvariantalgpolydim}{Corollary \ref{CoinvariantAlgPolyDim}}
\begin{coinvariantalgpolydim}
Let $k$ be an arbitrary field. Then for each $r$-tuple $J$, there exists a polynomial $P_J \in \Q[T]$ (depending on $k$) so that
$\dim_k \; \cC^{(r)}_J(n) = P_J(n)$ for all $n$ sufficiently large (depending on $k$ and $J$).
\end{coinvariantalgpolydim}

\noindent Theorem \ref{CoinvariantAlgebrasCoFI} and its corollaries were proven in type A over characteristic zero by Church--Ellenberg--Farb \cite[Theorem 3.4]{CEF}. In later work with Nagpal these authors extend their work to fields of arbitrary characteristic \cite[Proposition 4.2]{CEFN}, and in particular they prove Corollary \ref{CoinvariantAlgPolyDim} in type A \cite[Theorem 1.9]{CEFN}.

In the special case $r=1$, the algebras $\cC^{(1)}(n)$ are isomorphic to the cohomology rings of the generalized flag varieties associated to the Lie groups in type $\W$; see Section \ref{SectionCoinvariantAlgebras} for details. A corollary of Theorem \ref{CoinvariantAlgebrasCoFI} is representation stability and
the existence of character polynomials for these cohomology groups. 

In Section \ref{SectionCoinvariantAlgebras} we state the character polynomials in type B/C for $|J|\leq 3$; in general, computing these character polynomials is an open problem. 

\subsection{Remarks on the general theory}

We briefly highlight some key tools and results of the theory of $\FIW$--modules.\\

\noindent {\bf  The structure of finitely generated $\FIW$--modules. \quad} A crucial fact about finitely generated $\FIW$--modules is that they can be realized as quotients of sequences of the form $$ \left\{ \bigoplus_{m=0}^g \Big( \Ind_{\W_{n-m}}^{\W_n} k \Big)^{\oplus \ell_m} \right\}_n ,$$ where $k$ denotes the trivial $\W_{n-m}$--representation, and the multiplicities $\ell_m$ are finite (see Proposition \ref{FinGen}). Over fields of characteristic zero, the combinatorics of these induced representations is governed by the branching rules for each family $\W_n$ -- rules that are well understood for $S_n$ and $B_n$, though more complex for $D_n$ (see, for example, Geck-Pfeiffer \cite{GeckPfeiffer}).   

 We prove in Theorem \ref{Noetherian} that sub--$\FIW$--modules of finitely generated $\FIW$--modules are themselves finitely generated. This Noetherian property was proven for $\FI_A$--modules by Church--Ellenberg--Farb \cite[Theorem 2.6]{CEF} over Noetherian rings containing $\Q$, and later proven by Church--Ellenberg--Farb--Nagpal \cite[Theorem 1.1]{CEFN} over arbitrary Noetherian rings; our proof uses their results. These properties of finite generation are used extensively throughout this paper. \\

\noindent {\bf  Restriction and Induction of $\FIW$--Modules \quad} Given the inclusions of categories $\FI_A \hookrightarrow \FI_D \hookrightarrow \FI_{BC}$, there is a natural restriction operation of $\FI_{BC}$ and $\FI_D$--modules down to $\FI_D$ or $\FI_A$--modules, and we show in Proposition \ref{RestrictionPreservesFinGen} that the restriction of functors between these categories preserves the property of finite generation.

Given an inclusion of categories $\FIW \subset \oFIW$ and an $\FIW$--module $V$, the sequence of $\oW_n$--representations $\{\Ind_{\W_n}^{\oW_n} V_n\}_n$ does not in general have the structure of an $\oFIW$--module; see Remark \ref{NaiveInduction}. In Section \ref{Section:Induction} we show that there nonetheless exist induction operations on $\FIW$--modules using the theory of \emph{Kan extensions}; this insight owes to Peter May. These operations place the theory of $\FIW$--modules for these three families of groups in a unified setting, and moreover appear to be of theoretical interest in their own right. 

 \subsection{Relationship to earlier work} \label{SectionEarlierWork}
 
 \subsubsection{Recent work}

{\noindent \bf Representation stability. \quad} In 2010, Church--Farb \cite{RepStability} introduced the concept of \emph{representation stability} for sequences of rational representations of several families of groups: $S_n$, $B_n$, and the linear groups SL$_n(\Q)$, GL$_n(\Q)$, and Sp$_{2n}(\Q)$. For each family they formulated stability criteria in terms of the pattern of irreducible subrepresentations, patterns which they show appear in ubiquitous examples throughout mathematics. They gave a host of applications to classical representation theory, the cohomology of groups arising in geometric group theory, Lie algebras and their homology,  the (equivariant) cohomology of flag and Schubert varieties, and algebraic combinatorics. \\

{\noindent \bf FI--modules. \quad}  Two years later Church--Ellenberg--Farb \cite{CEF} significantly refined the theory for sequences of $S_n$--representations by introducing FI--modules. This new work accomplished several things: They established criteria for representation stability that are simple and easily verified -- a finitely generated FI--modules structure. They strengthened their results with the observation that the characters of a representation stable sequence have an associated character polynomial, and gave a number of consequences including polynomial growth of dimension. They gave a framework for studying sequences of $S_n$--representations over arbitrary coefficient rings, which does not depend on the combinatorial particulars of the classification of irreducible rational representations. The category FI, and the concept of finite generation, are natural and elementary constructs. Their theory provides a structured, unified context and a vocabulary to describe patterns and stability 
phenomenon that perhaps could not be captured otherwise. 

Using the theory of FI--modules, Church--Ellenberg--Farb proved new results about a number of fundamental sequences $V_n$ of $S_n$--representations. These include the cohomology of the $n$-point configuration space of a manifold, the cohomology of the moduli space of $n$-puncture surfaces, certain subalgebras of the cohomology of the genus $n$ Torelli group, and the diagonal $S_n$-coinvariant algebras on $r$ sets of $n$ variables. 

Jimenez Rolland \cite{JimenezRollandRepStability, JimenezRollandMCGasFI} studied additional applications of this theory to moduli spaces, pure mapping class groups, and diffeomorphism groups of certain punctured manifolds. She introduced FI[$G$]-modules associated to a group $G$: functors from FI to the category of $k[G]$--modules. She obtained classical homological stability results for certain wreath products. 
\\

{\noindent \bf Central stability. \quad} Putman \cite{Putman2012Stability} independently developed a theory that extends representation stability to positive characteristic. He established stability results for level $q$ congruence subgroups of $GL_n(R)$ for a large class of rings $R$ with ideals $q$. His main definition, \emph{central stability}, is closely related to the notion of a finitely generated FI--module. Putman proved that central stability implies representation stability and polynomial dimension growth. He integrated his theory of central stability with the classical homological stability machinery developed by Quillen. This representation-theoretic homological stability apparatus applies to a variety of geometric and algebraic applications over numerous coefficient systems. \\ 

{\noindent \bf FI--modules over Noetherian rings. \quad} Shortly after the appearance of Putman's work, the work on FI--modules \cite{CEF} were strengthened further by Church--Ellenberg--Farb--Nagpal \cite{CEFN}. These authors extended several results to broader classes of coefficients: they proved polynomial growth of dimension over fields of positive characteristic, and the Noetherian property over arbitrary Noetherian rings. They generalized their results for several of the above applications to coefficients in the integers or positive characteristic. \\ 

{\noindent \bf Twisted commutative algebras. \quad} In 2010, Snowden \cite{SnowdenSyzygies} independently proved, using different language, several fundamental properties of FI--modules. His work centres on modules over a class of objects called \emph{twisted commutative algebras}; FI--modules are an example. His results include the Noetherian and polynomial growth properties for finitely generated complex FI--modules, results which he used to study syzygies of Segre embeddings. See Sam--Snowden \cite{SamSnowdenIntroductionTCA} for an accessible introduction to the theory of twisted commutative algebras. 

Following the work of Church--Ellenberg--Farb, Sam--Snowden \cite{SamSnowdenGLEquivariant} performed a deeper analysis of the category of FI--modules over a field of characteristic zero and proved a number of algebraic and homological finiteness properties. We would be interested to better understand how the work of Snowden and Sam--Snowden relates to the theory of $\FIW$--modules developed here. \\

{\noindent \bf Gamma modules. \quad} The results of Church--Ellenberg--Farb on $\FI\sh$--modules have strong parallels to work of Pirashvili on $\Gamma$--modules \cite{PirashviliGamma, PirashviliHodge}, functors from the category of finite based sets and (not necessarily injective) based maps. Pirashvili's results include a characterization of integer $\Gamma$--modules \cite[Theorem 3.1]{PirashviliGamma} analogous to the classification of $\FI\sh$--modules \cite[Theorem 2.24]{CEF} and our results for $\FIW\sh$--modules \cite[Theorem \ref{FIW2-ClassifyingFISharp}]{FIW2}.
 
 \subsubsection{New obstacles and new phenomena}
 Much of the theory of $\FIW$--modules parallels the work of Church--Ellenberg--Farb \cite{CEF}, and frequently their methods of proof adapt to our more general context. Some additional hurdles and some new phenomena do emerge, however, for the Weyl groups $B_n$ and $D_n$. These include: \\
 
\noindent   {\bf Restriction and induction. \quad } The restriction and induction adjoint operations between the three categories $\FI_A$, $\FI_D$, and $\FI_{BC}$ give $\FIW$--modules a new level of structure. In Sections \ref{Section:Restriction} and \ref{Section:Induction} we define and study these operations from a category-theoretic perspective.  \\
 
 \noindent  {\bf Branching rules in type D. \quad } The combinatorics of the branching rules for the hyperoctahedral groups, like the symmetric groups, are well understood. With these formulas, many of the methods of proof used by Church--Ellenberg--Farb \cite{CEF} for $\FI_A$--modules adapt beautifully to $\FI_{BC}$--modules, including the proof that finite generation is equivalent to uniform representation stability. In contrast, the branching rules for the groups $D_n$ are more subtle, and it is not clear that the methods in \cite{CEF} adapt so readily to type D. We proceed instead by analyzing the restriction and inductions operations between $\FI_D$ and $\FI_{BC}$. To recover the main results in type D, we relate a finitely generated $\FI_D$--module $V$ to the $\FI_D$--module $\Res^{BC}_D \Ind^{BC}_D V,$  defined in Section \ref{Section:Induction}, and appeal to our results for finitely generated $\FI_{BC}$--modules.\\
 
\noindent {\bf Representation stability in type D. \quad } It was initially unclear how \emph{representation stability} ought to be defined for representations of the even-signed permutation groups $D_n$. The classification of irreducible $D_n$--representations (Section \ref{RepTheoryDn}), which involves \emph{unordered} pairs of partitions and 'split' representations in even degree, did not suggest any deterministic growth rules of the form defined by Church--Farb \cite{RepStability} for $S_n$ and $B_n$ (Section \ref{BackgroundRepStability}). More to the point,  it was not clear that we could expect any specific constraints on the patterns of irreducible representations in a class of sequences as broad and commonly occurring as the finitely generated $\FI_D$--modules. 

The definition of representation stability in type D was ultimately written late in the course of this project, after the discovery of an unanticipatedly strong result: If $V$ is a finitely generated $\FI_D$--module, then, for $n$ sufficiently 
large,  $V_n$ is the restriction of a $B_n$--representation.

\subsection*{Acknowledgments}

I would like to thank Tom Church, Jordan Ellenberg, and Benson Farb for laying the mathematical groundwork for this project, and for extensive discussion about their and this work.  I am also grateful to Rita Jimenez Rolland for many useful conversations. I would like to thank Peter May for helpful conversations and suggestions throughout the project. I would also like to thank Peter May and Daniel Schaeppi for sharing their expertise on category theory. I appreciate the close readings and detailed comments from the anonymous referees. Above all, I would like to express my gratitude to Benson Farb, my PhD advisor, for suggesting this project, and for his ongoing guidance. His support made this work possible.

I am grateful for the support of a PGS D Scholarship from the Natural Sciences and Engineering Research Council of Canada.

\section{Background} \label{SectionBackground}
\subsection{The Weyl groups of classical type}

The classical Weyl groups comprise three one-parameter families of finite reflection groups. The symmetric group $S_n$ is the Weyl group of type A$_{n-1}$; the hyperoctahedral group (or signed permutation group) $B_n$ is the Weyl group of the (dual) root systems of types B$_n$ and C$_n$, and its subgroup the even-signed permutation group $D_n$ is the Weyl group of type D$_n$. We briefly review the representation theory of these groups. 

We note that the finite dimensional complex representations of $S_n$, $B_n$, and $D_n$ are defined over the rational numbers \cite[Theorem 5.4.5, Theorem 5.5.6, Corollary 5.6.4]{GeckPfeiffer}.

\subsubsection{The symmetric group $S_n$} \label{RepTheorySn}

The rational representation theory of the symmetric group $S_n$ is well understood; a standard reference is Fulton--Harris \cite{FultonHarris}. The irreducible representations of $S_n$ are in natural bijection with the set of partitions $\y$ of $n$, which we denote $$\y = ( \y_0, \y_1, \ldots, \y_r), \qquad \text{ with } \y_0 \geq \y_1 \geq \cdots \geq \y_r \text{ and } \y_0 + \y_1 + \cdots + \y_r=n.$$ Each integer $\y_i$ is a \emph{part} or \emph{addend} of the partition. We write $\y \vdash n$ or  $|\y|=n$ to indicate the size of the partition. The \emph{length} $\ell(\y)$ of $\y$ is the number of parts. We write $V_{\y}$ to denote the $S_n$--representation associated to $\y$. 

\subsubsection{The hyperoctahedral group $B_n$} \label{RepTheoryWn}

The hyperoctahedral group $B_n$ is the wreath product $$B_n=\Z / 2 \Z \wr S_n := (\Z / 2\Z)^n \rtimes S_n,$$ where $S_n$ acts on $(\Z / 2\Z)^n$ by permuting the coordinates. It is the symmetry group of the $n$--hypercube, dually, the $n$--hyperoctahedron. There is a canonical representation of $B_n$ as the group of \emph{signed permutation matrices}, that is,  $n\times n$ generalized permutation matrices with nonzero entries $\pm 1$. We can also characterize the hyperoctahedral group as the symmetry group of the set $$\{ \{-1,1\}, \{-2, 2\}, \ldots, \{-n, n\} \},$$ where the $k^{th}$ factor of $(\Z/2\Z)^n$ transposes the elements in the block $\{-k, k \}$, and $S_n$ permutes the $n$ blocks. As such, $B_n$ is also called the \emph{signed permutation group}. 

It is often convenient to consider $B_n$ as a subgroup of the symmetric group $S_{\Omega}$ that acts on the $2n$ letters $$\Omega = \{-1, 1, -2, 2, \ldots, -n, n \}.$$ We frequently write elements of  $B_n$ in the cycle notation of permutations of $\Omega$. \\

{ \noindent \bf  The rational representation theory of $B_n$.  \quad} The representation theory of the hyperoctahedral group was developed by Young in the 1920s \cite{YoungHyperoctahedral}, and further refined by authors including Mayer \cite{MayerTypeCn}; Geissinger and Kinch \cite{GeissingerKinch}; al-Aamily, Morris, and Peel \cite{alAamilyMorrisPeel}; and Naruse \cite{NaruseCn}. It is described in \cite{GeckPfeiffer}.

The rational irreducible representations of $B_n$ can be built up from those of the symmetric group $S_n$.  These irreducible $B_n$--representations are classified by \emph{double partitions} of $n$, that is, ordered pairs of partitions $(\y, \nu) \text{ with}$ $|\y|+|\nu|=n.$ 

 For $\y \vdash n$, define $V_{( \y, \varnothing)}$ to be the $B_n$--representation pulled back from $S_n$--representation $V_{\y}$ under the surjection $\pi : B_n \twoheadrightarrow S_n$.  Let $\Q^{\varepsilon}$ denote the one-dimensional representation associated to the character $$\varepsilon: B_n \cong (\Z / 2\Z)^n \rtimes S_n \twoheadrightarrow \{ \pm 1 \}$$ where the canonical generators of $(\Z/2\Z)^n$ act by $(-1)$, and elements of $S_n$ act trivially. Define $$V_{( \varnothing, \nu)}  : = V_{( \nu, \varnothing)} \otimes \Q^{\varepsilon}.$$ 
 \noindent Then, generally, for $\y \vdash m$ and $\nu \vdash (n-m)$, we define 
$$ V_{(\y, \nu)} := \Ind_{B_m \times B_{n-m}}^{B_n} V_{(\y, \varnothing)} \boxtimes V_{(\varnothing,\nu)}, $$ 
where $\boxtimes$ denotes the external tensor product of the $B_m$--representation $V_{(\y, \varnothing)}$ with the $B_{n-m}$--representation $V_{(\varnothing,\nu)}$. Each double partition $(\y, \nu)$ yields a distinct irreducible representation of $B_n$, and every irreducible representation has this form. \\

{ \noindent \bf  The conjugacy classes of $B_n$.  \quad} The conjugacy classes of $B_n$ were described by Young \cite{YoungHyperoctahedral}. More modern exposition can be found in, for example, Carter \cite{CarterConjugacy}, or Naruse \cite{NaruseCn}. A similar classification is found in \cite[Chapter 3]{GeckPfeiffer}. 

\begin{defn}{\bf (Positive and negative cycles in $B_n$; cycle structure).} \label{DefnBnCycleType} 

An element of $B_n$, viewed as a permutation of $\{ \pm 1, \ldots, \pm n\}$, can be decomposed into cycles. 

A \emph{negative $r$--cycle} is a factor of the form $$\sigma = (s_1 \; s_2 \; \cdots s_r \; -s_1 \; -s_2 \; \cdots -s_r).$$ The element $\sigma$ is mapped to an $r$--cycle under the canonical surjection $B_n \twoheadrightarrow S_n$, and $\sigma^r$ is a product of the $r$ involutions $(s_i \; -s_i )$. Negative cycles reverse the sign of an odd number of digits $\{\pm 1, \ldots, \pm n\}$.

A \emph{positive $r$--cycle} is a factor of the form $$\alpha = (a_1 \; a_2 \; \cdots \; a_r)(- a_1 \; - a_2 \; \cdots \; - a_r) \qquad \text{with $|a_i| \neq |a_j|$ if $i \neq j$}.$$ The element $\alpha$ maps to an $r$-cycle in $S_n$, and $\alpha^r$ is the identity element. Positive cycles negate an even number of digits. 

For example, $(1 \; 2)(-1\; -2)$ and $(1 \; -2)(-1 \; 2)$ are both examples of positive two-cycles in $B_n$, whereas  $(1 \; 2 \; -1\; -2)$  and $(1 \; -2 \; -1\; 2)$ are both negative two-cycles. \end{defn}

The cycle structure of an element $\w \in B_n$ is encoded by double partitions $(\y, \nu)$ of $n$, where $\y$ designates the lengths of the positive cycles, and $\mu$ designates the lengths of the negative cycles. The double partition $(\y, \nu)$ is called the \emph{signed cycle type} of the element $\w$. The following result dates back to Young \cite{YoungHyperoctahedral}. See also \cite[Proposition 24]{CarterConjugacy} and \cite[Proposition 3.4.7]{GeckPfeiffer}. 

\begin{prop} {\bf (Classification of conjugacy classes of $B_n$).} Two elements $x, y \in B_n$ are conjugate in $B_n$ if and only if they have the same signed cycle type. Thus the conjugacy classes of $B_n$ are classified by double partitions $(\y, \nu)$ of $n$. \end{prop} 

{ \noindent \bf  Branching rules and Pieri's formula for $B_n$.  \quad} The branching rules for $B_n$ are a main tool in our development of the theory of $\FI_{BC}$--modules over fields of characteristic zero. These rules are described in, for example, Geck--Pfeiffer \cite[Lemma 6.1.3]{GeckPfeiffer}: 

\begin{equation} \Ind_{B_a \times B_{n-a}}^{B_n} V_{(\y^+,\y^-)}\boxtimes V_{(\mu^+, \mu^-)} = \bigoplus_{(\nu^+, \nu^-)} C_{\y^+, \mu^+}^{\nu^+} C_{\y^-, \mu^-}^{\nu^-}  V_{(\nu^+, \nu^-)}  \label{Eqn:BranchingRule} \end{equation} 

\noindent where $ C_{\y, \mu}^{\nu}$ denotes the Littlewood--Richardson coefficient. We will use in particular Pieri's formula, the case where $V_{(\mu^+, \mu^-)}$ is the trivial representation $k=V_{( (n-a), \varnothing)}$.
\begin{align}
 \Ind^{B_n}_{B_a \times B_{n-a}} V_{(\y^+, \y^-)} \boxtimes k & = \bigoplus_{(\nu^+, \, \nu^-)} C_{\y^+, (n-a)}^{\nu^+} C_{\y^-, \varnothing}^{\nu^-}  V_{(\nu^+,\nu^-)} \notag \\
 & = \bigoplus_{\nu^+} C_{\y^+, (n-a)}^{\nu^+} V_{(\nu^+,\y^-)} \notag \\
 & = \bigoplus_{\nu^+ } V_{(\nu^+,\y^-)} \label{Eqn:WnInducedDecomp}
\end{align}
where the final sum is taken over all partitions $\nu^+$ that can be constructed by adding $(n-a)$ boxes to $\y^+$, with no two boxes added to the same column. Some small cases are shown in Figure \ref{fig:BnBranchingRules}.

\begin{figure}[h!]
\begin{center}
\fbox{ \includegraphics[scale=1.1]{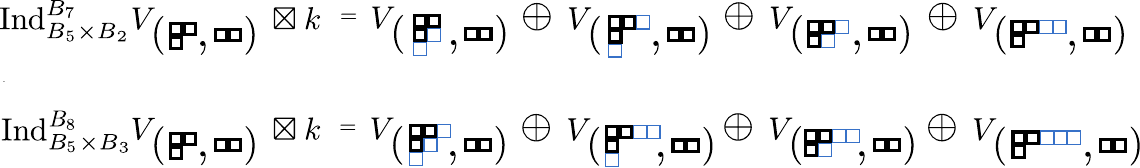}}
\caption{ {\small Illustrating the branching rules for $B_n$.}} 
\label{fig:BnBranchingRules}
\end{center}
\end{figure}

\noindent By Frobenius reciprocity, the multiplicity of $V_{(\y^+,\nu^-)}\boxtimes \; k$ in the restriction $\Res^{B_n}_{B_a \times B_{n-a}} V_{(\nu^+, \nu^-)}$ is 
\begin{align}
   \left\{ \begin{array}{ll} \label{Eqn:WnInvariantsDecomp}
         1 & \mbox{if $\nu^+$ can be constructed by removing $(n-a)$ boxes from distinct columns of $\y^+$,} \\
         0 & \mbox{otherwise.}\end{array}  \right. 
\end{align}

The decomposition of induced representations $\Ind_{B_{n-1}}^{B_n} V_{(\y^+, \y^-)}$ are described by Geck--Pfeiffer \cite[Lemma 6.1.9]{GeckPfeiffer}:
\begin{align}
 \Ind_{B_{n-1}}^{B_n} V_{(\y^+, \y^-)} = \bigoplus_{\overline{\y^+}} V_{(\overline{\y^+}, \y^-)} + \bigoplus_{\overline{\y^-}} V_{(\y^+, \overline{\y^-})} \label{Eqn:PieriGeneral} 
\end{align}
summed over all $\overline{\y^+}$ that can be constructed by adding a single box to $\y^+$, and all $\overline{\y^-}$ that can be constructed by adding a single box to $\y^-$.  By iteratively applying this law to the trivial $B_{n-m}$--module $k$, we find:
\begin{align}
 \Ind_{B_{n-m}}^{B_n} k & = \Ind_{B_{n-1}}^{B_n} \cdots \Ind_{B_{n-m}}^{B_{n-m+1}} V((n-m), \varnothing) \notag \\
                        & = \bigoplus_{\y^+, \y^-} V_{(\y^+, \y^-)} \label{Eqn:Pieri}  
\end{align}
summed over $V(\y^+, \y^-)$ with multiplicity equal to the number of ways that the double partition $(\y^+, \y^-)$ can be built up from $( (n-m), \varnothing)$ by adding one box at a time to either partition. There are no restrictions on columns, though the addition of each box must form a valid double partition. \\

{ \noindent \bf  Restriction from $B_n$ to $S_n$.  \quad} The restriction of a $B_n$--representation $V_{(\y^+, \y^-)}$ to $S_n \subseteq B_n$ is
\begin{align}
 \Res_{S_{n}}^{B_n} V_{(\y^+, \y^-)} = \bigoplus_{\y} C_{\y^+, \y^-}^{\y} V_{\y} \label{Eqn:ResWntoSn} 
\end{align}
where $C_{\y^+, \y^-}^{\y}$ again is the Littlewood--Richardson coefficient (Geck--Pfeiffer \cite[Lemma 6.1.4]{GeckPfeiffer}).

\subsubsection{The even-signed permutation group $D_n$} \label{RepTheoryDn}

We described a representation of the hyperoctahedral group $\varepsilon : B_n \to \Z / 2\Z$ that counts the number of $-1$'s (mod $2$) appearing in a signed permutation matrix $w$. The kernel of this map is the index--$2$ normal subgroup $D_n$ of $B_n$, the \emph{even-signed permutation group}. If we classify elements of $B_n$ by cycle type as in Definition \ref{DefnBnCycleType}, the subgroup $D_n$ comprises exactly those elements of $B_n$ with an even number of negative cycles.\\

{ \noindent \bf  The rational representation theory of $D_n$.  \quad} The representation theory of $D_n$ is given, for example, in \cite[Chapter 5.6]{GeckPfeiffer}. The irreducible representations derive from those of $B_n$. Given an irreducible representation $V_{(\y, \nu)}$ of $B_n$, the restriction to the action $D_n$ decomposes as either one or two distinct irreducible representations. When $\y \neq \nu$, the two irreducible $B_n$--representations  $V_{(\y, \nu)}$ and $V_{(\nu, \y)}$ restrict to the same representation of $D_n$; each distinct set of nonequal partitions $\{ \y, \nu \} $ gives a different irreducible representation $V_{\{\y, \nu\}}$ of $D_n$. When $n$ is even, for any partition $\y \vdash \frac{n}{2}$, the irreducible $B_n$--representation $V_{(\y, \y)}$ restricts to a sum of two nonisomorphic irreducible $D_n$--representations of equal dimension.  Thus, the irreducible representations of $D_n$ are classified by the set 
$$ \left\{ \{ \y, \nu \} \; \vert \; \y \neq \nu, \; |\y|+|\nu| = n \right\} \coprod  \left\{ (\y, \pm) \; \middle| \; |\y|=\frac{n}{2} \right\}, $$ 
with the 'split' irreducible representations $V_{(\y, +)}$ and $V_{(\y, -)}$ only occurring for even $n$. \\

{ \noindent \bf  The conjugacy classes of $D_n$. \quad} The structure of the conjugacy classes of $D_n$ was described by Young \cite{YoungHyperoctahedral}, and more recently by Carter \cite[Proposition 25]{CarterConjugacy} and \cite[Proposition 3.4.12]{GeckPfeiffer}. As with $B_n$, the conjugacy classes of $D_n$ are classified by signed cycle type, with one exception. When $n$ is even, the elements for which all cycles are positive and have even length are now split between two conjugacy classes, as follows:  

Suppose that $\y = (\y_1, \y_2, \ldots, \y_\ell)$ is a partition of $n$ with all parts $\y_i$ of even length. Then
$$ \alpha^+  := (1 \; 2 \; \cdots \; \y_1)\; (-1 \; -2 \; \cdots \; -\y_1) \; (1+\y_1 \; 2+\y_1 \ldots \y_2+\y_1) ( -1-\y_1 \; -2-\y_1 \ldots -\y_2-\y_1) \cdots $$
$$ \text{ and } \alpha^- : = (1 \;  -1) \alpha^+  (1 \; -1) $$
are representatives of the two conjugacy classes of elements with signed cycle type $(\y, \varnothing)$, which we will denote $(\y, +)$ and $(\y, -)$, respectively. In summary:

\begin{prop} {\bf (Classification of conjugacy classes of $D_n$)}
 The conjugacy classes of $D_n$ are classified by the set 
\begin{align*} & \{ (\y, \nu) \; | \; |\y| + |\nu| = n, \text{  $\nu$ has an even number of parts;} \text{if $\nu = \varnothing$ then not all parts of $\y$ are even } \} \\
\coprod \; & \left\{ (\y, \pm) \; | \; |\y| = n, \text{ all parts of $\y$ are even } \right\} 
\end{align*}
\noindent with the 'split' conjugacy classes $(\y, \pm)$ only occurring when $n$ is even. 
\end{prop}

\subsection{Representation stability} \label{BackgroundRepStability}

In a precursor to their work on FI--modules, Church and Farb \cite{RepStability} define a form of stability for a sequence $\{ V_n \}$ of $G_n$--representations, for various families of groups $G_n$ with inclusions $G_n \hookrightarrow G_{n+1}$, including the symmetric and hyperoctahedral groups. We recall their definitions, and additionally introduce a notion of stability for the even-signed permutation groups $D_n$.

For the symmetric groups $S_n$, in order to compare representations for different values of $n$, Church--Farb identify those irreducible representations associated to partitions of $n$ that differ only in their largest parts -- that is, two irreducible representations are considered 'the same' if the Young diagram for one can be constructed by adding boxes to the top row of the Young diagram for the other. 

Accordingly, for a partition $\y = ( \y_1, \y_2, \ldots, \y_t)$ of $m$, we write $V(\y)_n$ to denote the irreducible $S_n$--representation associated to $$\y[n]:= ( (n-m) , \y_1, \y_2, \ldots, \y_t)$$ whenever it is defined, that is,  
\begin{align*}
 V(\y)_n &:= \left\{ \begin{array}{ll}
         V_{\y[n]} & \mbox{$(n-m)\geq \y_1$},\\
        0 & \mbox{otherwise}.\end{array} \right.
\end{align*}
 We call  $\y[n] \vdash n$ the \emph{padded partition} associated to $\y$. 

Similarly, for the hyperoctahedral groups $B_n$, two double partitions are identified if they differ only in the largest part of the first partition. For a double partition $\y = ( \y^+, \y^-)$ with $\y^+ \vdash m $ and $\y^- \vdash \ell $, we define $$\y[n] : = ( \y^+[n-\ell], \y^-)$$ to be the \emph{padded double partition} associated to $\y=(\y^+, \y^-)$, and we write $V( \y )_n$ or $V( \y^+, \y^-)_n$ to denote the irreducible $B_n$--representation 
\begin{align*}
 V(\y)_n &:= \left\{ \begin{array}{ll}
         V_{\y[n]} & \mbox{$(n-\ell- m)\geq \y^+_1$},\\
        0 & \mbox{otherwise}.\end{array} \right.
\end{align*}

Finally, we introduce a stable notation for certain representations of the even-signed permutation groups $D_n$. Let $\y = ( \y^+, \y^-)$ be a double partition with $\y^+ \vdash m$ and $\y^- \vdash \ell $. Then we write $V( \y )_n$ to denote the $D_n$--representation $$V(\y)_n := \Res_{D_n}^{B_n} V(\y)_n.$$ Explicitly, $V(\y)_n$ is the $D_n$--representation
\begin{align*}
 V(\y)_n &= \left\{ \begin{array}{ll}
         V_{ \{ \y^+[n-\ell], \; \y^- \}} & \mbox{$(n-\ell -m)\geq \y^+_1$ and $\y^+[n-\ell] \neq \y^-$},\\
         V_{ \{  \y^-, \; + \}} \oplus V_{ \{  \y^-, \; - \}} & \mbox{$(n-\ell-m)\geq \y^+_1$ and $\y^+[n-\ell] = \y^-$}, \\
        0 & \mbox{otherwise}.\end{array} \right.
\end{align*}
 We note that  $V(\y)_n$ is an irreducible $D_n$--representation for all but at most one value of $n$.

\begin{defn} \label{DefnConsistent} {\bf (Consistent sequence).}
Let $\{ V_n \}$ be a sequence of $G_n$--representations with maps $\phi_n :V_n \to V_{n+1}$. The sequence $\{V_n, \phi_n \}$ is \emph{consistent} if $\phi_n$ is equivariant with respect to the $G_n$--action on $V_n$ and the $G_n$--action on $V_{n+1}$ under restriction to the subgroup $G_n \hookrightarrow G_{n+1}$. \end{defn}

\begin{defn} \label{DefnRepStability} {\bf (Representation stability).}
A consistent sequence $\{V_n, \phi_n\}$ of finite dimensional $G_n$--representations is \emph{representation stable} if it satisfies three properties: \\[-10pt]

\noindent I. \textbf{\emph{Injectivity.}} The maps $\phi_n :V_n \to V_{n+1}$ are injective, for all $n$ sufficiently large. \\[-10pt]

\noindent II. \textbf{\emph{Surjectivity.}} The image $\phi_n(V_n)$ generates $V_{n+1}$ as a $k[G_{n+1}]$--module, for all $n$ sufficiently large. \\[-10pt]

\noindent III. \textbf{\emph{Multiplicities.}} Decompose $V_n$ into irreducible $G_n$--representations: $$V_n = \bigoplus_{\y}c_{\y,n}V(\y)_n.$$ For each $\y$, the multiplicity $c_{\y,n}$ of $V(\y)_n$ is eventually independent of $n$. 

\end{defn}

\begin{defn} {\bf (Uniform representation stability).}
Let $\{ V_n, \phi_n \}$ be a representation stable sequence with the multiplicity $c_{\y,n}$ constant for all $n \geq N_{\y}$. The sequence $\{ V_n, \phi_n \}$ is \emph{uniformly} representation stable if $N = N_{\y}$ can be chosen independently of $\y$.
\end{defn}

\section{$\FIW$--modules and related constructions}

\subsection{The category $\FIW$}

In this section we recall the main definitions and establish some notation. Let $\W_n$ denote the Weyl group $S_n$, $D_n$, or $B_n$. In Definition \ref{DefnFIW} we defined the category $\FIW$ with objects indexed by the natural numbers $\Z_{\geq 0}$, and morphisms generated by its endomorphisms $\End(\bn) \cong \W_n$ and the canonical inclusions $I_n: \bn \hookrightarrow {\bf (n+1)}.$ 

 Throughout this paper, we will let $I_n$ denote this natural inclusion of sets, and write $I_{m,n}: \bm \to \bn$ to denote the composite $$ I_{m,n} := I_{n-1} \circ \ldots \circ I_m.$$  Any $\FIW$ morphism $f: \bm \to \bn$ factors as the composite of $I_{m,n}$ with some (signed) permutation $\s \in \W_n$. \\

{\noindent  \bf The stabilizer of $I_{m,n}$. \quad} The group $\W_n$ acts transitively on the set of morphisms $\Hom_{\FIW}(\bm, \bn)$ by postcomposition. We denote by $H_{m,n} = H_{m,n}^{\W}$ the stabilizer of $I_{m,n}$ in $\W_n$. As depicted in Figure \ref{fig:Stabilizer}, the group $H_{m,n}$ is the copy of $\W_{n-m} \subseteq \W_n$ that pointwise fixes the image $I_{m,n}(\bm) \subseteq \bn$. 

\begin{figure}[h!]
\begin{center}
\setlength\fboxsep{5pt}
\setlength\fboxrule{0.5pt}
\fbox{ \includegraphics[scale=3.5]{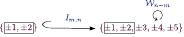}}
\caption{ {\small The stabilizer $H_{m,n}$}}
\label{fig:Stabilizer}
\end{center}
\end{figure}

\begin{rem}\label{HomFID=HomFIBC}  {\bf ($\Hom_{\FI_D}(\bm, \bn) = \Hom_{\FI_{BC}}(\bm, \bn)$ for $m\neq n$).} When $m < n$, the $\FI_D$ morphisms $\Hom_{\FI_D}(\bm, \bn)$ may reverse an even or odd number of signs. These morphisms are by definition generated by $I_{m, n}$ and $\End_{\FI_D}(\bn) \cong D_n$. For example, although $(1 \;-1)(n \; -n) \in D_n,$ the involution $(n \; -n)$ is in the stabilizer of $I_{m,n}$. Thus
$$(1 \;-1)(n \; -n) \circ I_{m,n} = (1 \;-1) \circ I_{m,n} \in \Hom_{\FI_D}(\bm, \bn) $$ negates only $\pm1$. When $m \neq n$, the set of $\FI_D$ morphisms $f: \bm \to \bn$ coincides exactly with the set of $\FI_{BC}$ morphisms $f: \bm \to \bn$.
 
\end{rem}


\subsection{$\FIW$--modules}

Recall from Definition \ref{DefnFIWModule} that an $\FIW$--module over a ring $k$ is a functor $V: \FIW \to k\text{--Mod}.$

For a fixed family of Weyl groups $\W_n$ and ring $k$, the set of all $\FIW$--modules over $k$ form a category. A \emph{map of $\FIW$--modules} $F: V \to V'$ is a natural transformation, that is, it is a sequence of maps $F_n: V_n \to V_n '$ that commute with the $\FIW$ morphisms in the sense that $$F_n \circ V(f) = V'(f) \circ F_m \qquad \text{for every $f \in \Hom_{\FIW}(\bm, \bn)$}.$$   

\begin{example} The spaces $V_n = \Q^n$ form an $\FIW$--module with the canonical action of $\W_n$ by (signed) permutation matrices, and the standard inclusions $(I_n)_*: \Q^n \hookrightarrow \Q^{n+1}.$
\end{example}

\begin{example} Church--Ellenberg--Farb showed in \cite[Proposition 2.56]{CEF} that, for any partition $\y$ of $n$, the sequence of $S_n$--representations $V_n = V(\y)_n$ admits an $\FI_A$--module structure. We will show in Definition \ref{Defn:FIModuleV(y)} and Proposition \ref{FIModuleV(y)} that, analogously, for any double partition $\y=(\y^+, \y^-)$ of $n$, the sequence of $B_n$--representations $V_n = V(\y)_n$ admits an $\FI_{BC}$--module structure. Restriction of this $\FI_{BC}$--module to $\FI_D$ gives the sequence of $D_n$--representations $V_n = V(\y)_n$ an $\FI_D$--module structure.
\end{example}

{\noindent \bf Recognizing $\FIW$--modules. \quad} An $\FIW$--module gives a consistent sequence of $\W_n$--representations in the sense of Definition \ref{DefnConsistent}, as the images of the natural inclusions $I_n$ give maps $\phi_n = (I_n)_*: V_n \to V_{n+1}$ compatible with the action of $\W_n = \End(\bn)$. Not all consistent sequences arise from $\FIW$--modules, however. The following lemma gives necessary and sufficient conditions for a consistent sequence $\{V_n, \phi_n \}$ of $\W_n$--representations to have the structure of an $\FIW$--module. 

\begin{lem} \label{PromotionToFI} {\bf ($\FIW$--modules vs. consistent sequences).}
  A consistent sequence $\{V_n, \phi_n \}$ of $\W_n$--representations can be promoted to an $\FIW$--module with $\phi_n = (I_n)_*$ if
and only if, for all $m, n$, the stabilizer $$H_{m,n}:= \Stab(I_{m,n}) \cong \W_{n-m}$$ acts trivially on the image of $I_{m,n}(V_m) \subseteq V_n$. 
\end{lem}

\begin{proof} An element of $\tau \in H_{m,n}$ by definition satisfies $\tau \circ I_{m,n} = I_{m,n}$, so given any $\FIW$--module $V$ these elements necessarily act trivially on the image $(I_{m,n})_*(V_m) \subseteq V_n.$  

Conversely, consider a consistent sequence $\{V_n, \phi_n\}$ of $\W_n$--representations. Define $\phi_{m,n}$ to be the composite $\phi_{n-1} \circ \cdots \circ \phi_m.$ Given any $f \in \Hom_{\FIW}(\bm, \bn)$, factor $f = \s \circ I_{m,n}$ for some $\s \in \W_n$. We can then realize $\{V_n, \phi_n\}$ as an $\FIW$--module by assigning $$ V: f \mapsto f_* := \s_* \circ \phi_{m,n};$$ the condition of the lemma is precisely the condition needed to ensure that this assignment is well-defined, independent of choice of factorization of $f$. It is straightforward to check that the consistency of the sequence $\{V_n, \phi_n \}$ ensures that the assignment $f \mapsto f_*$ respects composition. 
\end{proof}

The result of Lemma \ref{PromotionToFI} was proven for $\FI_A$--modules by Church--Ellenberg--Farb \cite[Lemma 2.1]{CEF}; they show that a consistent sequence $\{V_n, \phi_n \}$ of $S_n$--representations can be promoted to an $\FI_A$--module if and only if for all $m \leq n$, $\s, \s' \in S_n$, and $v \in V_n$ lying in the image of $V_m$ ,
$$ \s \vert _{ \{ 1, 2, \ldots, m \} } = \s' \vert _{ \{ 1, 2, \ldots, m \} } \qquad \Longrightarrow \qquad  \s (v) = \s'(v).$$

\begin{example}\label{NonExampleRegularReps}{\bf (The regular representations do not form an $\FIW$--module).} When $\W_n$ is any of $S_n$, $B_n$, or $D_n$, the sequence of regular representations $V_n := k[\W_n]$ is a consistent sequence that is not an $\FIW$--module. In each case, for example, the permutation that transposes $n$ and $(n-1)$ acts nontrivially on the image $I_{(n-2), n}(V_{n-2})$, violating the conditions of Lemma \ref{PromotionToFI}.
\end{example}

\begin{example}\label{NonExampleAltReps}{\bf (Alternating and sign representations do not form $\FIW$--modules).}  A second example: The sequence of alternating representations $V_n \cong k$ of the symmetric groups $S_n$, or its pullbacks to $B_n$ or $D_n$, give a consistent sequence with no $\FIW$--module structure. Again, the $2$-cycle that transposes $n$ and $(n-1)$ acts nontrivially on the image $I_{(n-2), n}(V_{n-2})$. For similar reasons, the sign representations $\varepsilon$ defined in Section \ref{RepTheoryWn} provide a consistent sequence of $B_n$--representations with no $\FI_{BC}$--module structure.    
\end{example}

In summary: to verify that a sequence $\{V_n, \phi_n \}$ has $\FIW$--module structure, we must check two conditions. The sequence must be consistent in the sense of Definition \ref{DefnConsistent}, and it must satisfy the condition on stabilizers described in Lemma \ref{PromotionToFI}.

\subsection{The $\FIW$--modules $M_{\W}(\bm)$ and $M_{\W}(U)$}

In analogy to \cite[Definition 2.5]{CEF}, we define the $\FIW$--modules $M_{\W}(\bm)$. These are in a sense the 'free' finitely generated $\FIW$--modules; we will see in Proposition \ref{FinGen} that every finitely generated $\FIW$--module is a quotient of a sum of $\FIW$--modules of this form. This property will be critical to our development of the theory of $\FIW$--modules. 

\begin{defn} {\bf (The $\FIW$--module $M_{\W}(\bm)$).} \label{Defn:M(m)} Define $M_{\W}(\bm)$ to be the $\FIW$--module such that $M_{\W}(\bm)_n$ is the $k$-module with basis $\Hom_{\FIW}(\bm, \bn)$ and an action of $\W_n$ by post-composition. 

Since $\W_n$ acts transitively on $\Hom_{\FIW}(\bm, \bn)$, we can identify the $\W_n$--set $\Hom_{\FIW}(\bm, \bn)$ with the cosets of the stabilizer $H_{m,n}:=\Stab(I_{m,n}) \cong \W_{n-m} \subseteq \W_n.$ This gives an isomorphism of $\W_n$--representations $$M_{\W}(\bm)_n \cong \Ind_{\W_{n-m}}^{\W_n} k$$ where $k$ has a trivial $\W_m$ action. Over a field of characteristic zero, the decomposition of these representations are described in Pieri's rules; see Equation (\ref{Eqn:Pieri}) for the hyperoctahedral formula.
 \end{defn} 
 
 Observe $M_{\W}(\bm)_n = 0$ when $n<m.$ The first nonzero degree $n=m$ is the regular representation $M_{\W}(\bm)_m \cong k[\W_m].$ In general, $M_{\W}(\bm)_n$ can be considered the permutation representation of $\W_n$ on the set of $m$-tuples $$\Big(f(1), f(2), \ldots, f(m)\Big), \qquad f(j) \in \bn,$$ that designate the images of the $\FIW$ morphisms $f: \bm \to \bn$.  
 
\begin{example}{\bf ($M_{\W}({\bf 0})$ and $M_{\W}({\bf 1})$)} The $\FIW$--module $M_{\W}({\bf 0})$ is the sequence of trivial representations $M_{\W}({\bf0})_n \cong k.$ 

The $\FI_A$--module $M_{A}({\bf 1})$ is the sequence of canonical $S_n$--representations as permutation matrices. Over characteristic zero, in the notation of Section \ref{BackgroundRepStability}, we get the following decomposition:
$$M_{A}({\bf 1})_n \cong V( \Y{1})_n \oplus V(\varnothing)_n \qquad \text{for all $n$}.$$

The $\FI_{BC}$--module $M_{BC}({\bf 1})$ is the sequence of $(2n)$-dimensional representations of $B_n$ permuting a basis $\{ e_{1}, e_{-1}, \ldots, e_{n}, e_{-n} \}$. Over characteristic zero, $M_{BC}({\bf 1})_n$ decomposes as follows.
$$ M_{BC}({\bf 1})_n = V( \varnothing, \Y{1})_n \oplus V( \Y{1}, \varnothing)_n \oplus V( \varnothing, \varnothing)_n \qquad \text{for all $n$.} \qquad $$ Here,
 $$ V( \varnothing, \Y{1})_n = \Big\langle (e_1 - e_{-1}), \ldots , (e_n - e_{-n}) \Big\rangle $$ is the canonical $B_n$--representation by signed permutation representations, and 
$$ V( \Y{1}, \varnothing)_n \oplus V( \varnothing, \varnothing)_n = \Big\langle (e_1 + e_{-1}), \ldots , (e_n + e_{-n}) \Big\rangle $$ 
is the pullback of the canonical $S_n$ permutation representation. It is an exercise to verify that these decompositions are consistent with Pieri's rule, Equation (\ref{Eqn:Pieri}).

The representation $M_{D}({\bf 1})_1$ is trivial, but for $n>1$ the $D_n$--representation $M_{D}({\bf 1})_n$ is the restriction of the $B_n$--representation $M_{BC}({\bf 1})_n$ described above. 
\end{example}

\begin{rem} \label{RemarkMDisMBC} Recall from Remark \ref{HomFID=HomFIBC} that $\Hom_{\FI_D}(\bm, \bn) = \Hom_{\FI_{BC}}(\bm, \bn)$ whenever $m\neq n$.
There is therefore an isomorphism if $D_n$--representations $$M_D(\bm)_n \cong \Res^{B_n}_{D_n} M_{BC}(\bm)_n \qquad \text{whenever $m \neq n$.}$$ 
 These isomorphisms will be crucial to our study of induction of $\FI_D$--modules in Section \ref{Section:Induction}.
\end{rem}

\subsubsection{An adjunction}

\begin{defn}
 Let $\W_m$--Rep denote the category of $\W_m$--representations over a ring $k$. For each fixed integer $m \geq 0$, analogous to the definition of $\mu_m$ given by Church--Ellenberg--Farb \cite{CEF}, we define the functor
\begin{align*}
\mu_m: \W_m\text{--Rep} & \longrightarrow \FIW \text{--Mod} \\
U & \longmapsto M_{\W}(\bm) \otimes_{k[\W_m]} U.
\end{align*}
\end{defn}
As in \cite[Proposition 2.6]{CEF}, we note that since $M_{\W}(\bm)_n \cong \Ind_{\W_{n-m}}^{\W_n} k$ is isomorphic to $k[\W_n / \W_{n-m}],$ we can equivalently describe $\mu_m$ by the formula:
\begin{align*}
 (\mu_m(U))_n &= \left\{ \begin{array}{ll}
         0 & \mbox{$ n < m $},\\
        \Ind_{\W_m \times \W_{n-m}}^{\W_n} U \boxtimes k & \mbox{$ n \geq m$}.\end{array} \right.
\end{align*}
where $\boxtimes$ denotes the external tensor product, and $k$ denotes the trivial $\W_{n-m}$--representation.

\begin{prop} \label{Adjunction}
 The functor $$\mu_m:  \W_m\text{--Rep}  \longrightarrow \FIW \text{--Mod}$$ is the left adjoint to the forgetful functor
 \begin{align*}
\pi_m :  \FIW \text{--Mod} & \longrightarrow \W_m \text{--Rep} \\
 V & \longmapsto V(\bm). 
\end{align*}
\end{prop}

The proof of the adjunction follows from the same argument given for \cite[Proposition 2.6]{CEF}, by considering any Weyl group $\W_m$ in place of the symmetric group $S_m$. 

We remark that $$\mu_m(k[\W_m]) = M_{\W}(\bm) \otimes_{k[\W_m]} k[\W_m] \cong M_{\W}(\bm) .$$ 

 More generally, if $U$ is a finite-dimensional $\W_m$--representation, we denote $\mu_m(U)$ by $M_{\W}(U)$. Following \cite[Definition 2.7]{CEF}, we extend the functor $M_{\W}$ to $\bigoplus_m \W_m$--Rep. 

\begin{defn}
 Define $M_{\W}$ to be the map 
\begin{align*}
M_{\W}: \bigoplus_m \W_m \text{--Rep} & \longrightarrow \FIW \text{--Mod} \\
  U_m & \longmapsto   \mu_m (U_m)
\end{align*}
\end{defn}

\subsection{Generation of $\FIW$--modules}

Church--Ellenberg--Farb defined notions of span, finite generation, and degree of generation for FI--modules, which apply equally in the more general context of $\FIW$--modules. These definitions are summarized below. 

\begin{defn} {\bf (Span; Generating set).}
 If $V$ is an $\FIW$--module, and $S$ is a subset of the disjoint union $\coprod V_n$, then the \emph{span} of $S$, denoted span$_V(S)$, is the minimal $\FIW$--submodule of $V$ containing the elements of $S$. We call span$_V(S)$ the \emph{sub-FI--module generated by $S$}.
\end{defn}

Recall from Definition \ref{DefnFinGen} that an $\FIW$--module $V$ is \emph{finitely generated} if there is a finite set of elements $S = \{v_1, \ldots, v_l \} \subseteq \coprod V_n$ such that span$_V(S)=V$. Moreover $V$ is \emph{generated in degree $\leq m$} if $V= \text{ span}_V(\coprod_{i=0}^{m} V_i).$ We call the minimum such $m$ the \emph{degree of generation} of $V$, if it exists.

\begin{example} \label{M(m)FinGen}
 The $\FIW$--module $M_{\W}(\bm)$ is generated in degree $m$ by the identity map $\text{id}_m$ in $\Hom_{\FIW}(\bm, \bm),$ the basis for $M_{\W}(\bm)_{m}$. More generally, given a nonzero $\W_m$--representation $U$, the $\FIW$--module $M_{\W}(U) :=  M_{\W}(\bm) \otimes_{k[\W_m]} U$ is generated in degree $m$ by $M_{\W}(U)_m = U$.
\end{example}

Just as in \cite[Remark 2.13, Proposition 2.16]{CEF}, the finitely generated $\FIW$--modules are precisely those which admit a surjection by an $\FIW$--module of the form $\oplus_{i} M_{\W}(\bm_i)$. 

\begin{prop}\label{FinGen}
 An $\FIW$--module is finitely generated in degree $\leq m$ if and only if it admits a surjection $\oplus_{i} M_{\W}(\bm_i) \twoheadrightarrow V$ for some finite sequence of integers $\{m_i\}$, with $m_i \leq m$ for each $i$. 
\end{prop}

\begin{proof}
Given any finitely generated $\FIW$--module $V = \Span_V( v_1, \ldots, v_{\ell})$, with $v_i \in V_{m_i}$, the map 
\begin{align*}
\bigoplus_{i=1}^{\ell} M_{\W}(\bm_i) & \longrightarrow V \\
f  & \longmapsto f_*(v_i)  \qquad f \in \Hom_{\W}(\bm_i, \bn), \text{ the basis for } M_{\W}(\bm_i)_n
\end{align*}
is the desired surjection of $\FIW$--modules.  

Conversely, the image of an $\FIW$--module $\bigoplus_{i=1}^{\ell} M_{\W}(\bm_i)$ under an $\FIW$--module map is generated by the images of the identity morphisms $\{\text{id}_{m_i} \}_{i=1}^{\ell}$. \end{proof}

Given an $\FIW$--module $V$, any $n$, and any $v \in V_n$, then we have a surjective map of $\FIW$--modules $$M_{\W}(\bn) \twoheadrightarrow \text{Span}_V(\{v\}) \subseteq V \qquad \text{ given by $f \mapsto f_*(v)$.}$$ Moreover, any map $M_{\W}(\bn) \to V$ can be described in this way by taking $v$ to be the image of id$_n \in M_{\W}(\bn)_n$. This observation is a form of Yoneda lemma for the category of $\FIW$--modules. 

\begin{rem} \label{M(U)Surjects} {\bf ($M_{\W}(U) \twoheadrightarrow \Span(U)$).} Given an $\FIW$--module $V$, and $\W_m$ subrepresentation $U$ of $V_m$, then by an argument as in Proposition \ref{FinGen}, the $\FIW$--module $$M_{\W}(U):= U \otimes_{k[\W_m]} M_{\W}(\bm)$$ surjects onto the span of $U$ in $V$.  \end{rem}

In \cite[Proposition 2.17]{CEF}, Church--Ellenberg--Farb describe the compatibility of degree of generation, and finite generation, with short exact sequences of FI--modules. Their results hold for $\FIW$--modules:

\begin{prop} Let $0 \to U \to V \to Q \to 0$ be a short exact sequence of $\FIW$--modules. If $V$ is generated in degree $\leq m$ (resp. finitely generated), then $Q$ is generated in degree $\leq m$ (resp. finitely generated). Conversely, if both $U$ and $Q$ are generated in degree \
$\leq m$ (resp. finitely generated), then $V$ is generated in degree $\leq m$ (resp. finitely generated).
\end{prop}

These statements can be shown by considering images or lifts of an appropriate generating set. 

\begin{defn}{\bf (Finite Presentation).}
 A finitely generated $\FIW$--module $V$ is \emph{finitely presented} with \emph{generator degree} $g$ and \emph{relation degree} $r$ if there is a surjection $$ \bigoplus_{m=1}^g M_{\W}(\bm)^{\oplus \ell_m} \twoheadrightarrow V $$ with a kernel finitely generated in degree at most $r$. 
\end{defn}

The Noetherian property, proved in Section \ref{SectionNoetherian} below, implies that all finitely generated $\FIW$--modules are in fact finitely presented. 

\subsubsection{The functor $H_0$}

In analogy with \cite[Definition 2.18]{CEF}, we define a functor $$H_0:  \FIW \text{--Mod} \to  \bigoplus_m \W_m \text{--Rep}$$ with the property that $H_0(M_\W(U))_m = U_m,$ that is, $H_0$ is a left inverse to $M_{\W}$.  

\begin{defn} \label{DefnH0} {\bf (The Functor $H_0$).} Given an $\FIW$--module $V$, we define the functor $H_0$ by
\begin{align*}
H_0 & : \FIW \text{--Mod}   \longrightarrow \bigoplus_m \W_m \text{--Rep} \\
& (H_0(V))_n =  V_n \bigg/ \bigg(\text{span}_V \big(  \coprod_{k<n} V_k \big) \bigg)_n
\end{align*}

The spaces $(H_0(V))_n$ are a minimal set of $\W_n$--representations generating the $\FIW$--module $V$. As noted in \cite{CEF}, these representations vanish for $n>m$ if and only if $V$ is generated in degree $\leq m$, and moreover $V$ is finitely generated if and only if $H_0(V)$ is a finitely generated $k$--module.

We can put an $\FIW$--module structure on the $\W_n$--representations $(H_0(V))_n$ by letting $I_n$ act by $0$ for all $n$. We denote this $\FIW$--module by $H_0(V)^{\FIW}$. \end{defn}

There is a natural surjection $V \twoheadrightarrow H_0(V)^{\FIW}.$ Note that we could equivalently characterize the $\FIW$--module $H_0(V)^{\FIW}$ as the largest quotient of $V$ with the property that all $\FIW$ morphisms $f: \bm \to \bn$ with  $m \neq n$ act by $0$: in any such quotient, all images $f_*(V_m) \subseteq V_n$ must necessarily be $0$.  

\begin{rem} \label{SurjectionHoV} {\bf($ M_{\W} (H_0(V)) \twoheadrightarrow V$).} 
 Let $V$ be an $\FIW$--module over characteristic zero. As suggested by Remark \ref{M(U)Surjects}, there is a (noncanonical) surjection $ M_{\W} (H_0(V)) \twoheadrightarrow V.$
The proof given in \cite[Proposition 2.43]{CEF} for $\FI_A$--modules applies directly to the cases of $\FI_{BC}$ and $\FI_D$. 
\end{rem}

\subsection{Restriction of $\FIW$--modules} \label{Section:Restriction}

The natural embeddings $S_n \hookrightarrow D_n \hookrightarrow B_n$ give inclusions of categories 
$\FI_A \hookrightarrow \FI_{D} \hookrightarrow \FI_{BC},$ which define restriction operations on the corresponding $\FIW$--modules. These operations, together with the \emph{induction} functors that we will define in Section \ref{Section:Induction}, will be our main tools for studying the interactions of the three families of Weyl groups. 
 
 Notably, we will show in Proposition \ref{RestrictionPreservesFinGen} that restriction of $\FIW$--modules preserves the property of finite generation. We will use this result to establish the Noetherian property for $\FI_D$ and $\FI_{BC}$--modules, Theorem \ref{Noetherian}. We use Proposition \ref{RestrictionPreservesFinGen} again to prove \cite[Theorem \ref{FIW2-PolyDimCharP}]{FIW2}, which states that the dimensions of finitely generated $\FI_D$ and $\FI_{BC}$--modules over arbitrary fields are eventually polynomial. In both cases, Proposition \ref{RestrictionPreservesFinGen} reduces the proofs to the type A case, which are established by Church--Ellenberg--Farb--Nagpal \cite{CEF}. 

\begin{defn}{\bf (Restriction).}
 Given a family of inclusions $\W_n \hookrightarrow \oW_n$, any $\oFIW$--module $V$ inherits the structure of an $\FIW$--module by restricting the functor $V$ to the subcategory $\FIW$ in $\oFIW$. We call this construction $\Res_{\W}^{\oW} V$, \emph{the restriction of $V$ to $\FIW$.}
\end{defn}

\begin{prop} \label{RestrictionPreservesFinGen} {\bf (Restriction preserves finite generation).}
For each family of Weyl groups $\W \subseteq \oW$, the restriction $\Res_{\W}^{\oW} V$ of a finitely generated $\oFIW$--module $V$ is finitely generated as an $\FIW$--module. Specifically, 
\begin{enumerate}
 \item \label{WntoSn} Given an $\FI_{BC}$--module $V$ finitely generated in degree $\leq m$, $\Res^{BC}_{A} V$ is finitely generated as an $\FI_A$--module in degree $\leq m$.
 \item \label{WntoW'n} Given an $\FI_{BC}$--module $V$ finitely generated in degree $\leq m$, $\Res^{BC}_{D} V$ is finitely generated as an $\FI_D$--module in degree $\leq m$.
\item \label{W'ntoSn} Given an $\FI_D$--module $V$ finitely generated in degree $\leq m$, $\Res^{D}_{A} V$ is finitely generated as an $\FI_A$--module in degree $\leq (m+1)$.
\end{enumerate}
\end{prop}

\begin{proof}[Proof of Proposition \ref{RestrictionPreservesFinGen}(\ref{WntoSn})]

The key to the proof is the fact that for each $m, n$ with $m \leq n$, the actions of $S_n$ on the right and $B_m$ on the left are together transitive on the cosets $ B_n / B_{n-m} \cong \Hom_{\FI_{BC}}(\bm, \bn).$  

 We first prove the claim for the $\FI_{BC}$--module $M_{BC}(\bm)$ for fixed $m$. Recall that $$M_{BC}(\bm)_n = \Span_k \{ e_f \; | \; f \in \Hom_{\FI_{BC}}(\bm, \bn) \};$$  we identify $\Hom_{\FI_{BC}}(\bm, \bn)$ with the set of inclusions $$f:\{ \pm1 , \pm 2, \ldots, \pm m \} \hookrightarrow \{ \pm1 , \pm 2, \ldots, \pm n \}$$ satisfying $f(-c) = -f(c)$ for all $c = \pm 1, \ldots, \pm m.$
 
Take as generating set $$S = \{ e_w \; | \; w \in \Hom_{\FI_{BC}}(\bm, \bm) \cong B_m \}, \qquad \text{basis for for $M_{BC}(\bm)_m$},$$ and take any inclusion $ f \in \Hom_{\FI_{BC}}(\bm, \bn)$; we will show $e_f$ is in the $\FI_A$ span of $S$. There is some $\sigma^{-1} \in S_n$ so that the postcomposite $ \sigma^{-1} \circ f$ has image $\{ \pm 1, \pm 2, \ldots, \pm m \} \subseteq \{ \pm 1, \pm 2, \ldots , \pm n \}.$

Additionally, there  is some $w^{-1} \in B_m$ so that the precomposite $ \sigma ^{-1} \circ f \circ w^{-1}$ is the natural inclusion $I_{m,n}$. Thus $f = \sigma \circ I_{m,n} \circ w,$ and $e_f =  \big( \sigma_* \circ (I_{m,n})_* \big) (e_w)$ is in the $\FI_A$--span of $S$. 

\begin{figure}[h!]
\begin{center}
\setlength\fboxsep{5pt}
\setlength\fboxrule{0.5pt}
\fbox{ \includegraphics[scale=4]{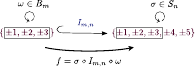}}
\caption{ {\small $\Hom_{\FI_{BC}}(\bm, \bn) = S_n \cdot  I_{m,n} \cdot B_m$ }}
\label{fig:Factorization}
\end{center}
\end{figure}

\noindent It follows that $\Res^{BC}_A M_{BC}(\bm)$ is finitely generated as an $\FI_A$--module by degree--$m$ generators. 

Now, let $V$ be any finitely generated $\FI_{BC}$--module. By Proposition \ref{FinGen}, there is an $\FI_{BC}$--module map  $\bigoplus_{a=0}^{m} M_{BC}(\ba)^{\oplus b_a} \twoheadrightarrow V$ which consists of a sequence of surjections of the underlying $k$--modules. Considered as a map of $\FI_A$--modules, this same map is a surjection $$\Res_{A}^{BC} \bigg( \bigoplus_{a=0}^{m} M_{BC}(\ba)^{\oplus b_a}  \bigg) =  \bigoplus_{a=0}^{m} \big( \Res_{A}^{BC}  M_{BC}(\ba) \big)^{\oplus b_a} \twoheadrightarrow \Res_{A}^{BC} V.$$ It follows that $\Res_{A}^{BC} V$ is finitely generated over $\FI_A$ by generators of degree $\leq m$. 
\end{proof}

\begin{proof}[Proof of Proposition \ref{RestrictionPreservesFinGen}(\ref{WntoW'n})]
 This follows from Proposition \ref{RestrictionPreservesFinGen}(\ref{WntoSn}), which implies that $\Res^{BC}_{D} V$ is finitely generated in degree $\leq m$ by the action of $\FI_A  \subseteq \FI_D$. 
\end{proof}

\begin{proof}[Proof of Proposition \ref{RestrictionPreservesFinGen}(\ref{W'ntoSn})]

 The proof of Proposition \ref{RestrictionPreservesFinGen}(\ref{W'ntoSn}) is similar to that of Proposition \ref{RestrictionPreservesFinGen}(\ref{WntoSn}). However, $B_m$ acts transitively by precomposition on the subset of maps in $\Hom_{\FI_{BC}}(\bm, \bn)$ with a given image, whereas when $n>m$ there are two orbits of maps in $\Hom_{\FI_D}(\bm, \bn)$ with a given image under the action of $D_m$ -- the orbit of those maps which reverse an even number of signs, and the orbit of those maps which reverse an odd number. For this reason, $\Res_{A}^{D} M_{D}(\bm)$ is not generated in degree $\leq m$. 

We again begin with the $\FI_D$--module $M_{D}(\bm)$. We have $$M_D(\bm)_n = \Span_k \{ e_f \; | \; f \in \Hom_{\FI_D}(\bm, \bn) \};$$ where each $f$ is an inclusion $f:\{ \pm1 , \pm 2, \ldots, \pm m \} \hookrightarrow \{ \pm1 , \pm 2, \ldots, \pm n \}$ satisfying $f(-c) = -f(c)$ for all  $c = \pm 1, \ldots, \pm m.$ If $m=n$, then $f$ must reverse an even number of signs; if $m<n$, then $f$ can reverse an even or odd number of signs. 

Take as generating set the bases for $M_D(\bm)_{m}$ and $M_D(\bm)_{m+1}$, $$S = \{ e_w \; | \; w \in \Hom_{\FI_D}(\bm, \bm) \text{ or } \Hom_{\FI_D}(\bm, {\bf (m+1)})  \}.$$  Suppose $n>m$, and let $ f \in \Hom_{\FI_D}(\bm, \bn)$. Take $\sigma^{-1} \in S_n$ so that $ \sigma^{-1} \circ f$ has image $\{ \pm 1, \pm 2, \ldots, \pm m \}$ in $\{ \pm 1, \pm 2, \ldots , \pm n \}.$  Then there is some $g \in \Hom_{\FI_D}(\bm, {\bf m+1})$ so that $ \sigma^{-1} \circ f = I_{m+1,n} \circ g$, and so $e_f = \sigma_* \circ (I_{m+1,n})_* (e_g).$ Thus $M_D(\bm)$ is generated by the generators $S$ in degrees $m$ and $(m+1)$. 

\begin{figure}[h!]
\begin{center}
\setlength\fboxsep{5pt}
\setlength\fboxrule{0.5pt}
\fbox{ \includegraphics[scale=4]{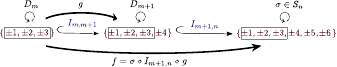}}
\caption{ {\small $\Hom_{\FI_{D}}(\bm, \bn) = S_n \cdot  I_{m+1,n} \cdot \Hom_{\FI_D}(\bm, {\bf (m+1)})$ }}
\label{fig:FactorizationTypeD}
\end{center}
\end{figure}

\noindent Again,  any finitely generated $\FI_D$--module $V$ admits a surjection by some $\FI_D$--module of the form $\bigoplus_{a=0}^{m} M_D(\ba)^{\oplus b_a}.$ It follows that $\Res^{D}_{A} V$ is generated by the images of generating sets for $\Res^{D}_{A} M_D(\bm_i)$ for each $i$, each in degree $\leq (m+1)$.
\end{proof}

\begin{rem} {\bf ($\Res^{B_n}_{S_n}$ does not preserve 'surjectivity' of consistent sequences).}
 We note the $\FI_{BC}$--module structure in Proposition \ref{RestrictionPreservesFinGen}(\ref{WntoSn})\ is necessary. Consider, for example, the sequence of regular representations $k[B_n]$ and inclusions $k[B_{n-1}] \hookrightarrow k[B_n].$ This sequence does not have an $\FI_{BC}$--module structure, but does form a consistent sequence in the sense of Definition \ref{DefnConsistent}. It 'surjects' in the sense of Definition \ref{DefnRepStability}, that is, for each $n$ the image of $k[B_{n-1}]$ generates $k[B_n]$ as a $k[B_n]$--module. The restriction of this sequence to $S_n$ gives a consistent sequence of $S_n$--representations that fails to 'surject', since (for example) the basis element of $k[B_n]$ corresponding to the signed permutation matrix $- id$ is not in the $S_n$--span of the image of $k[B_{n-1}]$ for any $n$. 
\end{rem}

\subsection{Induction of $\FIW$--modules} \label{Section:Induction}

In Section \ref{Section:Restriction} we analyzed the restriction functor on $\FIW$--modules. Just as with group representations, restriction has a left adjoint, a procedure for inducing $\FI_A$ and $\FI_D$--modules up to functors from $\FI_D$ or $\FI_{BC}$. This construction, which uses the theory of \emph{Kan extensions}, was described to us by Peter May. In this section we will define induction of $\FIW$--modules and establish some properties of this operation. 

For present purposes, we are particularly interested in studying induction from $\FI_{D}$ to $\FI_{BC}$. This will enable us to use our theory of $\FI_{BC}$--modules to recover results for finitely generated $\FI_D$--modules, including representation stability (Section \ref{SectionFinGenRepStability}) and existence of character polynomials \cite[Section \ref{FIW2-SectionFinGenCharPoly}]{FIW2}. The results for $\FI_{BC}$ make extensive use of the branching rules for the hyperoctahedral group, but the $D_n$ analogues of these rules are more troublesome. The properties of induction established here make our main results accessible in type D. 

\begin{rem}\label{NaiveInduction}{\bf (The naive definition of induction).} We note that the naive ``pointwise'' definition of induction of $\FIW$--modules is not well defined: If we were to define $\Ind^{\oW}_{\W} V$ so that in degree $n$ it were the representation $ \Ind^{\oW_n}_{\W_n} V_n,$ then the resulting sequence would not in general have the structure of an FI$_{\oW_n}$--module. 
 
Consider, for example, the sequence of trivial $D_n$--representations, with $V_n = k$ for all $n$, and all $\FI_D$ maps acting as isomorphisms. Then $\Ind_{D_n}^{B_n} k  \cong k \oplus k^{\varepsilon}$ is a sum of the trivial representation $k$ and the one-dimensional sign representation  $k^{\varepsilon}$ associated to the character $\varepsilon: B_n \to B_n / D_n \cong \{\pm 1\}.$ This cannot be a $\FI_{BC}$--module since, for example, the signed permutation $(-n \; n) \in B_n$ acts by multiplication by $-1$ on a summand of the image  $I_{m,n}(V_m)\subseteq V_n$ for any $m < n$, in violation of Lemma \ref{PromotionToFI}. 
\end{rem}

There is, however, a natural way to define induction of $\FIW$--modules, using a standard category--theoretic universal construct: the left Kan extension. General constructions and properties of Kan extensions are given in Mac Lane \cite[Chapter 10]{MacLaneWorking} (see also notes by Riehl \cite{RiehlHomotopyNotes}), which we briefly outline. Then in Definition \ref{DefnInd} below we will define induction of $\FIW$--modules using a concrete description of these constructions as they apply to the categories $\FIW$. 

Given a subcategory $\FIW \subseteq \FI_{\oW}$,  and an $\FIW$--module $V$, we denote by $\Ind_{\W}^{\oW} V$ the \emph{left Kan extension} of $V$ along the inclusion of categories. 

$$\xymatrix{ \FIW  \ar[r]^V  \ar@{^{(}->}[d] & k\sMod \\
 \oFIW \ar@{-->}[ur]_{\Ind_{\W}^{\oW} V} & }$$
 
 The induction map $$\Ind_{\W}^{\oW}: \FIW\text{-Mod} \longrightarrow \oFIW\text{--Mod}$$ is functorial on the functor category of $\FIW$--modules. In particular, given two $\FIW$--modules $V$ and $W$ and a map of $\FIW$--modules $F: V \to W$, there is a corresponding map of $\oFIW$--modules

$$\Ind_{\W}^{\oW}\; F: \Ind_{\W}^{\oW}\; V \longrightarrow \Ind_{\W}^{\oW}\; W;$$
assigned in a manner that respects composition of $\FIW$--module maps.

The functor $\Ind_{\W}^{\oW}$ is the left adjoint to $\Res_{\W}^{\oW}$, and satisfies the associated properties recognizable from the familiar adjunction for induction and restriction of group representations.  For any $\FIW$--module $V$,  there is a canonical map of $\FIW$--modules $$\eta_{V}: V \to \Res_{\W}^{\oW} ( \Ind_{\W}^{\oW} \; V )$$ defined by the \emph{unit} map, the natural transformation 
$ \eta: id \to (\Res_{\W}^{\oW}  \Ind_{\W}^{\oW} ).$

Given any $\oFIW$--module $U$ and $\FIW$--module map $V \to \Res_{\W}^{\oW} \; U$, there exists a unique map of $\oFIW$--modules $\a: \Ind_{\W}^{\oW} V \to U$ such that the following diagram commutes.

$$\xymatrix{ & \Res_{\W}^{\oW} \; (\Ind_{\W}^{\oW}\; V) \ar[dr]^{\Res_{\W}^{\oW} \a} &  \\
 V \ar[ru]^{\eta} \ar[rr] && \Res_{\W}^{\oW} U }$$

\noindent This correspondence defines a bijection
\begin{align*} \left\{ \begin{array}{c} \text{ $\FIW$\;--Module Maps } \\ V  \longrightarrow \Res_{\W}^{\oW} \; U \end{array} \right\} \quad \longleftrightarrow \quad \left\{ \begin{array}{c} \text{ $\oFIW$--Module Maps} \\ \Ind_{\W}^{\oW} \;V  \longrightarrow U  \end{array} \right\} 
\end{align*}
which is natural in the inputs $V$ and $U$. 

We can describe the induced functor explicitly. Following Mac Lane \cite[Chapter 10.4]{MacLaneWorking}, we define the $\oW_n$--representation $(\Ind_{\W}^{\oW} V)_n$ as a certain \emph{coend}, the coequalizer of two maps $\phi$ and $\psi$. 

$$ \xymatrix{
   \displaystyle  \bigoplus_{p \leq q \leq n} M_{\oW}({\bf q})_n \otimes_k M_{\W}({\bf p})_q \otimes_k V_p     { \ar@<-1ex>[r]_-{\psi}} { \ar@<1ex>[r]^-{\phi}} &  \displaystyle \bigoplus_{r \leq n} M_{\oW}({\bf r})_n \otimes_k V_r \ar[r] & (\Ind_{\W}^{\oW} V)_n \\
  f \otimes g \otimes v  \ar@{|->}[r]^{\phi} & f \otimes g_*(v) \\ 
 f \otimes g \otimes v  \ar@{|->}[r]^{\psi} & f \circ g \otimes v }
$$
 
In parallel with the $k$--modules $\Ind_H^G V := k[G] \otimes_{k[H]} V$, the induced functor $\Ind^{\oW}_{\W} V$ is sometimes called a \emph{tensor product of functors over a category} and written $\FI_{\oW} \otimes_{\FIW} V$. We summarize its construction in the following definition.

\begin{defn}{\bf (Induction).} \label{DefnInd} Given an $\FIW$--module $V$, and an inclusion of categories $\FIW \hookrightarrow \FI_{\oW}$, we define the \emph{induced $\FI_{\oW}$--module} $\Ind_{\W}^{\oW} V$ by 
$$ (\Ind_{\W}^{\oW} V)_n = \bigoplus_{r \leq n} M_{\oW}({\bf r})_n \otimes_k V_r \bigg/  \langle \quad f\otimes g_*(v) = (f\circ g) \otimes v \quad \vert \quad \text{ $g$ is an $\FIW$ morphism} \rangle .$$
with the action of $h \in \Hom_{\oW}(\bm, \bn)$ by $h_*: g\otimes v \longmapsto (h\circ g) \otimes v.$

We emphasize that induction is left adjoint to restriction, and satisfies the naturality properties described above. 
\end{defn}

We encourage the reader to verify that $(\Ind_{\W}^{\oW} V)_n $ is in fact a quotient of the induced representation $\Ind_{\W_n}^{\oW_n} (V_n)$ by relations which require the stabilizer $H_{\ell,n} = \Stab(I_{\ell,n})$ to act trivially on the image of $(\Ind_{\W}^{\oW} V)_{\ell} $ in $(\Ind_{\W}^{\oW} V)_n $, and so ensure an $\FIW$--module structure (as in Lemma \ref{PromotionToFI}.)



\begin{prop}\label{IndMisM}{\bf ($\Ind_{\W}^{\oW} M_{\W}(\bm) \cong M_{\oW}(\bm)$).} Given categories $\FIW \subseteq \oFIW$ and any integer $m$, there is an isomorphism of $\oFIW$\,--modules $$\Ind_{\W}^{\oW} M_{\W}(\bm) \cong M_{\oW}(\bm).$$
In other words, the functor $\Ind_{\W}^{\oW}$ preserves represented functors. 
\end{prop}

\begin{proof}[Proof of Proposition \ref{IndMisM}] It is straightforward to verify that the map
\begin{align*} 
\bigoplus_{m \leq r \leq n} M_{\oW}({\bf r})_n \otimes_k M_{\W}(\bm)_r & \longrightarrow M_{\oW}(\bm)_n \\
 g \otimes f & \longmapsto g \circ f  \qquad \text{with} \quad f \in \Hom_{\W}(\bm, {\bf r}), \quad g \in \Hom_{\oW}({\bf r}, \bn) 
\end{align*}
factors through an isomorphism of $\FI_{\oW}$--modules $$(\Ind_{\W}^{\oW} \; M_{\W}(\bm) )_n \xrightarrow{\cong} M_{\oW}(\bm)_n . \qedhere $$
\end{proof}

\begin{cor} \label{PromoteMtoInd}
Given an $\FIW$--module finitely generated in degree $\leq m$, the natural surjection of $\FIW$--modules of Proposition \ref{FinGen} $$S: \bigoplus^m_a M_{\W}(\ba)^{b_a} \longrightarrow  V$$ can be promoted to a surjection of $\FI_{\oW}$--modules $$(\Ind_{\oW}^{\W} \; S) : \bigoplus_a^m  M_{\oW}(\ba)^{b_a} \longrightarrow \Ind_{\W}^{\oW} V.$$
\end{cor}

\begin{cor}\label{IndGenRelDegrees} {\bf(Induction respects generation and relation degree).}
Suppose that $V$ is a finitely generated $\FIW$--module with degree of generation $\leq g$ and relation degree $\leq r$. Then $\Ind_{\W}^{\oW} V$ is also finitely generated, with generation degree $\leq g$ and has relation degree $\leq r$. \end{cor}

\begin{proof}[Proof of Corollary \ref{IndGenRelDegrees}] Suppose $V$ is an $\FIW$--module with generation degree $\leq g$ and relation degree $\leq r$, that is, there exists an exact sequence of $\FIW$--modules:
$$ \bigoplus_{m=1}^r M_{\W}(\bm)^{\oplus p_m} \longrightarrow  \bigoplus_{m=1}^g M_{\W}(\bm)^{\oplus \ell_m} \twoheadrightarrow V. $$
The functor  $ \Ind_{\W}^{\oW} \; $ is a left adjoint and therefore it is right exact (see Weibel \cite[Theorem 2.6.1]{WeibelIntro}). Applying this functor to the above sequence yields the exact sequence 
$$ \bigoplus_{m=1}^r M_{\oW}(\bm)^{\oplus p_m} \longrightarrow  \bigoplus_{m=1}^g M_{\oW}(\bm)^{\oplus \ell_m} \twoheadrightarrow \Ind_{\W}^{\oW} V,  $$
and we conclude that $g$ and $r$ are upper bounds on the generation and relation degree of  $ \Ind_{\W}^{\oW} V$.

This result can also be verified concretely: If $V$ is generated by a finite set of elements $v_m \in V_m$, then it is easily seen that $\Ind_{\W}^{\oW} \; V$ is generated by the images of the elements $$\text{id}_m \otimes v_m \in  M_{\oW}(\bm)_m \otimes_k V_m \qquad \text{ in $(\Ind_{\W}^{\oW} \; V)_m$.}  \qedhere$$
\end{proof}

\begin{prop}\label{VintoResIndV} {\bf( $V \hookrightarrow (\Res_{\W}^{\oW} \; \Ind_{\W}^{\oW} \; V)$).} Given an $\FIW$--module $V$ and an inclusion of categories $\FIW \hookrightarrow \oFIW$, the natural map of $\FIW$--modules 
\begin{align*}
V & \longrightarrow \Res_{\W}^{\oW} \; \Ind_{\W}^{\oW} \; V  \\
v \in V_n & \longmapsto id_n \otimes v \qquad \in M_{\oW}(\bn)_n \otimes V_n
\end{align*}
 is injective.
\end{prop}

It should not be surprising that the map $V \longrightarrow (\Res_{\W}^{\oW} \; \Ind_{\W}^{\oW} \; V)$ is injective since, heuristically, the relations defining the quotient $ (\Ind_{\W}^{\oW} \; V)_n$ come from relations already imposed on $V$ by its $\FIW$--module structure. 

It seems that there ought to be a proof of Proposition \ref{VintoResIndV} -- injectivity of the unit map $\eta_V$ -- using formal properties of the adjunction, and the fact that injectivity holds for the represented functors $M_{\W}(\bm)$. We have found this formal proof to be elusive, however. A proof specific to the three categories $\FIW$ is given below. 


\begin{proof}[Proof of Proposition \ref{VintoResIndV}] 

We first address the case where $\W_n$ is the symmetric group $S_n$, and $\oW_n$ is $D_n$ or $B_n$. To prove the proposition, it suffices to show that the underlying maps of $k$--modules $$V_n \longrightarrow \Big(\Res_{A}^{\oW} \; \Ind_{A}^{\oW} \; V \Big)_n$$ are injective; to do so we will construct left inverses $\widetilde{L}_n$ of the $k$--module maps. 

Fix $n$. We define a map of $k$--modules $$ L: \bigoplus_{r \leq n} M_{\oW}({\bf r})_n \otimes_k V_r \longrightarrow V_n$$ as follows: For any pure tensor of the form $$g \otimes v \in M_{\oW}({\bf r})_n \otimes_k V_r \qquad \text{with $\FI_{\oW}$ morphism } g: {\bf r} \to \bn,$$ we can factor $g$ as $g = \s \circ \tilde{g}$, where $$\s  \in (\Z /2 \Z)^n \in B_n \cong (\Z /2 \Z)^n \rtimes S_n,$$ and $\tilde{g}$ is a (uniquely determined) $\FI_A$--morphism. We assign 
\begin{align*}
 L: \bigoplus_{r \leq n} M_{\oW}({\bf r})_n \otimes_k V_r &\longrightarrow V_n \\
 g\otimes v & \longmapsto \tilde{g}(v) 
\end{align*}
Because the assignment $g \longmapsto \tilde{g}$ respects composition, $$L(f\circ h \otimes v) = L(f \otimes h(v)) \qquad \text{for all $\FI_A$ morphisms $h$.}$$ Hence, $L$ factors through the quotient $(\Ind_{A}^{\oW} \; V)_n$

$$ L: \bigoplus_{r \leq n} M_{\oW}({\bf r})_n \otimes_k V_r \twoheadrightarrow (\Ind_{A}^{\oW} \; V)_n \xrightarrow{\widetilde{L\;}} V_n.$$

The composite map of $k$--modules

\begin{align*} \begin{matrix}
 V_n  & \longrightarrow & (\Ind_{A}^{\oW} \; V)_n = (\Res_{A}^{\oW} \; \Ind_{A}^{\oW} \; V)_n & \xrightarrow{\widetilde{L\;}} &  V_n \\
 v & \longmapsto & id_n \otimes v & \longmapsto & id_n(v)=v
 \end{matrix}
\end{align*}
\noindent is the identity, which implies that the natural map of $\FI_A$--modules $$ V_n \longrightarrow (\Res_{A}^{\oW} \; \Ind_{A}^{\oW} \; V)_n $$ is injective.

Next we address the induction of $\FI_D$--modules up to $\FI_{BC}$--modules.  We will use the same outline as in the first case, but there are additional subtleties: unlike with $S_n$,  the group $D_n$ is not a quotient of $B_n$, and there is no way to associate an $\FI_D$ morphism $\tilde{g}$ to each $\FI_{BC}$ morphism $g$ in a manner that respects composition. 

We will, however, still define a left inverse $\widetilde{L\;}$ as above. Again, for each fixed $n$, we define a map of $k$--modules
$$ L: \bigoplus_{r \leq n} M_{BC}({\bf r})_n \otimes_k V_r \longrightarrow V_n$$ 
on the pure tensors $g \otimes v$ with $g: {\bf r} \to \bn$ and  $v \in V_r$, as follows. 

If $r \neq n$, then $\Hom_{\FI_{BC}}({\bf r}, \bn) \cong \Hom_{\FI_{D}}({\bf r}, \bn)$ by Remark \ref{HomFID=HomFIBC}, and so $g(v)$ is a well-defined element of $V_n$. In this case we define
$$L: g \otimes v \longmapsto g(v) \in V_n.$$

Similarly, suppose $g \in \End_{\FI_{BC}}(\bn)$ but $v$ is in the image of $V_{r}$ for some $r<n$, say, $v=f(u)$ for some $\FI_D$--morphism $f: {\bf r} \to \bn$. Then $$g\circ f \in \Hom_{\FI_{BC}}({\bf r}, \bn) = \Hom_{\FI_{D}}({\bf r}, \bn),$$ and $(g \circ f)(u)$ is a well-defined element of $V_n$. In this case we define
$$L: g \otimes f(u) \longmapsto (g \circ f)(u) \in V_n.$$
Both assignments satisfy $$L(g\circ h \otimes v) = L(g \otimes h(v)) \qquad \text{for all $\FI_D$--morphisms $h$.}$$

Finally, suppose $g \otimes v$ is a pure tensor with $$g \in B_n \cong \End_{\FI_{BC}}(\bn) \quad \text{and} \quad v \in V_n \text{ such that } v \notin \text{Span}_{V}(V_{n-1}).$$ Since $D_n \subseteq B_n$ has index two, either $g \in D_n,$ or $(-n \; n)\circ g \in D_n.$

We define
\begin{align*}
 L( g\otimes v) =  &  \left\{ \begin{array}{ll}
         g(v) & \mbox{if $g \in D_n$ },\\
         \big( (-n \; n)\circ g\big) (v) & \mbox{if $g \notin D_n$}.\end{array} \right.
\end{align*}
In this case, too, $$L(g \otimes h(u))=L(g\circ h \otimes u) \qquad \text{ for all $\FI_D$--morphisms $h$:}$$ since, by assumption on $v$, we can write $v=h(u)$ only if $h$ is an element of $\End_{\FI_D}(\bn) \cong D_n$, and so $g \in D_n$ if and only if $g\circ h \in D_n$. 

Once again, $L$ will factor through the quotient $(\Ind_D^{BC} V)_n$, and gives the desired left inverse. The map $V \to \Res_D^{BC} \; \Ind_D^{BC} \; V$ is injective, as desired. 
\end{proof}

Having established Proposition \ref{VintoResIndV}, we can now prove a critical fact about finitely generated $\FI_D$--modules. 

\begin{prop} \label{VisResIndV} {\bf( $V_n \cong (\Res_{D}^{BC}\Ind_{D}^{BC} V)_n$ for $n$ large).} Suppose $V$ is an $\FI_D$--module finitely generated in degree $\leq m$. Then $$ V_n \xrightarrow{\cong} (\Res_{D}^{BC}\; \Ind_{D}^{BC} V)_n$$ is an isomorphism of $D_n$--representations for all $n>m$. In particular, every finitely generated $\FI_D$--module $V$ is, for $n$ greater than its degree of generation, the restriction of an $\FI_{BC}$--module. 
\end{prop}

\begin{proof}[Proof of Proposition \ref{VisResIndV}] The map $$V_n \to (\Res_{D}^{BC}\; \Ind_{D}^{BC} V)_n$$ is injective by Proposition \ref{VintoResIndV}, and so it suffices to show that this map is surjective for $n>m$.  Since $V$ is finitely generated in degree $\leq m$, by Proposition \ref{FinGen} we have a surjection of $\FI_D$--modules $$ S: \bigoplus_{a=0}^{m} M_D(\ba)^{\oplus b_a} \twoheadrightarrow V.$$ Inducing both sides up to $\FI_{BC}$ gives a surjective map

$$ \bigoplus_{a=0}^{m} M_{BC}(\ba)^{\oplus b_a} = \Ind_{D}^{BC} \; \bigg(\bigoplus_{a=0}^{m} M_D(\ba)^{\oplus b_a} \bigg)   \xrightarrow{\Ind_{D}^{BC} \; S} \Ind_{D}^{BC} \;  V $$

\noindent where the first equality follows from Proposition \ref{IndMisM}. By naturality of the unit map $\eta$, these maps fit together into a commutative diagram
\begin{align*}
\xymatrix{
 & \displaystyle \bigoplus_{a=0}^{m} M_D(\ba)^{\oplus b_a} \ar[d] \ar[rrr]^{S} &&& V \ar[d] \\
& \displaystyle \bigoplus_{a=0}^{m}  \Res_D^{BC} M_{BC}(\ba)^{\oplus b_a} \ar[rrr]^{\Res_D^{BC} \; \Ind_D^{BC} \; S} &&& \Res_D^{BC} \; \Ind_D^{BC} \; V 
}
\end{align*}

By Remark \ref{RemarkMDisMBC}, the left vertical arrow $$ \bigoplus_{a=0}^{m} M_D(\ba)_n^{\oplus b_a} \longrightarrow \bigoplus_{a=0}^{m} \Res_{D}^{BC} M_{BC}(\ba)_n^{\oplus b_a} $$ is an isomorphism of $D_n$--representations for $n>m$, and so the composite
$$ \bigoplus_{a=0}^{m} M_D(\ba)_n^{\oplus b_a} \xrightarrow{\cong} \bigoplus_{a=0}^{m}  \Res_D^{BC} M_{BC}(\ba)_n^{\oplus b_a} \twoheadrightarrow (\Res_D^{BC} \; \Ind_D^{BC} \; V )_n $$ is surjective for $n>m$. By commutativity, the right vertical arrow $$V_n \longrightarrow (\Res_D^{BC} \; \Ind_D^{BC} \; V)_n $$ must also surject for these values of $n$, which proves the claim.
\end{proof}

\begin{rem} 
Let $V$ be an $\FI_D$--module. In Proposition \ref{VisResIndV} we proved the isomorphism of $D_n$--representations $$ V_n \xrightarrow{\cong} (\Res_{D}^{BC}\; \Ind_{D}^{BC} V)_n \qquad \text{for $n>m$.}$$ We note that the analogous statements about $\Ind_{A}^{D}$ and $\Ind_{A}^{BC}$ are false. This is apparent from the $\FI_A$--modules $M_{A}(\bm)$ with $m>0$. We have
$$ \rank_{k} \; M_{A}(\bm)_n = [S_n : S_{n-m}] = \frac{n!}{(n-m)!} $$
In contrast, by Proposition \ref{IndMisM}, we have $ \Ind_{A}^{D} M_A(\bm) = M_D(\bm)$ and $ \Ind_{A}^{BC} M_A(\bm) = M_{BC}(\bm),$ with
\begin{align*}
 & \rank_{k} \; M_{D}(\bm)_n = [D_n : D_{n-m}] = \left\{ \begin{array}{c} 2^{m-1} n! \qquad  n=m \\ \displaystyle \frac{2^m \, n!}{(n-m)!} \qquad n>m  \end{array} \right. \\& \\ 
 & \rank_{k} \; M_{BC}(\bm)_n = [B_n : B_{n-m}] =  \frac{2^m \, n!}{(n-m)!}. 
\end{align*}
We see that $M_A(\bm)_n$ is a proper sub--$k[S_n]$--module of $ \Res_{A}^{D} \; \Ind_{A}^{D} M_A(\bm)_n$ and $ \Res_{A}^{BC} \Ind_{A}^{BC} M_A(\bm)_n$ for all $n$.
\end{rem}

\section{Constraints on finitely generated $\FIW$--modules}

Church--Ellenberg--Farb \cite{CEF} relate finite generation of an $\FI_A$--module to certain constraints on the shape of the Young diagrams in the irreducible representations of each representation $V_n$. We develop analogous results for $\FI_D$ and $\FI_{BC}$.

\subsection{The weight of an $\FIW$--module} \label{Section:Weight}

\begin{defn} \label{Defn:Weight} {\bf (Weight).} Let $k$ be a field of characteristic zero. Church--Ellenberg--Farb \cite[Definition 2.50]{CEF} define the \emph{weight of an $\FI_A$--module} to be $\leq d$ if for every $n \geq 0$ and every irreducible constituent $V(\y)_n$ of $V_n$ has $|\y| \leq d$ (in the notation described in Section \ref{BackgroundRepStability}). 

Similarly, we define the \emph{weight of a $B_n$--representation} $V_n$ to be $\leq d$ if every irreducible representation $V(\y)_n = V(\y^+, \y^-)_n$ in $V_n$ satisfies $|\y^+| + |\y^-| \leq d.$ We define the \emph{weight of an $\FI_{BC}$--module} $V$ to be $\leq d$ if $V_n$ has weight $\leq d$ for each $n$.  We define the \emph{weight of an $\FI_{D}$--module} $V$ as the weight of the $\FI_{BC}$--module $\Ind_{D}^{BC} V$. 

An $\FIW$--module $V$ has \emph{finite weight} if it is of weight $\leq d$ for some $d \geq 0$, and we call the minimum such $d$ the weight of $V$, weight$(V)$. We say that the \emph{weight of a Young diagram $\y = (\y_0, \y_1, \ldots, \y_\ell)$} is $$\y_1 + \cdots + \y_{\ell} = |\y|-\y_0.$$ 
\end{defn}
Over characteristic zero, the weight of submodules and quotients of $V$ are at most weight$(V)$.

\begin{prop}\label{WeightM(m)}
 Let $k$ be a field of characteristic zero. Then (in the notation of Section \ref{BackgroundRepStability}), a $\W_n$--representation $V(\y)_n$ is contained in $M_{\W}(\bm)$ if and only if $|\y| \leq m$. In type $D$, $M_{D}(\bm)$ decomposes completely into representations of the form $V(\y)_n$.
\end{prop}
In type A and B/C, it is immediate that all representations over characteristic zero decompose into irreducible representations of the form $V(\y)_n$. In type D, this is precisely the statement that all 'split' irreducible representations occur in pairs $V\{\mu, +\} \oplus V\{\mu, -\}$. 

\begin{proof}[Proof of Proposition \ref{WeightM(m)}]
 The branching rules for the symmetric groups implies that $$V(\y)_n \quad \text{ occurs in } M_{A}(\bm)_n \cong \Ind_{S_{n-m}}^{S_n} k \qquad$$ if and only if  $\y[n]$ can be built from the partition $(n-m)$ by adding one box at a time;  these are exactly those diagrams $\y[n]$ with largest part $n-|\y| \geq (n-m)$, equivalently, with $ |\y| \leq m$.

 Similarly, by the branching rules for the hyperoctahedral group (Equation (\ref{Eqn:Pieri}), Section \ref{RepTheoryWn}), $$V(\y)_n = V(\y^+, \y^-)_n \text{ appears in $M_{BC}(\bm)_n$}$$ precisely when $(\y^+, \y^-)_n$ contains the double partition $((n-m), \varnothing)$, that is, when the largest part of $\y^+$, $n-(|\y^+|+|\y^-|)$, is at least $(n-m)$. We conclude that $V(\y)_n$ is contained in $M_{BC}(\bm)$ if and only if $|\y^+| + |\y^-| \leq m$.

Finally, in type D, by Remark \ref{RemarkMDisMBC} we have $$ M_D(\bm)_n = \left\{ \begin{array}{l}
         k[D_m]  \qquad \mbox{if $n=m$},\\
         \Res^{BC}_D M_{BC}(\bm)_n  \qquad \mbox{if $n>m$}, \end{array} \right. $$

When $n=m$, all $D_n$--representations $V(\y)_m$ necessarily satisfy $|\y|<m$, and conversely every irreducible subrepresentation appears in the regular representation $k[D_m]$ with multiplicity equal to its dimension. The split representations $V_{\{\y^-, +\}}$ and $V_{\{\y^-, -\}}$, being of equal dimension, occur in pairs. For $n>m$, the result follows immediately from the identification  $M_D(\bm)_n \cong \Res^{BC}_D M_{BC}(\bm)_n$ and the result in type B/C.
\end{proof}

Theorem \ref{WeightM(m)} (with Proposition \ref{IndMisM} in type D) imply:

\begin{cor} \label{M(m)isWeightm} The $\FIW$--module $M_{\W}(\bm)$ has weight $m$.
\end{cor}

\begin{thm} \label{WnDiagramSizes}{\bf (Degree of generation bounds weight).}  Suppose that $V$ is an $\FIW$--module over a field of characteristic zero.
  If $V$ is finitely generated in degree $\leq m$, then weight($V$) $\leq m$.
\end{thm}
\begin{proof}[Proof of Theorem \ref{WnDiagramSizes}]
 
By Proposition \ref{FinGen}, any $\FIW$--module $V$ finitely generated in degree $\leq m$ is a quotient of some $\FIW$--module of the form $ \bigoplus_{a=0}^m M_{\W}(\ba)^{b_a}$. Therefore, we conclude Theorem \ref{WnDiagramSizes} from Proposition \ref{WeightM(m)} and (for $\FI_D$--modules) Corollary \ref{PromoteMtoInd}.

Theorem \ref{WnDiagramSizes} is proven for $\FI_A$--modules in \cite[Proposition 2.51]{CEF}. 
\end{proof}

Theorem \ref{WnDiagramSizes} strongly constrains which irreducible representations can occur in $V_n$ once $n$ is large relative to the degree of generation of $V$. The following corollary gives some examples of irreducible components which are excluded. 

\begin{cor}\label{NoSign} Suppose that $V$ is an $\FIW$--module over a characteristic zero field, generated in degree $\leq m$. 
\begin{itemize} 
 \item If $\W_n$ is $S_n$, then for all $n>(m+1)$ the $S_n$--representation $V_n$ cannot contain the alternating representation. 
 \item If $\W_n$ is $D_n$ or $B_n$, then for all $n>(m+1)$ the $\W_n$--representation $V_n$ cannot contain the pullback of the alternating representation.
 \item  If $\W_n$ is $B_n$, then for all $n>m$ the $B_n$--representation $V_n$ cannot contain the 'sign' representation associated to the character $$\varepsilon: B_n \twoheadrightarrow B_n/D_n \cong \{\pm 1\}.$$
\end{itemize}
\end{cor}

\begin{proof}[Proof of Corollary \ref{NoSign}]
 If $V$ is an $\FIW$--module generated in degree $\leq m$ as above, then weight$(V)$ is at most $m$ by Theorem \ref{WnDiagramSizes}. The alternating $S_n$--representation $V(1, 1, \ldots, 1)$ has weight $(n-1)$, as does its pullback to $B_n$, $V((1, 1, \ldots, 1), \varnothing),$ so neither representation can occur in $V_n$ once $n>(m+1)$. 
 
 The $B_n$--representation $V(\varnothing, (n))$ has weight $n$, so it can occur in $V_n$ only when $n\leq m$. 
 
 The alternating $S_n$--representation pulls back to the $D_n$--representation $V\{(1, 1, \ldots, 1), \varnothing\}.$ Suppose there existed an $\FI_D$--module $V$ wherein this pullback occurred in $V_n$ for some $n>(m+1)$.  By the classification of $D_n$--representations described in Section \ref{RepTheoryDn}, this pullback $D_n$--representation is the restriction of either the $B_n$--representation $V((1, 1, \ldots, 1), \varnothing)$ of weight $(n-1)$ or the $B_n$--representation $V(\varnothing, (1, 1, \ldots, 1))$ of weight $n$; it does not occur in the restriction of any other $B_n$--representation. By Proposition \ref{VisResIndV}, the $\FI_{BC}$--module $\Ind_{D}^{BC} V$ must contain one of these two representations in degree $n$, contradicting Theorem \ref{WnDiagramSizes}. 
\end{proof}

\begin{prop} {\bf(Split representations do not occur in finitely generated $\FI_D$--modules for $n>2m$).} \label{NoSplitIrreps} Let $k$ be a field of characteristic zero, and suppose that $V$ is an $\FI_D$--module over $k$ finitely generated in degree $\leq m$. Then for any $n>2m$, the $D_n$--representation $V_n$ does not contain any 'split' irreducible representations, that is, all of its irreducible components are of the form $V_{\{ \y,\; \mu\}}$ for $\y \neq \mu$.
\end{prop}

\begin{proof}[Proof of Proposition \ref{NoSplitIrreps}] There is a surjection of $\FI_D$--modules $ \bigoplus_{a=0}^m M_{D}(\ba)^{b_a} \twoheadrightarrow V$ by Proposition \ref{FinGen}, so every irreducible component of $V_n$ must appear in $M_D(\ba)_n$ for some $a \leq m$. Moreover, by Remark \ref{RemarkMDisMBC} we have an isomorphism of $D_n$--representations $M_D(\ba)_n = \Res_{D_n}^{B_n} M_{BC}(\ba)_n,$ and so every irreducible component of $V_n$ must appear in $\Res_{D_n}^{B_n}M_{BC}(\ba)_n$ for some $a \leq m$. 

The branching rules (Equation (\ref{Eqn:Pieri}), Section \ref{RepTheoryWn}) imply that the irreducible representation $V_{(\y, \mu)} \subseteq M_{BC}(\ba)_n$ only if $$((n-m), \varnothing) \subseteq ((n-a), \varnothing) \subseteq (\y, \mu),$$ and so in particular $$|\y| \geq (n-m) > m \geq |\mu| \qquad \text{ for all $n > 2m$.}$$ Thus, $V_{(\y, \mu)} \subseteq M_{BC}(\ba)_n$ only if $|\y| \neq |\mu|$, and so by restriction to $D_n$ we conclude that when $n > 2m$, $V_n$ only contains irreducible components of the form $V_{\{ \y,\; \mu\}}$ for $\y \neq \mu$.
\end{proof}

\subsection{Coinvariants and stability degree} \label{Section:CoinvariantsAndStabilityDeg}

{ \noindent \bf  Shifted $\FIW$--modules.  \quad} The category $\FIW$ contains isomorphic copies of itself as proper subcategories. We use these inclusions to define the shifting operation on $\FIW$--modules. As in \cite[Section 2.4]{CEF}, by shifting and passing to coinvariants, we define the \emph{stability degree} of an $\FIW$--module, and, in Lemma \ref{FinGenImpliesStabilityDeg}, we find a lower bound on the stability degree of a finitely presented $\FIW$--module. In Section \ref{SectionFinGenRepStability}, we will use this concept to prove the equivalence of finite generation of an $\FI_{BC}$--module with representation stability in the sense of Church--Farb \cite{RepStability}. 

We first introduce some notation: given nonnegative integers $n$ and $a$, we denote the $\FIW$ object
$$ {\bf (n+a)} := \{ \pm 1, \pm 2,  \ldots \pm (n+a) \}.$$
We remark that addition here corresponds to disjoint union of finite sets; $ {\bf (n+a)} \cong \bn \sqcup \ba $. 

\begin{defn} {\bf (Shifts $\amalg_{[-a]} : \FIW \to \FIW$)}
For each $a \geq 0$, there are maps 
\begin{align*} 
 \bn = \{ \pm 1, \ldots, \pm n \} & \hookrightarrow \{ \pm 1, \ldots, \pm (n+a) \} = ({\bf n+a}) \\
 d & \longmapsto (d+a)
\end{align*}
These maps define functors 
\begin{align*}
\amalg_{[-a]} : \FIW & \longrightarrow \FIW \\
\bn & \longmapsto {\bf (n+a) } \\ 
\{ f: \bm \to \bn \} & \longmapsto \{ \amalg_{[-a]}(f) :{\bf (m+a)} \to {\bf (n+a) } \}
\end{align*}
where
\begin{align*}
 \amalg_{[-a]}(f) \text{ maps } \left\{ \begin{array}{l}
         d \mapsto d  \qquad \mbox{if $d \leq a$},\\
         (d+a) \mapsto (f(d)+a).  \end{array} \right.
\end{align*}  
\end{defn}

\noindent Figure \ref{fig:Shifted} gives a schematic of the functor $\amalg_{[-2]}$.

\begin{figure}[h!]
\begin{center}
\setlength\fboxsep{5pt}
\setlength\fboxrule{0.5pt}
\fbox{ \includegraphics[scale=3.1]{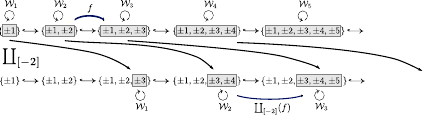} }
\caption{ {\small The functor $\amalg_{[-a]}(f): \FIW \to \FIW$}} 
\label{fig:Shifted}
\end{center}
\end{figure}

\begin{defn}{\bf (Shifted $\FIW$--modules).}
 Given an $\FIW$--module $V: \FIW \to k$--Mod, we define the \emph{shifted} $\FIW$--module $S_{+a}V$ by $S_{+a}V = V \circ \amalg_{[-a]}.$
\end{defn}

The $\W_n$--representation $(S_{+a}V)_n$ is the restriction of $V_{n+a}$ to the copy of $\W_n$ acting on $\{\pm(1+a), \ldots \pm(n+a)\} \subseteq {\bf (n+a) }$. 

\begin{defn} {\bf (Coinvariants functor $\t$).} We define $\text{$\t: \FIW-$Mod $ \to \FIW-$Mod}$ to be the coinvariants functor, as follows: for an $\FIW$--module $V$, let $\t V$ be the $\FIW$--module with
$$ (\t V)_n = (V_n)_{\W_n} := k \otimes_{k[\W_n]} V_n $$ 
\end{defn}

That is, $(\t V)_n$ is the largest quotient of $V_n$ on which $\W_n$ acts trivially. When $k$ is a field of characteristic zero, the map $V_n \twoheadrightarrow (V_n)_{\W_n}$ is the projection onto the invariant subspace $(V_n)^{\W_n}$.

\begin{defn} {\bf (The graded $k[T]$--module $\Phi_a(V)$).} Fix an integer $a \geq 0$. We define
\begin{align*}
  \Phi_a : \FIW \text{--Mod} & \longrightarrow k[T]\text{--Mod} \\ 
        V & \longmapsto \bigoplus_{n \geq 0} (\t \circ S_{+a} V)_n  = \bigoplus_{n \geq 0} (V_{n+a})_{\W_n}     
\end{align*}
The action of $T$ is by the maps  $(V_{n+a})_{\W_n} \to (V_{n+1+a})_{\W_{n+1}}$ induced by the maps $(I_{n+a})_* : V_{n+a} \to V_{n+1+a}$.  
\end{defn}

\begin{rem} \label{CoinvariantsBranchingRule}
 Note that each graded piece $\Phi_a(V)_n = (V_{n+a})_{\W_n}$ has the structure of an $\W_a$--module, and $T$ acts $\W_a$--equivariantly. Over characteristic zero, the multiplicity of a $\W_a$--representation $U$ in $(V_{n+a})_{\W_n}$ is equal to the multiplicity of $U \boxtimes k$ in the restriction $\Res^{\W_{n+a}}_{\W_a \times \W_n} V_{n+a}$, given by the branching rules (Equation (\ref{Eqn:WnInvariantsDecomp}), Section \ref{RepTheoryWn}.) \end{rem}

\begin{defn} {\bf (Injectivity degree; Surjectivity degree).} An $\FIW$--module has \emph{injectivity degree $ \leq s$} (respectively, \emph{surjectivity degree $ \leq s$}) if for every $a \geq 0$ and for all $n \geq s$, the map $\Phi_a(V)_n \to \Phi_a(V)_{n+1}$ induced by $T$ is injective (respectively, surjective). The minimum such $s$ is called the \emph{injectivity degree} (respectively, the \emph{surjectivity degree}).
\end{defn}

\begin{defn} \label{Defn:StabDeg} {\bf (Stability degree).} An $\FIW$--module has \emph{stability degree $\leq s$} if for every $a \geq 0$ and $n \geq s$, the map $\Phi_a(V)_n \to \Phi_a(V)_{n+1}$ induced by $T$ is an isomorphism of vector spaces; equivalently, an isomorphism of $\W_a$--representations. The \emph{stability degree} is the minimum such $s$; it is the maximum of injectivity and surjectivity degree.
\end{defn}

Explicitly, $V$ has stability degree $\leq s$ if $(V_{n+a})_{\W_{n}} \cong (V_{n+1+a})_{\W_{n+1}}$for every $a \geq 0$ and $n \geq s$.

\noindent Figure \ref{fig:StabilityDegree} shows an $\FIW$--module $V$ with stability degree $3$. 

\begin{figure}[h!]
\begin{center}
\setlength\fboxsep{5pt}
\setlength\fboxrule{0.5pt}
\fbox{ \includegraphics[scale=2]{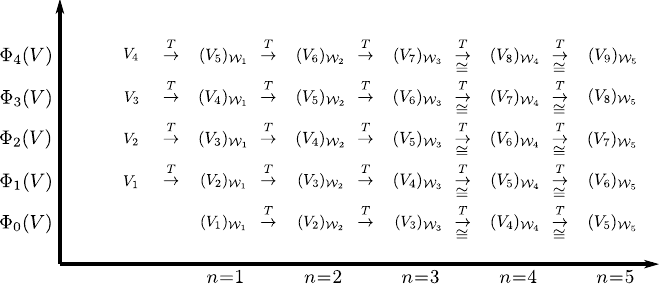}}
\caption{Stability degree $3$}
\label{fig:StabilityDegree} 
\end{center}
\end{figure}

The following results (Remark \ref{StabilityToQuotients}, Propositions \ref{StabilityDegreeM(m)}, \ref{FiniteGenerationSurjectivityDegree}, and  \ref{StabilityKernelCokernel}, and Lemma \ref{FinGenImpliesStabilityDeg})  are proven by Church--Ellenberg--Farb \cite[Section 2.4]{CEF} for $\FI_A$--modules, and their proofs generalize readily to $\FI_{BC}$ and $\FI_D$.

\begin{rem} \label{StabilityToQuotients}
 Let $V$ be an $\FIW$--module with surjectivity degree $\leq s$ and injectivity degree $\leq t$. Since $\Phi_a$ is right exact, quotients of $V$ also have surjectivity degree $\leq s$. Furthermore, when $k$ contains $\Q$, $\Phi_a$ is left exact, and any submodule of $V$ has injectivity degree $\leq t$. 
\end{rem}

Remark \ref{StabilityToQuotients} generalizes \cite[Remark 2.36]{CEF}. The following proposition is proven in \cite[Proposition 2.37]{CEF} for $\FI_A$. We adapt their proof to $\FI_{BC}$ and $\FI_D$.

\begin{prop} \label{StabilityDegreeM(m)}
 For $m \geq 0$, the $\FIW$--module $M_{\W}(\bm)$ has injectivity degree $0$ and surjectivity degree $m$.
\end{prop}

\begin{proof}[Proof of Proposition \ref{StabilityDegreeM(m)}]
Fix $a \geq 0$. By definition, $(S_{+a}M_{\W}(\bm))_n$ is the free vector space over $\Hom_{\FIW}(\bm, {\bf (n+a) })$, where $\W_n$ acts by postcomposition, permuting $$\{\pm(1+a), \ldots, \pm (n+a) \} \subseteq {\bf (n+a) }.$$ In type BC, two maps in $\Hom_{\FI_{BC}}(\bm, {\bf (n+a) })$ are in the same orbit if they restrict to the same function on their inverse images of $\{\pm 1, \ldots, \pm a \}$. Thus we can identify $\Phi_a(M_{BC}(\bm))$ with the vector spaces with bases:
\begin{align*} n \neq 0, \quad B^{BC}_{\leq n} = \{  f: & \{ \pm 1 , \ldots, \pm m \} \to \{\pm 1, \ldots,  \pm a, \star \} \;  \vert \;  \\ &  \; \text{Away from $\star$, $f$ injects and  $f(-d) = -f(d)$;}   \quad |f^{-1}(\star)| \leq n \} \end{align*}
 $\displaystyle \; \; n=0, \quad B^{BC}_{\leq 0} = \Hom_{\FI_{BC}}(\bm, \ba)$.

Type D, however, is slightly more subtle. When $n>(m-a)$, then the orbit of a map $g \in \Hom_{\FI_{D}}(\bm, {\bf (n+a) })$ is no longer determined by the restriction of $g$ to $g^{-1}(\{\pm 1, \ldots, \pm a \})$. Since $g$ can reverse an even or odd number of signs, but all elements of $D_n$ only reverse an even number, the orbit of $g$ will also depend on whether $g$ reverses an even or odd number of signs of numerals in $g^{-1}(\{\pm (a+1), \ldots, \pm (n+a) \})$. 

Thus $\Phi_a(M_{D}(\bm))$ are the vector spaces freely spanned by:
\begin{align*}  
 n \neq 0, n>(m-a), \quad & B^{D}_{\leq n}  \text{ is two copies of }  \\ 
& \{  \; f: \{ \pm 1 , \ldots, \pm m \} \to \{\pm 1, \ldots, \pm a, \star \} \; \vert  \;  \\ 
& \text{Away from $\star$, $f$ injects and  $f(-d) = -f(d)$;} \quad |f^{-1}(\star)| \leq n \}  
\end{align*}
\begin{align*} n \neq 0, n=m-a, \quad &  B^{D}_{\leq (m-a)} \text{ is one copy of } \\ 
 \{ & \; f: \{ \pm 1 , \ldots, \pm m \} \to \{\pm 1, \ldots, \pm a, \star \} \; \vert  \\ & 
\; \text{Away from $\star$, $f$ injects and  $f(-d) = -f(d)$;}  \quad |f^{-1}(\star)| \leq n \}  \end{align*}
$\displaystyle \qquad n=0,\quad B^{D}_{\leq 0} = \Hom_{\FI_{D}}(\bm, \ba)$\\

 In all cases, however, we have inclusions $B^{\W}_{\leq n} \hookrightarrow B^{\W}_{\leq(n+1)}$, and so $M_{\W}(\bm)$ has injectivity degree $0$. Moreover, once $n \geq m$, all maps $f$ automatically satisfy the condition on $|f^{-1}(\star)|$, and so $B^{\W}_{\leq n} = B^{\W}_{\leq(n+1)}$ in this range. We conclude that $M_{\W}(\bm)$ has surjectivity degree $m$. 
\end{proof}

\begin{prop}\label{FiniteGenerationSurjectivityDegree}
 If an $\FIW$--module $V$ is generated in degree $\leq m$, then $V$ has surjectivity degree $\leq m$. 
\end{prop}

Church--Ellenberg--Farb prove this result for $\FI_A$ in \cite[Proposition 2.39]{CEF}. 

\begin{proof}[Proof of Proposition \ref{FiniteGenerationSurjectivityDegree}]
By Lemma \ref{FinGen}, any $\FIW$--module $V$ generated in degree $\leq m$ admits a surjection from some $\FIW$--module $\bigoplus_{a=0}^m M_{\W}(\ba)^{\oplus b_a}$, which has surjectivity degree $m$ by Proposition \ref{StabilityDegreeM(m)}. The result follows from Remark \ref{StabilityToQuotients}.
\end{proof}

\begin{prop} \label{StabilityDegreeM(U)}
 Given a nonzero $\W_m$--representation $U$, the $\FIW$--module $M_{\W}(U)$ has injectivity degree $0$ and surjectivity degree $\leq m$. 
\end{prop}

\begin{proof}[Proof of \ref{StabilityDegreeM(U)}]
Since $M_{\W}(U)$ is generated in degree $m$, it has surjectivity degree $\leq m$ by Proposition \ref{FiniteGenerationSurjectivityDegree}. Church--Ellenberg--Farb show that $M_A(U)$ has injectivity degree $0$ \cite[Proposition 2.38]{CEF} by noting that (in the notation of the proof of Proposition \ref{StabilityDegreeM(m)}), $k[B^A_{\leq n}]$ embeds as a $k[S_m]$--equivariant summand of $k[B^A_{\leq (n+1)}]$, and so the maps 
$$ \Phi_a (M_A(U))_n =  \Phi_a (M_A(m))_n \otimes_{k[S_m]} U \longrightarrow \Phi_a (M_A(m))_{n+1} \otimes_{k[S_m]} U = \Phi_a (M_A (U))_{n+1}$$
are injective for all $a,n$. Applying the same argument to $B^{BC}_{\leq n}$ and $B^D_{\leq n}$, we conclude that $M_{BC}(U)$ and $M_D(U)$ have injectivity degree $0$. 
\end{proof}

\begin{prop}\label{StabilityKernelCokernel}
 Let $f:V \to U$ be a morphism of $\FIW$--modules, and assume that $k$ contains $\Q$. Suppose that $V$ has injectivity degree $\leq B$ and surjectivity degree $\leq C$, and that $U$ has injectivity degree $\leq D$ and surjectivity degree $\leq E$. Then $\ker(f)$ has injectivity degree $\leq B$ and surjectivity degree $\leq \max(C,D)$, and $ \coker(f)$ has injectivity degree $\leq \max(C,D)$ and surjectivity degree $\leq E$.  
\end{prop}

\begin{proof}[Proof of Proposition \ref{StabilityKernelCokernel}]
 The proof given for $\FI_A$ in \cite[Proposition 2.44]{CEF} carries through directly. The idea is to use exactness of the functor $\Phi_a$, and perform diagram chases on the following diagrams:
{\small
\begin{align*} \mkern-36mu \mkern-36mu
 \xymatrix{
& 0 \ar[r] & \Phi_a(\ker f ) \ar[d] \ar[r] & \Phi_a(V) \ar[d] \ar[r] & \Phi_a(U) \ar[d] \ar@{}[dr]|{\text{and}} &   \Phi_a(V) \ar[d] \ar[r] & \Phi_a(U) \ar[d] \ar[r] & \Phi_a(\coker f ) \ar[d]  \ar[r] & 0 \\
& 0 \ar[r] & \Phi_a(\ker f ) \ar[r] & \Phi_a(V) \ar[r] & \Phi_a(U) &   \Phi_a(V) \ar[r] & \Phi_a(U) \ar[r] & \Phi_a(\coker f ) \ar[r] & 0 \\
} \end{align*} }
\end{proof}

\begin{lem} \label{FinGenImpliesStabilityDeg} Suppose $k$ contains $\Q$.
 Let $V$ be a finitely presented $\FIW$--module with generator degree $g$ and relation degree $r$. Then $V$ has stability degree $ \leq \max(r,g)$.  
\end{lem}

We mimic the proof of \cite[Proposition 2.47]{CEF}.

\begin{proof}[Proof of Lemma \ref{FinGenImpliesStabilityDeg}]
 By assumption, there is an exact sequence $$ 0 \longrightarrow K \longrightarrow \bigoplus_{a=0}^g M_{\W}(\ba)^{\oplus b_a} \longrightarrow V \longrightarrow 0$$ with kernel $K$ generated in degree $\leq r$. By Proposition \ref{StabilityDegreeM(m)}, $\bigoplus_{i=0}^g M_{\W}(\bm_i)^{\oplus b_i}$ has injectivity degree $0$ and surjectivity degree $g$. By Proposition \ref{FiniteGenerationSurjectivityDegree}, $K$ has surjectivity degree $\leq r$. The result then follows from Proposition \ref{StabilityKernelCokernel}.
\end{proof}




\begin{lem} \label{StabilityDegreeBoundsLambda1} {\bf ($ \y_1^+ \leq$ stability degree ).}
Suppose that $V$ is an $\FI_{BC}$--module over a field $k$ of characteristic zero, and suppose $V$ has stability degree $s$. For any $n \geq 0$, and (in the notation of Section \ref{BackgroundRepStability}) any $V(\y^+, \y^-)_n$ in $V_n$, the largest part $\y^+_1$ of $\y^+$ satisfies $\y^+_1 \leq s$. 
\end{lem}

Lemma \ref{StabilityDegreeBoundsLambda1} parallels \cite[Proposition 2.42]{CEF}, and we adapt this argument. In Theorem \ref{StabilityDegreeImpliesRepStability}, we will use Lemma \ref{StabilityDegreeBoundsLambda1} to relate the stability degree of an $\FI_{BC}$--module $V$ to the representation stability of $\{V_n\}$. 

\begin{proof}[Proof of Lemma \ref{StabilityDegreeBoundsLambda1}]

Let $\y$ and $\mu$ denote the double partitions $\y = (\y^+, \y^-)$ and $\mu = (\mu^+, \mu^-)$, and let $\y^+_1$ and $\mu^+_1$ denote the largest parts of $\y^+$ and $\mu^+$, respectively. Let $m=|\mu|$, and denote $\nu = \y[n]$.  First we note that every irreducible representation $V(\y)_n = V_{\nu}$ in $M_{BC}(V_\mu)_n$ must satisfy $\y^+_1 \leq \mu^+_1$. By the branching rules for $B_n$, Equation (\ref{Eqn:WnInducedDecomp}), \begin{gather*}
V(\y)_n = V_{\nu} \quad \subseteq \qquad  M_{BC}(V_\mu)_n = \Ind_{B_m \times B_{n-m}}^{B_n} V_{\mu} \boxtimes k  \\
 \Longleftrightarrow                                                                                    \\                                                                                                                                                                                                                                                                                                                           
\nu^-=\mu^-      \quad \text{and} \quad \text{ $\nu^+$ can be built from $\mu^+$ by adding $(n-m)$ boxes in distinct columns.}                                                                                                                                                                                                                                                                                                                                                                                                 \end{gather*}
A box can be added to the end of the second row $\mu_2^+$ only if $\mu_1^+>\mu_2^+.$ This implies that $\y^+_1$, the second row in $\nu^+ = \y^+[n-|\y^-|]$, must be no larger than $\mu^+_1$, the first row of $\mu^+$. Some small cases are shown in Figure \ref{fig:y_1+StabilityDegree}.

\begin{center}
\begin{figure}[h!]
\fbox{ \includegraphics[scale=1.3]{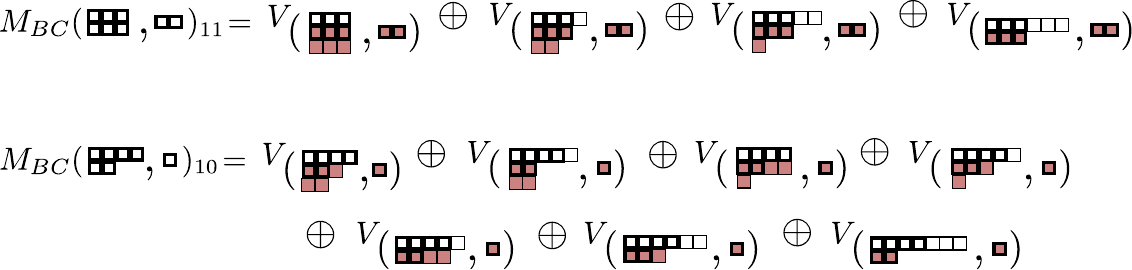}}
\caption{ {\small Illustrating the inequality $\y^+_1 \leq \mu^+_1$.  The double partitions $\mu$ are bolded, and $\y$ are shaded in.}} 
\label{fig:y_1+StabilityDegree}
\end{figure}
\end{center}

\vspace{-.8cm}
We next claim that if $V$ has stability degree $s$, then every irreducible subrepresentation $V_{\mu}$ of $H_0(V)_m$ must satisfy $\mu_1^+ \leq s$. Recall from Definition \ref{DefnH0} that $H_0(V)^{\FI_{BC}}$ denotes the $\FI_{BC}$--module structure on $H_0(V)$ wherein all morphisms $I_n$ act by zero. The surjection $V \twoheadrightarrow H_0(V)^{\FI_{BC}}$ implies by Remark \ref{StabilityToQuotients} that $H_0(V)^{\FI_{BC}}$ has surjectivity degree $\leq s$. Suppose that $V_{\mu} \subseteq H_0(V)_m$, and let $\eta$ be the double partition such that $\mu = \eta[m]$. Remark \ref{CoinvariantsBranchingRule} and the branching rules (Equation (\ref{Eqn:WnInvariantsDecomp}), Section \ref{RepTheoryWn}) imply that $(V_{\mu})_{B_{\mu^+_1}} = V_{\eta}$. It follows that when $a = m- \mu^+_1$, $$\Phi_a(H_0(V)^{\FI_{BC}})_{\mu^+_1} =(H_0(V)^{\FI_{BC}}_m)_{B_{\mu^+_1}} \qquad \text{contains $V_{\eta}$,}$$  and so in particular the coinvariants $\Phi_a(H_0(V)^{\FI_{BC}})_{\mu^+_1}$ are nonzero. Since $T$ acts by zero, the 
surjectivity degree of $H_0(V)^{\FI_{BC}}$ must be greater than $\mu_1^+$, and we conclude that $\mu_1^+ \leq s$.  

Now, suppose that $V$ is an $\FI_{BC}$--module over characteristic zero with stability degree $s$. We've shown that every irreducible component $V_{\mu}$ of $H_0(V)$ satisfies $\mu^+_1 \leq s$, and so by the first paragraph above, for each $n$, every $V(\y)_n$ in $M_{BC}(H_0(V))_n$ satisfies $\y^+_1 \leq \mu^+_1 \leq s$. Since by Remark \ref{SurjectionHoV} there is a surjection $M_{BC}(H_0(V)) \twoheadrightarrow V$, we conclude that for each $n$, every irreducible constituent $V(\y)_n$ of $V_n$ satisfies $\y^+_1 \leq s$. \end{proof}

\subsection{The Noetherian property}\label{SectionNoetherian}

In \cite[Theorem 2.60]{CEF}, Church--Ellenberg--Farb prove a Noetherian property for $\FI_{A}$--modules over Noetherian rings containing $\Q$. This theorem is later generalized by Church--Ellenberg--Farb--Nagpal in \cite[Theorem 1.1]{CEFN} to arbitrary Noetherian rings. We will show that the same is true for modules over all three categories $\FIW$.

\begin{thm} {\bf ($\FIW$--modules are Noetherian).} \label{Noetherian} Let $k$ be a Noetherian ring. Then any sub--$\FIW$--module of a finitely generated $\FIW$--module over $k$ is itself finitely generated.
\end{thm}

\begin{proof}
 Suppose that $V$ is a finitely generated $\FIW$--module, and $U$ is any sub--$\FIW$--module. When $\W$ is $S_n$, the result is \cite[Theorem 1.1]{CEFN}. If $\W_n$ is $B_n$ or $D_n$, then by Propositions \ref{RestrictionPreservesFinGen}(\ref{WntoSn}) and (\ref{W'ntoSn}), the restriction of $V$ to $\FI_A$ is finitely generated as an $\FI_A$--module. Thus by \cite[Theorem 1.1]{CEFN}, $U$ is finitely generated over $\FI_A \subseteq \FIW$, and therefore it is finitely generated over $\FI_{\W}$. 
\end{proof}

\subsection{Finite generation and representation stability} \label{SectionFinGenRepStability}

Representation stability is a concept introduced by Church and Farb in \cite{RepStability}; the definition is given in Section \ref{BackgroundRepStability}. In \cite[Section 2.7]{CEF}, Church--Ellenberg--Farb prove that for an $\FI_A$--module $V$ over a field $k$ of characteristic zero, finite generation is equivalent to uniform representation stability for the sequence of $S_n$--representations $\{V_n\}$. We will show, in Theorems \ref{FinGenImpliesRepStability} and \ref{RepStabilityImpliesFinGen}, the analogous results for $\FI_{BC}$. In summary:

\begin{thm} \label{FinGenIffRepStable} {\bf (Summary: Finite generation $\Longleftrightarrow$ Uniform representation stability).}
 Let $V$ be an $\FIW$--module over a field of characteristic zero. The $V$ is finitely generated if and only if $\{V_n \}$ is uniformly representation stable with respect to the maps induced by the natural inclusions $I_n : \bn \hookrightarrow {\bf (n+1)}$. 
\end{thm}




We first show that, in the notation of Section \ref{BackgroundRepStability}, sequences of $B_n$--representations of the form $\oplus_{\y} c_{\y} V(\y)_n$ over characteristic zero are determined up to isomorphism by their coinvariants. We will use this result to relate finite generation with representation stability.

\begin{lem}{\bf (The $B_a$--representations $(V(\y)_n)_{B_{n-a}}$ stabilize for $n \geq a + \y^+_1$)} \label{CoinvariantsIndependent} Suppose $k$ is a field of characteristic zero. Given a double partition $\y = (\y^+, \y^-)$, the $B_a$--representations $(V(\y)_n)_{B_{n-a}}$ are independent of $n$ for $n \geq a + \y^+_1$, where $\y^+_1$ denotes the largest part of $\y^+$. 
\end{lem}

\begin{proof}[Proof of Lemma \ref{CoinvariantsIndependent}]
By Remark \ref{CoinvariantsBranchingRule} and the branching rules (Equation \ref{Eqn:WnInvariantsDecomp}), 
$$(V(\y^+, \y^-)_n)_{B_{n-a}} = V\big(\y^+\big[n-|\y^-|\big], \y^-\big)_{B_{n-a}} = \bigoplus_{\mu^+} V(\mu^+, \y^-) $$ summed over all partitions $\mu^+$ that can be constructed from $\y^+[n-|\y^-|]$ by removing $(n-a)$ boxes from distinct columns. Once $(n-a) > \y^+_1$, at least one box must be removed from the top row of $\y^+[n-|\y^-|]$. Removing one box from the top row of the padded partition $\y^+[n-|\y^-|]$ associated to $n$ yields the padded partition $\y^+[(n-1)-|\y^-|]$ associated to $(n-1)$, and removing the remaining  $((n-1)-a)$ boxes from distinct columns gives the decomposition of $(V(\y^+, \y^-)_{n-1})_{B_{(n-1)-a}}$. Thus, as $B_a$--representations,
$$(V(\y^+, \y^-)_{n-1})_{B_{(n-1)-a}} \cong (V(\y^+, \y^-)_n)_{B_{n-a}} \qquad \text{for all $n>a + \y^+_1$,}$$ and the lemma follows. \end{proof}

\begin{lem} \label{CoinvariantsDetermineRep} {\bf ($B_n$--representations are determined by their coinvariants).} Assume that $k$ is a field of characteristic zero. Let $\Lambda$ be a set of double partitions $\y = (\y^+, \y^-)$ of size at most $d$, and set $M = \max_{\Lambda} \y^+_1$, where $\y^+_1$ denotes the largest part of $\y^+$. Let $n \geq m \geq (d+M)$ be nonnegative integers. Suppose that for some $B_m$--representation $V_m$ and $B_n$--representation $V_n$,  
$$ V_m \cong \bigoplus_{\y \in \Lambda} b_{\y} V(\y)_m \qquad \text{ and } \qquad V_n \cong \bigoplus_{\y \in \Lambda} c_{\y} V(\y)_n,$$  If for each $0 \leq a \leq d$, the coinvariants $ (V_m)_{B_{m-a}} \cong (V_n)_{B_{n-a}} $ are isomorphic as $B_a$--representations, then $c_\y = b_\y$  for all $\y \in \Lambda$. 
\end{lem}

\begin{cor} \label{Coinvariants0Vanishes}
With $n$ and $d$ as above, the coinvariants $(V_n)_{B_{n-a}} = 0$ for all $0 \leq a \leq d$ if and only if $V_n=0$.  
\end{cor}

Church--Ellenberg--Farb prove an analogous result to Lemma \ref{CoinvariantsDetermineRep} for the symmetric group in \cite[Lemma 2.40 and Proposition 2.58]{CEF}. We adapt their methods in the following proof. 

\begin{proof}[Proof of Lemma \ref{CoinvariantsDetermineRep}]
 We will prove that $c_\y = b_\y$ for all $|\y| \leq p$ for each $p$ with $0 \leq p \leq d$, proceeding by induction on $p$. 

If $p=0$, then the only double partition $\y$ of size at most $p$ is the double partition $\y = ( \varnothing, \varnothing)$ associated to the trivial representation. Taking $a=0$, we see $$c_\y = \dim\big((V_n)_{B_{n}}\big) = \dim\big((V_m)_{B_{m}}\big) = b_\y.$$ The conclusion follows for $p=0$. 

 Consider some double partition $\y = (\y^+, \y^-)$. By Remark \ref{CoinvariantsBranchingRule} and the branching rules (Equation (\ref{Eqn:WnInvariantsDecomp}), Section \ref{RepTheoryWn}), the multiplicity of $V(\nu^+, \nu^-)$ in $\big(V(\y)_n\big)_{B_{n-a}}$ is: 
\begin{align*} 
 &  \left\{ \begin{array}{ll}
         1 & \mbox{if $\nu^+$ can be built by removing $(n-a)$ boxes from $\y^+[n-|\y^-|]$,} \mbox{ at most one box per column},\\
         0 & \mbox{otherwise}.\end{array} \right.
\end{align*} 
Since the largest part of  $\y^+[n-|\y^-|]$ is $(n-|\y|)$, the coinvariants must vanish when $|\y| > a$, and when $|\y|=a$, $(V(\y)_n)_{B_{n-a}}$ is a single copy of the $B_a$--representation $V_{\y}$.  

Now, suppose (as inductive hypothesis) that $c_{\y} = b_{\y}$ for all $|\y| < p$, and consider the coinvariants corresponding to $a=p$. We have: 
\begin{align*}
 (V_m)_{B_{m-p}} \; &= \bigoplus_{|\y|=p} c_{\y} V_{\y} \bigoplus_{|\y|<p} c_{\y} \; \big(V(\y)_m\big)_{B_{m-p}} \qquad
 (V_n)_{B_{n-p}} \; &= \bigoplus_{|\y|=p} b_{\y} V_{\y} \bigoplus_{|\y|<p} b_{\y} \; \big(V(\y)_n\big)_{B_{n-p}}
\end{align*}
By the inductive hypothesis, the subrepresentations of $V_m$ and $V_n$ of weight $< p$ are isomorphic. Since by assumption $p+ \max_{\Lambda} \y_1^+ \leq d+M \leq m, n$, Lemma \ref{CoinvariantsIndependent} implies that the coinvariants of these subrepresentations are isomorphic. Thus $(V_m)_{B_{m-p}} \cong (V_n)_{B_{n-p}}$ only if $c_\y = b_\y$ for all $\y$ with $|\y|=p$. The lemma follows by induction. 
\end{proof}

\begin{thm} \label{StabilityDegreeImpliesRepStability} {\bf (Finitely generated $\FIW$--modules are uniformly representation stable).} Suppose that $k$ is a characteristic zero field, and that $V$ is a $\FI_{BC}$--module with weight $\leq d$, and stability degree $N$. Then, $\{V_n \}$ is uniformly representation stable with respect to the maps $\phi_n: V_n \to V_{n+1}$ induced by the natural inclusions $I_n: \bn \hookrightarrow {\bf (n+1)}$. The sequences stabilizes for $n \geq N+d$.
\end{thm}

The arguments used in \cite[Theorem 2.58]{CEF} carry through to type B/C; we briefly give these arguments here.  

\begin{proof}[Proof of Theorem \ref{StabilityDegreeImpliesRepStability}]

We note that, by Lemma \ref{StabilityDegreeBoundsLambda1}, for all $n$ and all irreducible components $V(\y^+, \y^-)_n$ in $V_n$, the largest part $\y^+_1$ of $\y^+$ is less than $N$. We can therefore apply Lemma \ref{CoinvariantsDetermineRep} and Corollary \ref{Coinvariants0Vanishes} to the representations $V_n$ for any $n \geq N+d \geq \max \y_1^+ + d$.\\

{\noindent \bf I. Injectivity.} Let $K_n$ denote the kernel of $\phi_n$. By assumption that $V$ has stability degree $N$, the composite $$ (V_n)_{B_{n-d}} \to (V_{n+1})_{B_{n-d}} \to (V_{n+1})_{B_{n+1-d}}$$ is an isomorphism for $n \geq N+d$, which implies that the first map is injective. The operation of taking coinvariants is exact in characteristic zero, and so it follows that its kernel is isomorphic to $(K_n)_{B_{n-d}}$. Thus $(K_n)_{B_{n-d}}=0$, and so  $K_n=0$ by Corollary \ref{Coinvariants0Vanishes}. This proves injectivity of $\phi_n$ for $n \geq N+d$. \\

 {\noindent \bf II. Surjectivity.}
 To prove that $\phi_n(V_n)$ generates $V_{n+1}$ as a $k[B_{n+1}]$--module, it suffices to show that the induced map $\Ind(\phi_n): \Ind_{B_n}^{B_{n+1}} V_n \to V_{n+1}$ is surjective. Let $C_{n+1}$ denote the cokernel of this map. The composition $$ (V_n)_{B_{n-d}} \to (\Ind_{B_n}^{B_{n+1}} V_n)_{B_{n+1-d}} \to (V_{n+1})_{B_{n+1-d}}$$ is an isomorphism for $n \geq N+d$ by assumption. Thus $(C_{n+1})_{B_{n-d}}$ vanishes, and so $C_{n+1}$ vanishes by Corollary \ref{Coinvariants0Vanishes}, and $I_n$ surject for $n \geq N+d$.  \\

{\noindent \bf III. Multiplicity Stability.}
By assumption, $(V_n)_{B_{n-a}} \cong (V_{n+1})_{B_{n+1-a}}$ for all $a \geq 0$ and $n \geq N+a$.

Thus for $n \geq N+d$, Lemma \ref{CoinvariantsDetermineRep} implies that the multiplicity of each irreducible $V(\y)_n$ in $V_n$ is constant. This completes the proof. 
\end{proof}


\begin{thm} \label{FinGenImpliesRepStability}{\bf (Finitely generated $\FI_{\W}$--modules are uniformly representation stable).} Suppose that $k$ is a field of characteristic zero, and $\W_n$ is $S_n$, $D_n$ or $B_n$. Let $V$ be a finitely generated $\FI_{\W}$--module. Take $d$ to be an upper bound on the weight of $V$, $g$ an upper bound on its degree of generation, and $r$  an upper bound on its relation degree.  Then, $\{V_n \}$ is uniformly representation stable with respect to the maps induced by the natural inclusions $I_n: \bn \to {\bf (n+1)}$, stabilizing once $n \geq \max(g,r)+d$; when $\W_n$ is $D_n$ and $d=0$ we need the additional condition that $n \geq g+1$.
\end{thm}

\begin{proof}[Proof of Theorem \ref{FinGenImpliesRepStability}]

Suppose first that $\W_n$ is $S_n$ or $B_n$. By Lemma \ref{FinGenImpliesStabilityDeg}, $V$ has stability degree $\max(g,r)$. The conclusion follows from \cite[Proposition 2.58]{CEF} in type A and Theorem \ref{StabilityDegreeImpliesRepStability} in type B/C. 

Next suppose $\W_n$ is $D_n$, and consider the $\FI_{BC}$--module $\Ind_D^{BC} V$. $\Ind_D^{BC} V$ has weight $d$ by the definition of weight for $\FI_D$--modules. By Lemma \ref{IndGenRelDegrees}, $\Ind_D^{BC} V$ will also have generator degree $\leq g$ and relation degree $\leq r$. Hence $\{(\Ind_{D}^{BC} \; V)_n\}$ is uniformly representation stable with respect to the $B_n$ action for $n \geq \max(g,r)+d$. However, by Proposition \ref{VisResIndV}, 

$$V_n \cong (\Res_D^{BC} \Ind_D^{BC} V)_n \qquad \text{as $D_n$-representations, \quad for $n \geq g+1$}$$

\noindent It follows that, as a sequence of $D_n$--representations,  $(\{V_n\}, (I_n)_*)$ satisfies the injectivity (I) and multiplicity stability (III) criteria for uniform representation stability in the range $n \geq \max(g+1, g+d, r+d)$. Moreover $(\{V_n\}, (I_n)_*)$ must satisfy the "surjectivity" (II) criterion for all $n \geq g$ by assumption that $g$ bounds its degree of generation. We conclude that $\{V_n\}$ is a uniformly representation stable sequence of $D_n$--representations with respect to the maps $(I_n)_*$, stabilizing for $n \geq \max(g+1, g+d, r+d)$. 
 \end{proof}

Note that, since an $\FI_{BC}$--module $V$ must have weight $\leq g$ by Theorem \ref{WnDiagramSizes}, it is uniformly representation stable for $n \geq \max(2g, g+r)$.

\begin{thm} \label{RepStabilityImpliesFinGen} {\bf (Uniformly representation stable $\FI_{\W}$--modules are finitely generated).} Suppose conversely that $V$ is an $\FI_{\W}$--module, and that $\{V_n, (I_n)_* \}$ is uniformly representation stable for $n \geq N$. Then $V$ is finitely generated in degree $\leq N$. 
\end{thm}

\begin{proof}
 For $n \geq N$, the ''surjectivity'' criterion for representation stability implies that $(I_n)_*(V_n)$ generates $V_{n+1}$ as a $k[\W_{n+1}]$--module. Since each vector space $V_n$ is finite dimensional by assumption, we can take bases for $\{ V_m \}_{m \leq N}$ to be our finite generating set. 
\end{proof}

\begin{rem} {\bf ($\FIW$--modules cannot be non-uniformly representation stable).} We note that the assumption of uniformity of representation stability was not needed for Theorem \ref{RepStabilityImpliesFinGen}, that is, the proof would hold equally for any (possibly non-uniformly) representation stable sequence $\{V_n, (I_n)_*\}$ which satisfied just the "surjectivity" criterion in the range $n \geq N$. It follows that, over characteristic zero, any sequences of either $S_n$ or $B_n$--representations that is non-uniformly representation stable cannot admit an $\FIW$--module structure. If such a sequence were an $\FIW$--module, representation stability would imply finite generation, which would imply uniform representation stability, a contradiction. The alternating representations of $S_n$ and sign representation of $B_n$ are examples of a consistent sequences that are non-uniformly representation stable, but not $\FIW$--modules.
\end{rem}

\subsection{The functor $\t_{\geq d}$ and the $\FIW$--module $V(\y)$}

Given an $\FI_{A}$ or $\FI_{BC}$--module $V$, we will define a filtration $\t_{\geq d}V$ of $V$ into sub$-\FIW$--modules by weight. We will use this filtration to construct an $\FI_{BC}$--module structure on the sequence of $B_n$--representations $\{ V(\y^+, \y^-)_n \}$, and an $\FI_D$--module structure on the sequence of $D_n$--representations $\{ V\{\y^+[n-|\y^-|], \y^-\} \}$, associated to any double partition $(\y^+, \y^-)$. The $\FI_{BC}$--module $V(\y)$ will feature later in the proof of Theorem \ref{MurnaghanWn}, the type B/C and D analogues of Murnaghan's stability theorem for tensor products of $S_n$--representations.

\begin{prop}{\bf(The functor $\t_{\geq d}$).} Suppose that $\W_n$ is $S_n$ or $B_n$, and that $k$ is a field of characteristic zero. Fix $d\geq 0$. Any $\FIW$--module $V$ over $k$ contains a sub--$\FIW$--module, which we denote  $\t_{\geq d}V$, defined by
$$(\t_{\geq d}V)_n  \text{ is the sum of all components $V(\y)_n$ of $V_n$ with $|\y| \geq d$.} $$
\end{prop}
\begin{proof}
 For $\W_n = S_n$, this is proven by Church--Ellenberg--Farb \cite[Proposition 2.54]{CEF}. Their arguments also adapt to type B/C:

The spaces $(\t_{\geq d}V)_n$ are by construction $B_n$--invariant, and so it suffices to show that the image of $(I_{m,n})_* : (\t_{\geq d}V)_m \to V_n$ lies in $(\t_{\geq d}V)_n$ for all $m,n$. The branching rules for the hyperoctahedral group (deduced from Equation (\ref{Eqn:PieriGeneral}) and Frobenius reciprocity) assert that a $B_m$--representation $V(\mu)_m = V(\mu^+, \mu^-)_m$ is contained in the restriction of a $B_n$--representation $V(\y)_n = V(\y^+, \y^-)_n$ only if $\mu[m] \subseteq \y[n]$. In particular, if $|\mu| \geq d$ then necessarily $|\y| \geq d$.  
Since $(I_{m,n})_*$ is $B_m$--equivariant with respect to the restriction of $V_n$ to $B_m$, the branching rules imply that $(I_{m,n})_*\big( (\t_{\geq d}V)_m \big) \subseteq (\t_{\geq d}V)_n$, as required. 
\end{proof}

The decomposition of the filtration $$M_{A}\bigg(\Y{2,1}\bigg) = \t_{\geq 1} M_{A}\bigg(\Y{2,1}\bigg) \supseteq \t_{\geq 2} M_{A}\bigg(\Y{2,1}\bigg) \supseteq \t_{\geq 3} M_{A}\bigg(\Y{2,1}\bigg) \supseteq \t_{\geq 4} M_{A}\bigg(\Y{2,1}\bigg) =0.$$ is show in Figure \ref{MAFiltrationExample}.

\begin{figure}[h!]
\resizebox{\textwidth}{!}{ \fbox{
\xymatrix @-2pc{
 & n=2&& n=3 && n=4 && n=5 && n=6 \\& &&  &&  &&  &&  \\ 
M_{A}\bigg(\Y{2,1}\bigg) & 0  &\to &   V_{\Y{2,1}}  &\to & V_{\Y{2,1,1}}\oplus V_{\Y{2,2}} \oplus V_{\Y{3,1}} &\to & V_{\Y{2,2,1}} \oplus V_{\Y{3,1,1}} \oplus V_{\Y{3,2}}\oplus V_{\Y{4,1}} &\to & V_{\Y{3,2,1}} \oplus V_{\Y{4,1,1}} \oplus V_{\Y{4,2}}\oplus V_{\Y{5,1}}\\ 
\t_{\geq 2} M_{A}\bigg(\Y{2,1}\bigg) & 0 &\to & 0 &\to & V_{\Y{2,1,1}}\oplus V_{\Y{2,2}} &\to & V_{\Y{2,2,1}} \oplus V_{\Y{3,1,1}} \oplus V_{\Y{3,2}} &\to & V_{\Y{3,2,1}} \oplus V_{\Y{4,1,1}} \oplus V_{\Y{4,2}}\\
\t_{\geq 3} M_{A}\bigg(\Y{2,1}\bigg) & 0 &\to& 0 &\to & 0 &\to &  V_{\Y{2,2,1}} &\to & V_{\Y{3,2,1}} \\   
\t_{\geq 4} M_{A}\bigg(\Y{2,1}\bigg) & 0 &\to& 0 &\to & 0 &\to &  0 &\to & 0 \\   
} } }
\caption{Filtration of $M_A(2,1)$ by weight.}
\label{MAFiltrationExample}
\end{figure}

Suppose $k$ is a characteristic zero field. Let $\y$ be a partition with largest part $\y_1$. In \cite[Proposition 2.56]{CEF}, Church--Ellenberg--Farb define the $\FI_A$--module $$V(\y) := \t_{\geq |\y|}M_A(\y).$$ They show that $V(\y)$ is finitely generated in degree $|\y|+\y_1$, and satisfies 
\begin{align*}
V(\y)_n = \left\{ \begin{array}{ll}
         V_{\y[n]} & \mbox{if $n \geq |\y| + \y_1$,} \\
         0 & \mbox{otherwise}.\end{array} \right. 
\end{align*}
 We give the analogous construction for $\FI_{BC}$. 

\begin{defn}\label{Defn:FIModuleV(y)} {\bf(The $\FI_{BC}$--module $V(\y)=V(\y^+, \y^-)$).} Let $k$ be a field of characteristic zero. Let $\y=(\y^+, \y^-)$ be a double partition. Define the $\FI_{BC}$--module $V(\y)$ by $$V(\y) := \t_{\geq |\y|}M_{BC}(\y).$$ 
 \end{defn}
This is consistent with the notation given in Section \ref{BackgroundRepStability}.
\begin{prop} \label{FIModuleV(y)} Let $\y^+_1$ denote the largest part of $\y^+$. The $\FI_{BC}$--module $V(\y)=V(\y^+, \y^-)$ satisfies
\begin{align*}
V(\y)_n = \left\{ \begin{array}{ll}
         V_{\y[n]} & \mbox{if $n \geq |\y^+| + |\y^-| + \y^+_1$,} \\
         0 & \mbox{otherwise}.\end{array} \right. 
\end{align*}
and $V(\y)$ is finitely generated in degree $|\y^+| + |\y^-| + \y^+_1$. 
 \end{prop}

\begin{proof}
Let $a=|\y|$. By definition, 
 \begin{align*}
M_{BC}(\y) &= \left\{ \begin{array}{ll}
         0 & \mbox{$ n < a $},\\
        \Ind_{B_a \times B_{n-a}}^{B_n} V_{\y} \boxtimes k & \mbox{$ n \geq a$}.\end{array} \right.  
 \end{align*}
and so the branching rule (Equation (\ref{Eqn:WnInducedDecomp})) implies that, for $n \geq |\y|$, $$M_{BC}(\y)_n = \bigoplus_{\mu^+ } V(\mu^+,\y^-)$$
summed over all partitions $\mu^+$ constructed by adding $(n-|\y^+|-|\y^-|)$ boxes in distinct columns of $\y^+$. An irreducible component $V(\mu^+,\y^-)$ can appear in $(\t_{\geq |\y|}M_{BC}(\y))_n$ only if $\y^+_1$ boxes are added to each column of $\y^+$ below the top row. This happens only once $n \geq |\y^+| + |\y^-| + \y^+_1$, and gives the single irreducible $V(\y^+, \y^-)_n$. 

Since $V(\y)$ consists of a single irreducible $B_n$--representation for all $n \geq |\y^+| + |\y^-| + \y^+_1$, to prove finite generation in degree $|\y^+| + |\y^-| + \y^+_1$ it suffices to show that the maps $V(\y)_n \to V(\y)_{n+1}$ are nonzero in this range. 

We can realize $V(\y)$ as a sub--$\FI_{BC}$--module of $M_{BC}(\ba)$: $$ V(\y) \subseteq M_{BC}(\y) \cong  \left\{ \Ind_{B_a \times B_{n-a}}^{B_n} V_{\y} \boxtimes k \right\}_n \subseteq  \left\{ \Ind_{B_a \times B_{n-a}}^{B_n} k[B_a] \boxtimes k \right\}_n \cong M_{BC}(\ba)  .$$ From the definition of $M_{BC}(\ba)$ in terms of $\FI_{BC}$ morphisms $\Hom_{\FI_{BC}}(\ba, -)$, we see that the maps $M_{BC}(\ba)_n \to M_{BC}(\ba)_{n+1}$ are injective, so the maps $V(\y)_n \to V(\y)_{n+1}$ are injective, and the conclusion follows. 
\end{proof}

By restricting the $\FI_{BC}$--module $V(\y^+, \y^-)$ to the subcategory $\FI_D$, we construct an $\FI_D$--module with the following properties.

\begin{cor}\label{FIDModuleV(y)} Given any ordered pair of partitions $\y = (\y^+,\y^-)$ (with $\y^+_1$ the largest part of $\y^+$), there is an $\FI_D$--module $V(\y)_n$ such that 
\begin{align*}
V(\y)_n = \left\{ \begin{array}{ll}
         V_{\{\y^+[n-|\y^-|],\; \y^-\}} & \mbox{if $n \geq |\y^+| + |\y^-| + \y^+_1$, and $\y^+\big[n-|\y^-|\big] \neq \y^-$ } \\
         V_{\{\y^-,\;+\}} \oplus V_{\{\y^-,\;-\}} & \mbox{if $n \geq |\y^+| + |\y^-| + \y^+_1$, and $\y^+\big[n-|\y^-|\big] = \y^-$} \\
         0 & \mbox{otherwise}.\end{array} \right. 
\end{align*}
\end{cor}

As examples, the $\FI_{BC}$--module $V\Big(\Y{1}\;,\; \Y{2,1}\Big)_n$ and its restriction to $\FI_D$ are shown in Figure \ref{V(y,y)Examples}.

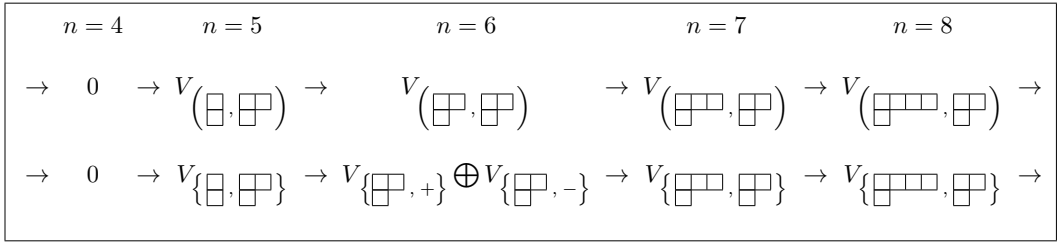
\begin{figure}[h!]
\resizebox{ \textwidth}{!}{ \fbox{
\xymatrix @-2pc{
 & n=4 && n=5 && n=6 && n=7 && n=8 &  \\ 
 & &&  &&  &&  && &  \\ 
 \to & 0  &\to &   V_{\Big( \, \Y{1,1}\;,\;\Y{2,1}  \, \Big)} &\to &  V_{\Big( \, \Y{2,1}\;,\;\Y{2,1} \, \Big)}  &\to &   V_{\Big( \, \Y{3,1}\;,\;\Y{2,1} \, \Big)} &\to &   V_{\Big( \, \Y{4,1}\;,\;\Y{2,1} \, \Big)} &\to \\
 & &&  &&  &&  && &  \\  & &&  &&  &&  && &  \\ 
 \to & 0  &\to &   V_{ \left\{ \, \Y{1,1}\;,\;\Y{2,1} \, \right\} } &\to &  V_{ \, \left\{\Y{2,1}\;,\; +  \right\} } \bigoplus V_{\left\{ \, \Y{2,1}\;,\; - \right\} }  &\to &   V_{\left\{ \, \Y{3,1}\;,\;\Y{2,1}  \, \right\}} &\to &   V_{\left\{ \, \Y{4,1}\;,\;\Y{2,1} \, \right\}} &\to \\ 
  & &&  &&  &&  && &  \\  & &&  &&  &&  && &  \\ 
} } }
\caption{The $\FI_{BC}$ and $\FI_D$--modules $V((1), (2,1))_n$.}
\label{V(y,y)Examples}
\end{figure}

We remark that \cite[Proposition 2.56]{CEF}, Proposition \ref{FIModuleV(y)} and Proposition \ref{FIDModuleV(y)} provide a sort of converse to Theorem \ref{FinGenIffRepStable}: Any sequence of finite dimensional $\W_n$--representations over characteristic zero of the form $\bigoplus_{\y} c_{\y} V(\y)_n$ admits the structure of a finitely generated $\FIW$--module.

\section{Tensor products and $\FIW$--algebras} \label{SectionFIAlgebras}

In this section we define the tensor product of $\FIW$--modules, and show that it respects weight and degree of generation. As a consequence we derive Theorem \ref{MurnaghanWn}, the hyperoctahedral analogue of Murnaghan's theorem on the stability of Kronecker coefficients. We define graded $\FIW$--modules and $\FIW$--algebras, and study some finiteness properties (\emph{finite type} and \emph{slope}) of these objects.

\subsection{Tensor products of $\FIW$--modules} 

\begin{defn} {\bf (Tensor product of $\FIW$--modules).} Given $\FIW$--modules $V$ and $W$, the \emph{tensor product} $V \otimes W$ is the $\FIW$--module such that $(V \otimes W)_n = V_n \otimes W_n$ and the $\FIW$--morphisms act diagonally. 
\end{defn}

\begin{prop}{\bf (Tensor products respect finite generation).}\label{TensorsPreserveFinGen} If $V$ and $W$ are finitely generated $\FIW$--modules, then so is $V \otimes W$. If $V$ is generated in degree $\leq m$ and $W$ in degree $\leq m'$, then $V \otimes W$ is generated in degree $\leq m+m'$. If $k$ is a field of characteristic zero, then weight($V \otimes W$) $\leq$ weight($V$) + weight($W$). 
\end{prop}
 
\begin{proof}[Proof of Proposition \ref{TensorsPreserveFinGen}] 

To prove finite generation, we follow the arguments of \cite[Proposition 2.61]{CEF}. By Proposition \ref{FinGen}, the $\FIW$--modules $V$ and $W$ are quotients of $\FIW$--modules of the form $\oplus_{a=0}^{m}M_{\W}(\ba)^{b_a}$ and $\oplus_{a=0}^{m'}M_{\W}(\ba)^{b_a}$, respectively. It is therefore enough to show that the $\FIW$--module $X : = M_{\W}(\bm) \otimes M_{\W}(\bm')$ is finitely generated in degree $\leq (m+m')$. 

The $\W_n$--representation $X_n$ is, by definition, $ \Span_k \{ (f,f') \in \Hom_{\FIW}(\bm, \bn) \times  \Hom_{\FIW}(\bm', \bn) \}.$
When $n \geq m+m'$, for given $(f,f') \in  X_n$ there exists some $(g,g') \in X_{m+m'}$ and some $h \in \Hom_{\FIW}({\bf m+m'}, \bn)$ so that $h_*(g,g') := (h\circ g, h \circ g') = (f,f').$ We conclude that $X$ is finitely generated in degree $\leq (m+m').$ 

To prove subadditivity of weights, it suffices to show that, in the notation of Section \ref{BackgroundRepStability}, any $\W_n$--representation $V(\nu)_n$ occurring in the product $V(\mu)_n \otimes V(\y)_n$ must satisfy $|\nu| \leq \big(|\mu| + |\y|\big)$. 

Fix $n$. By Proposition \ref{WeightM(m)}, $V(\mu)_n$ and $V(\y)_n$ occur in $M_{\W}\big(|\mu|\big)_n$ and  $M_{\W}\big(|\y|\big)_n$, respectively, and so $V(\mu)_n \otimes V(\y)_n$ is a $\W_n$--subrepresentation of  $M_{\W}\big(|\mu|\big)_n \otimes M_{\W}\big(|\y|\big)_n$. Since $M_{\W}\big(|\mu|\big) \otimes M_{\W}\big(|\y|\big)$ is generated in degree $\leq \big(|\mu| + |\y|\big)$, by Theorem \ref{WnDiagramSizes} it has weight $\leq \big(|\mu| + |\y|\big)$.
\end{proof}

\subsection{An analogue of Murnaghan's stability theorem in types B/C and D} \label{SectionMurnaghan}

It is a classical result of Murnaghan \cite{MurnaghanKronecker} that the \emph{Kronecker coefficients} of the symmetric group, the structure constants $K^{\nu}_{\y, \mu}$ in the decomposition of a tensor product $$V(\y)_n \otimes V(\mu)_n = \bigoplus_{\nu}  K^{\nu}_{\y, \mu} V(\nu)_n,$$ are eventually constant in $n$. The analogous statement is true for the Weyl groups in type B/C and D, and we present a proof here using our results on representation stability for finitely generated $\FIW$--modules -- just as Church--Ellenberg--Farb do in type A \cite[Theorem 2.65]{CEF}. 

The theory of $\FIW$--modules offers a natural conceptual setting for these stability results; in this context, Murnaghan's stability theorem and its variants are immediate consequences of the fact that tensor products of finitely generated $\FIW$--modules are themselves finitely generated.

\begin{thm}\label{MurnaghanWn}{\bf (Murnaghan's stability theorem for $B_n$).} For any pair of double partitions $\y= (\y^+, \y^-)$ and $\mu = (\mu^+, \mu^-)$, there exist nonnegative integers $g^{\nu}_{\y, \mu}$, independent of $n$, such that for all $n$ sufficiently large:
\begin{equation}\label{Eqn:MurnaghanBn} V(\y)_n \otimes V(\mu)_n = \bigoplus_{\nu}  g^{\nu}_{\y, \mu} V(\nu)_n. \end{equation} 
The coefficients $g^{\nu}_{\y, \mu}$ are nonzero for only finitely many double partitions $\nu$.
\end{thm}

\begin{proof}[Proof of Theorem \ref{MurnaghanWn}]
 The $\FI_{BC}$--modules $V(\y)$ and $V(\mu)$ are finitely generated by Proposition \ref{FIModuleV(y)}, and so by Proposition \ref{TensorsPreserveFinGen} their product $V(\y)\otimes V(\mu)$ is finitely generated, and therefore is uniformly representation stable by Theorem \ref{FinGenImpliesRepStability}. 
\end{proof}

We remark that Steven Sam suggested to us an alternate proof of Theorem \ref{MurnaghanWn}, using properties of character polynomials. 

By restricting both sides of Equation (\ref{Eqn:MurnaghanBn}) to action of $D_n$, we conclude:

\begin{cor}\label{MurnaghanDn}{\bf (Murnaghan's stability theorem for $D_n$).} With double partitions $\y= (\y^+, \y^-)$ and $\mu = (\mu^+, \mu^-)$ as above, for all $n$ sufficiently large the tensor product of the $D_n$--representations $V(\y)_n \otimes V(\mu)_n$ has a stable decomposition:
$$ V(\y)_n \otimes V(\mu)_n = \bigoplus_{\nu}  g^{\nu}_{\y, \mu} V(\nu)_n $$ 
where $g^{\nu}_{\y, \mu}$ are the structure constants of Equation (\ref{Eqn:MurnaghanBn}).
\end{cor}

 The observation that Murnaghan's theorem follows from the theory of finitely generated $\FI_A$--modules is given by \cite[Theorem 2.65]{CEF}. Theorem \ref{MurnaghanWn} is a natural counterpart to Murnaghan's theorem, however, we have consulted with a number of experts and have not been able to find the result in the literature.

\subsection{Graded $\FIW$--modules and graded $\FIW$--algebras}

In analogy to Church--Ellenberg--Farb \cite[Section 2.10]{CEF}, we define graded $\FIW$--modules, $\FIW$--algebras, $\FIW$--ideals, and the dual notions for each. We define the finiteness criteria finite type and slope. 

\begin{defn}{\bf (Graded $\FIW$--modules; Finite type; Slope).} A \emph{graded $\FIW$--module} $V=\oplus_i V^i$ is a functor from $\FIW$ to the category of graded $k$--modules. Each graded piece $V^i$ is an $\FIW$--module; we say $V$ has \emph{finite type} if $V^i$ is a finitely generated for all $i$. 

Suppose $k$ is a field of characteristic zero, and let $V$ be a graded $\FIW$--module supported in nonnegative degrees. We say that the \emph{slope} of $V$ is $\leq m$ if $V^i$ has weight $\leq m\cdot i$ for all $i$.  \end{defn}

\begin{example} The polynomial algebras $V_n = k[x_1, \ldots, x_n]$ from Example \ref{Example:PolyAlg} form a graded $\FIW$--module of finite type, graded by total degree. The graded piece $V_n^d := k[x_1, \ldots, x_n]_{(d)}$ is finitely generated in degree $\leq d$. When $k$ is a field of characteristic zero $V$ has slope $\leq 1$ by Theorem \ref{WnDiagramSizes}.  
\end{example}

\noindent The tensor product of graded $\FIW$--modules $U=\oplus_i U^i$ and $W=\oplus_j W^j$ is the graded $\FIW$--module $$U\otimes W = \bigoplus_{\ell} (U \otimes W)^{\ell} := \bigoplus_{\ell} \bigg( \bigoplus_{i+j=\ell} (U^i \otimes W^j) \bigg). $$

By applying Proposition \ref{TensorsPreserveFinGen} to each summand $(U^i \otimes W^j)$, we conclude that the induced grading on the tensor product of graded $\FIW$--modules respects weight and finite generation properties, in the following sense.

\begin{prop} \label{TensorsPreserveFinType} {\bf (Tensor product preserves finite type and slope).} Let $U$ and $W$ be graded $\FIW$--modules of finite type, supported in nonnegative grades, with $U^0 \cong W^0 \cong M_{\W}({\bf 0})$. Then the tensor product $U \otimes W$ is a graded $\FIW$--module of finite type. When $k$ is a characteristic zero field,  $U\otimes W$ will have slope $\leq m$ whenever $U$ and $V$ have slopes $\leq m$. 
\end{prop}

\noindent Church--Ellenberg--Farb prove this result in type A \cite[Proposition 2.70]{CEF}.


\begin{defn}{\bf($\FIW$--algebras).} A \emph{(graded) $\FIW$--algebra} $A = \bigoplus A^i$ is a functor from $\FIW$ to the category of (graded) $k$--algebras. A sub--$\FIW$--module $V$ \emph{generates} $A$ as an $\FIW$--algebra if $V_n$ generates $A_n$ as a $k$--algebra for all $n$.  \end{defn}

\begin{defn} {\bf(Free associative $\FIW$--algebras).} Given a graded $\FIW$--module $V$, we define the \emph{free associative algebra on $V$} as the graded $\FIW$--algebra  $$ k \langle V \rangle := \bigoplus_{j=0}^{\infty} V^{\otimes j}.$$ Any $\FIW$--algebra $A$ generated by $V$ admits a surjection of $\FIW$--algebras $k\langle V \rangle \twoheadrightarrow A.$
\end{defn}

Proposition \ref{TensorsPreserveFinType} implies that $k \langle - \rangle$ respects the weight and finite generation properties of the gradings of a graded $\FIW$--module $V$, and consequently so does any $\FIW$--algebra that $V$ generates. Propositions \ref{k<>PreservesFinType} and \ref{ModulesGenerateFiniteType} are proven in type A by Church--Ellenberg--Farb \cite[Proposition 2.73 and Theorem 2.74]{CEF}.

\begin{prop}{\bf (The functor $k \langle - \rangle$ preserves finite type and slope).} \label{k<>PreservesFinType} Let $V$ be a graded $\FIW$--module supported in nonnegative grades, with $V^0 \cong M_{\W}({\bf 0})$. If $V$ has finite type, then $k \langle V \rangle$ has finite type. If $V$ is a graded $\FIW$--module over characteristic zero with slope $\leq m$, then $k \langle V \rangle$ has slope $\leq m$.
\end{prop}

We can deduce Proposition \ref{k<>PreservesFinType} by applying Proposition \ref{TensorsPreserveFinType} to each summand $V^{\otimes j}$ of the free associative algebra $k \langle V \rangle$. The assumption on the support of $V$ implies that any graded piece of $k \langle V \rangle$ only involves finitely many summands $V^{\otimes j}$.

 If $A$ is an $\FIW$--algebra generated by an $\FIW$--module $V$, we can deduce Propositions \ref{ModulesGenerateFiniteType} and \ref{SlopeFinGenFromGenerators}  from the surjection of graded $\FIW$--algebras $k\langle V \rangle \twoheadrightarrow A$.

\begin{prop}{\bf (Finite type $\FIW$--modules generate $\FIW$--algebras of finite type and slope).} \label{ModulesGenerateFiniteType} Suppose that $A$ is an $\FIW$--algebra generated by a graded $\FIW$--module $V$ of finite type, supported in nonnegative grades. Then if $V$ has finite type, so does $A$. For $k$ a field of characteristic zero, if $V$ has slope $\leq m$ then $A$ has slope $\leq m$. 
\end{prop}

\begin{prop} {\bf } \label{SlopeFinGenFromGenerators}  Let $A$ be an $\FIW$--algebra generated by a graded $\FIW$--module $V$ concentrated in grade $d$. If $V$ is finitely generated in degree $\leq m$, then the $i^{th}$ graded piece $A^i$ is finitely generated in degree $ \leq \left( \frac{i}{d}\right)m$, and moreover if $k$ is a characteristic zero field then weight$ (A^i) \leq \left( \frac{i}{d}\right)$weight$(V)$. 
\end{prop}

\begin{defn}{\bf ($\FIW$--ideals).} Given a graded $\FIW$--algebra $A$, an $\FIW$--ideal $I$ of $A$ is a graded sub--$\FIW$--algebra of $A$ such that $I_n$ is a homogeneous ideal in $A_n$ for each $n$. 
\end{defn}

\begin{defn} {\bf (Co--$\FIW$--modules, Co--$\FIW$--algebras, finite type).} A \emph{graded co--$\FIW$--module} is a functor from the dual category $\FIW^{op}$ to the category of graded $k$--modules, and similarly a \emph{graded co--$\FIW$--algebra} is a functor to the category of graded $k$--algebras. When $k$ is a field, then we say that a graded co--$\FIW$--module $V$ has \emph{finite type} if its dual  $V^*$, defined by $V^*_n = \Hom_k(V_n, k)$, has finite type. Similarly, $V$ has \emph{slope} $\leq m$ if $V^*$ does. 
\end{defn}

\begin{prop}{\bf (Finite type co--$\FIW$--modules generate co--$\FIW$--algebras of finite type).} \label{CoModulesGenerateFiniteType}
 Let $k$ be a Noetherian commutative ring. Suppose that $A$ is a graded co--$\FIW$--algebra containing a graded co--$\FIW$--module $V$ supported in positive grades. If $V$ has finite type, then the subalgebra $B$ of $A$ generated by $V$ is a graded co--$\FIW$--algebra of finite type.  When $k$ is a field of characteristic zero and $V$ is a graded co--$\FIW$--module  of slope $\leq m$, then $B$ has slope $\leq m$.
\end{prop}

\begin{proof}[Proof of Proposition \ref{CoModulesGenerateFiniteType}]
The proposition follows just as in the proof of \cite[Proposition 2.77]{CEF}, by considering the dual space $B^*$ as a graded sub--$\FIW$--algebra of $k\langle V \rangle^*.$ Theorem \ref{Noetherian}, the Noetherian property for $\FIW$--modules over Noetherian rings, implies that each graded piece of $B^*$ is finitely generated. Moreover, over a characteristic zero field, the weights of the graded pieces of $k\langle V \rangle^*$ give an upper bound of on the weights of those of $B^*$, and the result follows. 
\end{proof}


\section{An application: The cohomology of generalized flag varieties and diagonal coinvariant algebras} \label{SectionCoinvariantAlgebras}

 Let $\W_n$ be a finite reflection group acting on an $n$--dimensional vector space $V$ over a field $k$. Let $x_1, x_2, \ldots, x_n$ denote a basis for $V$. Then $$k[\bX^{(r)}(n)] := k[x_1^{(1)}, \ldots x_n^{(1)}, \ldots, x_1^{(r)}, \ldots, x_n^{(r)}]$$ is a polynomial ring isomorphic to the symmetric algebra on $V^{\bigoplus r}$; the algebra $k[\bX^{(r)}(n)]$ has an action of $\W_n$ induced by the diagonal action of $\W_n$ on $V^{\bigoplus r}$. This ring has a natural grading by $r$--tuples $J=(j_1, \ldots, j_r) \in \Z_{\geq 0}^r,$ where $j_i$ designates the total degree in variables $x_1^{(i)}, \ldots x_n^{(i)}$. 

Let $\cI_n$ be the ideal generated by the constant-term-free $\W_n$--invariant polynomials. The \emph{$r$-diagonal coinvariant algebra} is the $k$--algebra $\cC^{(r)}(n) := k[\bX^{(r)}(n)] / \cI_n.$ 

Since $\cI_n$ is homogeneous with respect to the multigrading on $k[\bX^{(r)}(n)]$ , the quotient has the same multigrading
 $$\cC^{(r)}(n) = \bigoplus_{d=0}^{\infty} \bigoplus_{|J|=d} \cC^{(r)}_J(n).$$

\noindent The structure of $\cC^{(r)}(n)$ as a $\W_n$--representation over characteristic zero has been the subject of extensive study. The coinvariant algebra $\cC^{(1)}(n)$ appeared in classical representation theory and Lie theory;  Borel \cite{Borel1953} proved that the algebra $\cC^{(1)}(n)$ is the cohomology of a generalized flag manifold, which we will define below. The diagonal coinvariant algebras $\cC^{(2)}(n)$ were first investigated in type A by Garsia and Haiman \cite{GarsiaHaiman93} for their relationship to \emph{MacDonald polynomials}, but these algebras were subsequently found to have rich connections to numerous objects in algebraic combinatorics; see Haiman \cite{HaimanCombinatorics} for a survey. 

In 2002 Haiman established a formula for the characters of the $S_n$--representations $\cC^{(2)}(n)$ in terms of MacDonald polynomials, and deduces a number of combinatorial consequences for the spaces $\cC^{(2)}(n)$ \cite{HaimanVanishing}. A refinement of the formulas for these characters was conjectured by Haglund--Haiman--Loehr--Remmel--Ulyanov \cite{HaglundHaimanLoehrRemmelUlyanov}.

In 2003 Gordon \cite{GordonDiagonal} studied coinvariant algebras associated to a Coxeter group $W_n$. He resolved a conjecture of Haiman \cite{HaimanConjectures} by computing the Hilbert series of a quotient ring closely related to $\cC^{(2)}(n)$. Bergeron and Biagioli computed the trivial and alternating component of $\cC^{(2)}(n)$ in type B \cite{BergeronBiagioli}. In 2011, Bergeron analyzed the algebras $\cC^{(r)}(n)$ associated to a general complex reflection group $W=G(m,p,n)$ \cite{BergeronComplexReflection}. Bergeron shows, for fixed group $W$, the multigraded Hilbert polynomial associated to $\cC^{(r)}(n)$ can be described in terms of Schur polynomials in a form independent of $r$, and Bergeron computes these series in special cases. In general, the structure (or even dimension) of $\cC^{(r)}(n)$ is not known for $n>3$. Additional background on coinvariant algebras can be found in Bergeron's book \cite{BergeronBook}. 

Church--Ellenberg--Farb \cite[Theorem 3.4]{CEF} proved that when $\W_n$ is $S_n$ acting on the representation $V_n = M_{A}({\bf 1})_n$ over a field $k$ of characteristic zero, the resultant coinvariant algebra $\cC^{(r)} := k[M_{A}({\bf 1})^{\bigoplus r}]/\cI$ is a graded co--$\FI_{A}$--module of finite type, and that moreover the graded pieces $(\cC_J^{(r)})^*$ of the dual $\FI_A$--module have weight at most $|J|$. Together with Nagpal, these authors showed that even over positive characteristic, the dimensions of the graded pieces are eventually polynomial \cite[Theorem 1.9]{CEFN}. We can extend their results as follows.

\begin{thm}\label{CoinvariantAlgebrasCoFI}{\bf (Diagonal coinvariant algebras are finite type).} Let $k$ be a field, and let $V_n \cong k^n$ be the canonical representation of $\W_n$ by (signed) permutation matrices. Given $r \in \Z_{>0}$, the sequence of coinvariant algebras $$\cC^{(r)} := k[V_{\bullet}^{\bigoplus r}]/\cI$$ is a graded co--$\FIW$--algebra of finite type. When $k$ has characteristic zero, the weight of the multigraded component $\cC^{(r)}_J$ is $\leq |J|$.
\end{thm}

\begin{proof}[Proof of Theorem \ref{CoinvariantAlgebrasCoFI}] Let $\W_n$ be the Weyl group $S_n$, $D_n$, or $B_n$, then let $V$ be the $\FIW$--module associated to the canonical $n$--dimensional $\W_n$--representations $V_n \cong k^n$ by (signed) permutation matrices. Then $V$ is $M_A({\bf 1})$ in type A, $M_{BC}(\varnothing, \Y{1})$ in type B/C, and $\Res_{D}^{BC} M_{BC}(\varnothing, \Y{1})$ in type D. 

In each type, the sequence $\cC^{(r)}$ has a co--$\FIW$--module structure, determined as follows. The polynomial rings $$k[V_{n}^{\bigoplus r}] \cong k[x_1^{(1)}, \ldots x_n^{(1)}, \ldots, x_1^{(r)}, \ldots, x_n^{(r)}]$$ admit a co--$\FIW$--module structure defined by the natural $\W_n$--action and the projection maps 

\begin{align*}
(I_n)^*: k[V_{n+1}^{\bigoplus r}]  & \longrightarrow k[V_{n}^{\bigoplus r}]  \\
 x_j^{(i)} & \longmapsto \left\{ \begin{array}{ll}
         x_j^{(i)} & \mbox{$ i \leq n  $},\\
        0 & \mbox{$ i = n+1$}.\end{array} \right.
\end{align*}

\noindent The ideals $\cI_n$ form a co--$\FIW$--submodule under this action (even though they are not preserved by the natural $\FIW$--module structure on the spaces $k[V_{\bullet}^{\bigoplus r}]$), and hence the quotient space $\cC^{(r)}$ inherits a co--$\FIW$--module structure (though not an $\FIW$--module structure) for any fixed $r$.

Since $\cC^{(r)}(n)^*$ is generated as an algebra by its degree $1$ part, the co--$\FIW$--algebra $\cC^{(r)}$ has finite type by Proposition \ref{CoModulesGenerateFiniteType}.  

Over characteristic zero, the $\FIW$--module $(V_{\bullet}^{\bigoplus r})^*$ has weight $1$ by Theorem \ref{WnDiagramSizes}. The graded piece $ \cC^{(r)}_J(n)$ is a subquotient of the degree $|J|$ tensor product on $(V_n)^{\bigoplus r}$, and weight is additive under tensor products by Proposition \ref{TensorsPreserveFinGen}.  
\end{proof}


\begin{cor} \label{CoinvariantAlgRepStable}
 Let $k$ be a field of characteristic zero. For $n$ sufficiently large (depending on the $r$-tuple $J$), the sequence $\cC^{(r)}_J(n)$ is uniformly multiplicity stable. 
\end{cor}

Since representations of $\W_n$ are self-dual (a consequence of \cite[Corollary 3.2.14]{GeckPfeiffer}), the characters of $\cC_J^{(r)}(n)$ are given by the those of its dual. By results of \cite{FIW2}, these characters are given by a character polynomial, with degree bounded by \cite[Theorems \ref{FIW2-SnPersinomial} and \ref{FIW2-WnPersinomial}]{FIW2}. 

\begin{cor} \label{CoinvariantAlgCharPoly}
Let $k$ be a field of characteristic zero. For $n$ sufficiently large (depending on the $r$-tuple $J$), the characters of $\cC^{(r)}_J(n)$ are given by a character polynomial $F_J$ of degree $\leq |J|$. In particular the dimension of $\cC^{(r)}_J(n)$ is given by a polynomial $\dim_k \cC^{(r)}_J(n) = F_J(n, 0, 0, 0 \ldots )$ for all $n$ in the stable range. 
\end{cor}

The results in \cite[Theorem \ref{FIW2-PolyDimCharP}]{FIW2} further imply that over fields of any characteristic, the dimensions of the graded pieces of $\cC^{(r)}$ are eventually polynomial. This result was proved in type A by Church--Ellenberg--Farb--Nagpal \cite[Theorem 1.9]{CEFN}.
\begin{cor} \label{CoinvariantAlgPolyDim}
Let $k$ be an arbitrary field. Then for each $r$-tuple $J$, there exists a polynomial $P_J \in \Q[T]$ (depending on $k$) so that
$\dim_k \; \cC^{(r)}_J(n) = P_J(n)$
for all $n$ sufficiently large (depending on $k$ and $J$).
\end{cor}

{\noindent \bf The cohomology of generalized flag manifolds. \quad} Take $k$ to be the complex numbers $\C$. Let ${\G^{\W}_n}$ be a semisimple complex Lie group with Weyl group $\W_n$, and let $\B^{\W}_n$ be a Borel subgroup of $\G^{\W}_n$.  Borel proved that the complex coinvariant algebra $\cC^{(1)}(n)$ is isomorphic as a graded $k[\W_n]$--algebra to the cohomology $H^*(\G^{\W}_n/\B^{\W}_n; \C)$ of the \emph{generalized flag manifold} $\G^{\W}_n/\B^{\W}_n$ \cite{Borel1953}; the isomorphism multiplies the grading by $2$. Specifically, we have\\

\begin{tabular}{l r@{} @{}l } 
 \text{Type A$_{n-1}$}:& $\G_n^A$ &= $\SL_n(\C)$ \\ 
\multicolumn{3}{l}{  $\G_n^A / \B^A_n = \{ 0 \subseteq V_1 \subseteq V_2 \subseteq \cdots \subseteq V_n = \C^n \; \vert \; \dim_{\C} V_m = m \} $} \\ 
\multicolumn{3}{l}{  The complete flag variety } 
\end{tabular}

\begin{tabular}{l r@{} @{}l }
 $\text{Type B$_{n}$}:$  & $\G_n^B $&$= \SO_{2n+1}(\C) $ \quad(Quadratic form $Q$)\\ 
\multicolumn{3}{l}{  $\G_n^B / \B^B_n = \{ 0 \subseteq V_1  \subseteq \cdots \subseteq V_{2n+1} = \C^{2n+1} \;  \vert \; \dim_{\C} V_m = m , \;\; Q(V_i,V_{2n+1-i})=0 \}$} \\ 
\multicolumn{3}{l}{  The variety of complete flags equal to their orthogonal complements } 
\end{tabular}

\begin{tabular}{l r@{} @{}l }
 $\text{Type C$_{n}$}:$ & $\G_n^C $&$= \Sp_{2n}(\C) $ \quad(Symplectic form $L$) \\
\multicolumn{3}{l}{  $\G_n^C / \B^C_n = \{ 0 \subseteq V_1  \subseteq \cdots \subseteq V_{2n} = \C^{2n} \;  \vert \; \dim_{\C} V_m = m , \;\; L(V_i,V_{2n-i})=0 \}$} \\ 
\multicolumn{3}{l}{  The variety of complete flags equal to their symplectic complements } 
\end{tabular}

\begin{tabular}{l r@{} @{}l }
 $\text{Type D$_{n}$}:$ & $\G_n^D $&$= \SO_{2n}(\C) $ \quad(Quadratic form $Q$)\\
\multicolumn{3}{l}{  $\G_n^D / \B^D_n = \{ 0 \subseteq V_1  \subseteq \cdots \subseteq V_{2n} = \C^{2n} \;  \vert \; \dim_{\C} V_m = m , \;\; Q(V_i,V_{2n-i})=0 \}$} \\ 
\multicolumn{3}{l}{  The variety of complete flags equal to their orthogonal complements } \\
\end{tabular}
\\

See (for example) Fulton--Harris \cite{FultonHarris} for more details. Theorem \ref{CoinvariantAlgebrasCoFI} therefore implies: 

\begin{cor} Let $\W$ denote type A, B, C, or D. 
 The cohomology rings $H^*(\G^{\W}_n/\B^{\W}_n; \C)$ are graded co--$\FI_{\W}$--algebras of finite type, that is, for each $m$, $H^m(\G^{\W}_n/\B^{\W}_n; \C)$ are co--$\FIW$--modules of weight $\leq \frac{m}{2}$. In particular, for each $m$, the sequence of $\W_n$--representations $\{ H^m(\G^{\W}_n/\B^{\W}_n; \C)\}_n$ is uniformly representation stable, and the characters are eventually equal to a character polynomial of degree at most $\frac{m}{2}$.
\end{cor}

We can compute the character polynomials for the $r$-diagonal coinvariant algebras $\cC^{(r)}$ for small values of $r$ by hand, by computing the trace of the action of $\W_n$ at each point in a resolution for $\cC^{(r)}(n)$ by $k[\W_n]$--modules. When $\W_n$ is $B_n$, we find the following characters $\chi_J^{(r)}(n)$ of $\cC_J^{(r)}(n)$.  

\begin{align*} 
\chi^{(1)}_{(1)}  &= X_1 - Y_1 \qquad & & (n \geq 1) 
\end{align*}
\begin{align*} 
\chi^{(1)}_{(2)}  &= X_1 + Y_1 + { X_1 \choose 2 } + { Y_1 \choose 2 } + X_2 - Y_2 - X_1Y_1 - 1  \qquad & (n \geq 2) 
\end{align*}
\begin{align*} 
\chi^{(1)}_{(3)}  &= 2 { X_1 \choose 2 } - 2 {Y_1 \choose 2 } + {X_1 \choose 3 } + X_1 { Y_1 \choose 2 } - Y_1 { X_1 \choose 2 } - {Y_1 \choose 3 } \\& +X_3 -Y_3 +X_1X_2 - Y_1X_2 - X_1Y_2 +Y_1Y_2 \qquad  & (n \geq 3)  
\end{align*}
\begin{align*} 
\chi^{(2)}_{(1,1)} &= X_1 + Y_1 + 2 { X_1 \choose 2 } + 2 { Y_1 \choose 2 } - 2 X_1Y_1 - 1 \qquad & (n \geq 2) 
\end{align*}
\begin{align*} 
\chi^{(2)}_{(2,1)} &= Y_1 - X_1 + 4 {X_1 \choose 2} - 4 { Y_1 \choose 2 } + X_2X_1 -X_2Y_1 - X_1Y_2 + Y_1Y_2 \\ & + 3{X_1 \choose 3 } - 3 {Y_1 \choose 3} +3X_1{Y_1 \choose 2} -3 Y_1{X_1 \choose 2} & ( n \geq 3) \end{align*}
\begin{align*} 
\chi^{(3)}_{(1,1,1)} &= -2X_1 + 2Y_1 + 6{X_1 \choose 2} - 6 { Y_1 \choose 2} + 6{X_1 \choose 3 } - 6{Y_1 \choose 3} + 6X_1 {Y_1 \choose 2} - 6Y_1 {X_1 \choose 2} & ( n \geq 3)
 \end{align*} 

We note that the character of $\chi^{(r)}_{(j_1, \ldots, j_r)} = \chi^{(r+1)}_{(j_1, \ldots, j_r, 0)}$, and moreover the characters $\chi^{(r)}_{(j_1, \ldots, j_r)}$ are fixed under permutations of the ordered $r$--tuple $J$. It follows that the above character polynomials determine all characters $\chi^{(r)}_{J}$ for $|J| \leq 3$. 
 
\begin{problem} For each graded piece $\cC^{(r)}_J$, compute the associated character polynomial and the stable decomposition into irreducible representations. Determine the stable ranges of each.   \end{problem}

{\footnotesize 


\bibliographystyle{alpha}
\bibliography{MasterBibliography} 
\setlength{\itemsep}{1pt} 

 \quad \\ 
\noindent \textit{Jennifer C. H. Wilson \\ University of Chicago \\ wilsonj@math.uchicago.edu} \\
}
\end{document}